\documentclass[12pt]{amsart}

\usepackage{amssymb,amscd,amsthm, verbatim,amsmath,color,fancyhdr, mathrsfs}
\usepackage{graphicx}
\usepackage{turnstile}
\usepackage{multicol}
\usepackage{enumerate}
\usepackage{tikz-cd}
\usepackage{dsfont}
\usepackage{bbm}

\usepackage[english]{babel}
\usepackage[T1]{fontenc}
\usepackage{hyperref}

\usepackage[letterpaper, left=2.5cm, right=2.5cm, top=2.5cm,
bottom=2.5cm,dvips]{geometry}

\theoremstyle{definition}
\newtheorem{thm}{Theorem}[section]
\newtheorem{prop}[thm]{Proposition}
\newtheorem{obs}[thm]{Observation}
\newtheorem{lem}[thm]{Lemma}
\newtheorem{cor}[thm]{Corollary}

\newtheorem{defi}[thm]{Definition}

\newtheorem{conj}[thm]{Conjecture}

\newtheorem{hyp}[thm]{Hypothesis}

\newcommand{\bigslant}[2]{{\raisebox{.2em}{$#1$}\left/\raisebox{-.2em}{$#2$}\right.}}

\title{Classification of unstable circulants of square-free order}
\author{Bartłomiej Bychawski}
\date{}

\begin{document}

\maketitle

\begin{abstract}
    In this paper we prove that for circulants of squarefree orders Wilson's conjecture hold, that is each nontrivially unstable circulant of such order has Wilson type. We show that actually only criteria (C.1) and (C.4) are needed.
\end{abstract}

\section{Introduction} \label{SECTION: Introduction}

    
    
    In this paper we will study algebraic properties of graph. We define a graph to be a pair $(V,E)$ where $V$ is a finite set called \textit{vertex set} and $E$ is the symmetric subset of $X \times X$ such that for any $x \in X$, $(x,x) \notin E$. $E$ is called the \textit{edge set} and is sometimes identified with a binary relation on $V$. A \textit{neighbourhood of a vertex} $v \in V$ is defined as
    $$
    N(v) = \bigl\{ w \in V \mid (v,w) \in E \bigr\}.
    $$
    Graph $\Gamma = (V,E)$ is called \textit{connected} when between each pair of vertices there exists a path made of edges between them. It is called \textit{bipartite} when the set of vertices can be partitioned int two disjoint sets $V_1$ and $V_2$ such that there are no edges between vertices from the same subset $V_i$ for $i \in \{1,2\}$. A graph is called \textit{reduced} when for each pair of distinct verticies $v \neq w \in V$ it follows that $N(v) \neq N(w)$.
    
    Tensor product of graphs $\Gamma = (V,E)$ and $\Sigma = (W,F)$ is the graph $\Gamma \times \Sigma$ with vertex set $V \times W$ and edge set defined by
    $$
    \Bigl\{ \bigl( (v_1,w_1)(v_2,w_2) \bigr) \in (V \times W) \times (V \times W) \mid (v_1,v_2) \in E \text{ and } (w_1,w_2)\in F \Bigr\}.
    $$
    It is easy to see that $\text{Aut}(\Gamma) \times \text{Aut}(\Sigma) \leq \text{Aut}(\Gamma) \times \Sigma)$, however these groups does not have to be equal. In case when both $\Gamma$ and $\Sigma$ are connected, non-bipartite and reduced, full description of the group $\text{Aut}(\Gamma) \times \Sigma)$ was given by Dörfler \cite[Theorem 8.18]{HandOfProdGraphs}. Is is known, that this result cannot be extended to the case when at least one of them is bipartite. A lot of such complications can be explained by the fact that $\text{Aut}(\Gamma) \times \text{Aut}(K_2) \neq \text{Aut}(\Gamma \times K_2)$, where $K_2$ is the graph with vertex set $\{0,1\}$ and edge set $\{(0,1),(1,0)\}$. Graph $K_2$ is therefore just made of two vertices connected by an edge. For above reason we call a graph $\Gamma$ \textit{stable} if $\text{Aut}(\Gamma) \times \text{Aut}(K_2) = \text{Aut}(\Gamma \times K_2)$ and \textit{unstable} otherwise.

    It was early noticed, that if a graph $\Gamma= (V,E)$ is disconnected, bipartite and $\text{Aut}(\Gamma) \neq \{ id_V\}$ or is not reduced, then $\Gamma$ is unstable. Graphs $\Gamma$ which satisfy some of the above are called \textit{trivially unstable}. If $\Gamma$ is connected, non-bipartite, reduced and unstable, we call it \textit{non-trivially unstable} \cite{WILSON2008359}.

    A permutation $\tau \in \text{Aut}(\Gamma \times K_2)$ is called \textit{an unexpected symmetry of $\Gamma \times K_2$} if $\tau \notin \text{Aut}(\Gamma) \times S_2$. If $\Gamma$ is connected and non-bipartite, each symmetry of $\Gamma \times K_2$ either switches or fixes set-wise subsets $V \times \{0\}$ and $V \times \{1\}$. For that reason concept of \textit{two-fold automorphisms} was introduced. A pair $(\sigma_1,\sigma_2) \in \text{Sym}(V) \times \text{Sym}(V)$ is called a two-fold automorphism if permutation 
    $\sigma: V \times \{0,1\} \rightarrow V \times \{0,1\}$ given by the formula $(v,i) \mapsto (\sigma_{i+1}(v),i)$ is an automorphism of $\Gamma \times K_2$. The group of all two-fold automorphisms of a graph $\Gamma$ with coordinate-wise composition of functions is denoted $\text{Aut}^{\text{TF}}(\Gamma)$. Unexpected symmetries of $\Gamma \times K_2$ correspond to two-fold automorphisms of $\Gamma$ such that $\sigma_1 \neq \sigma_2$.

    For general graphs there is no simple classification of unstable graphs. For that reason most of the research focused on well structured families of graphs, such as \textit{Cayley graphs}. For a group $G$ and its subset $S \subseteq G$ such that $S^{-1} = S$ and $e_G \notin S$, by $\text{Cay}(G,S)$ we denote the graph with vertex set $G$ and edge set defined by $\bigl\{ (g,h) \in G \times G \mid h g^{-1} \in S \bigr\}$.

    For any $n\geq 1$ by $\mathbb{Z}_n$ we denote the quotient of $\mathbb{Z}$ by $n\mathbb{Z}$ with the action of addition. Cayley graphs over groups $\mathbb{Z}_n$ are called \textit{circulants}, since automorphism group of these graphs contains a cyclic subgroup which acts regularly on vertices.

    Wilson \cite[Apendix A]{WILSON2008359} conjectured that any non-trivially unstable circulant satisfies one conditions (C.1)-(C.4) he listed. Later it was found that conditions (C.2) and (C.3) contained a flaw, and were replaced by (C.2)' and (C.3)'. Repaired list can be found in \cite[Theorem 1.4]{HujdurovicMitrovicMorrisOverview}. In the same paper Hujdurović,  Mitrović and Morris found circulants which are non-trivially and does not satisfy any of (C.1), (C.2)', (C.3)' or (C.4) \cite[Example 3.9]{HujdurovicMitrovicMorrisOverview}, hence conjecture of Wilson fails for arbitrary $n$.

    It \cite{FERNANDEZ202249} Fernandez and Hujdurović proved that there are no non-trivially unstable circulants of odd order, which was later generalized by Morris in \cite{MorrisOddAbelianGroups} for Cayley graphs over arbitrary abelian group of odd order. Lately in \cite{StabilityAndSchurRings} Hujdurović and Kovács with use of Schur ring theory classified all unstable circulants of order $n=2p^e$, where $p$ is an odd prime and $e \geq 1$. In this paper we extend their methods to derive classification of unstable circulants of square-free order. \\
    
    In Section \ref{SECTION: Schur Rings and function alpha} we transform the problem into one involving the group $\text{Aut}^{\pi}(\Gamma)$ (cf. Definition \ref{Defi: TF-projections}) and function $\alpha$ (cf. Definition \ref{Defi: function alpha}) which superficially speaking measures the inexpediency of a given two-fold automorphism.

    Sections \ref{SECTION: Chain automorphisms} and \ref{SECTION: Replacement Property} are devoted to prove
    
    \begin{thm} \label{Thm: if n is squarefree, all pairs (Gamma,H) satisfy replacement property}
        Let $n$ be any integer, $\Gamma = \text{Cay}(\mathbb{Z}_n,S)$ be any Cayley graph and $H \leq \mathbb{Z}_n$ be such that $|H|$ is coprime to $\frac{n}{|H|}$. Then pair $(\Gamma,H)$ satisfy \textit{replacement property}.
    \end{thm}
    
    Replacement property is a way to regularize the action of a given symmetry of $\Gamma \times K_2$. Section \ref{SECTION: Chain automorphisms} develops a language of \textit{chain graphs} and \textit{chain automorphisms} which studies properties of a particular infinite digraph associated to a graph. This machinery is the main tool in the proof of Theorem \ref{Thm: if n is squarefree, all pairs (Gamma,H) satisfy replacement property}.

    Later in the proof of the main theorem we analyze certain primitive group actions. For that reason in Section \ref{SECTION: Group theoretical results} we prove two important group theoretic and cohomological results. Before we state them we have to give a couple definitions. Let $X$ be a finite set and $\varphi \in \text{Sym}(X)$. Then by $\iota(\varphi)$ we understand an element of $\text{Aut}(\text{Sym}(X))$ given by $\psi \mapsto \varphi \circ \psi \circ \varphi^{-1}$. Permutation $\mathfrak{i}: \mathbb{Z}_k \rightarrow \mathbb{Z}_k$ is given by the formula $x \mapsto -x$. Now we are ready to state

    \begin{thm} \label{Thm characterization of primitive group actions with regular cyclic subgroup and additional assumptions}
         Let standard action of $G \leq \text{Sym}(X)$ on $X = \mathbb{Z}_k$ be primitive. If moreover $k$ is odd, ${\left( \mathbb{Z}_k \right)}_r \leq G$ and ${\iota}(\mathfrak{i}) \in \text{Aut}(G)$, then up to an isomorphism of group actions one of the following holds:
     \begin{enumerate}[i.]
         \item $\mathbb{F}_p \leq G \leq \text{Aff}(\mathbb{F}_p)$, $X = \mathbb{F}_p$ where $k=p$ is an odd prime;
         \item $A_{k} \leq G \leq S_{k}$ with $k \geq 5$ and standard action of permutation groups on elements;
         \item $\text{PGL}_{2}(\mathbb{F}_{2^\ell}) \leq G \leq \text{P}\Gamma\text{L}_{2}(\mathbb{F}_{2^\ell})$, $X = \mathbb{P}^1\mathbb{F}_{2^\ell}$ for some positive $\ell \geq 2$ with standard action of projective group on lines.
     \end{enumerate}
    \end{thm}

    \begin{thm} \label{Thm: action of a nontrivial cocycle from group G acting primitively on X with additional assumptions}
        Let standard action of $G \leq \text{Sym}(X)$ on $X = \mathbb{Z}_k$ be primitive. Moreover let $k$ be an odd integer, ${\left( \mathbb{Z}_k \right)}_r \leq G$ and ${\iota}(\mathfrak{i}) \in \text{Aut}(G)$. If $\omega : G \rightarrow \mathbb{F}_2[X]$ is a nonzero cocycle such that $\omega {|}_{{\left( \mathbb{Z}_n \right)}_r} \equiv \vec{0}$, then
        $$
        \omega(g) = \begin{cases} 
      \vec{0} & \text{ if } g \in G_0 \\
     \sum_{x\in X} \vec{e}_x & \text{ otherwise }
   \end{cases},
        $$
        where $G_0$ is the unique subgroup of $G$ of index $2$.
    \end{thm}

    In Section \ref{SECTION: Main results} we combine conclusions of Theorem \ref{Thm: if n is squarefree, all pairs (Gamma,H) satisfy replacement property} and Theorem \ref{Thm: action of a nontrivial cocycle from group G acting primitively on X with additional assumptions} to prove the main result of this paper.
    \begin{thm} \label{THM: characterization of unstable circulants of squarefree order}
       Let $n$ be an even square-free integer and let $\Gamma = \text{Cay}(\mathbb{Z}_{n},S)$ be a connected and nonbipartite graph. Then $\Gamma$ is unstable if and only if
        \begin{enumerate}[i.]
            \item there exists nonzero $h \in 2\mathbb{Z}_{n}$ such that $S \cap 2\mathbb{Z}_{n} + h = S \cap 2\mathbb{Z}_{n}$;
            \item or there exists positive integer $l$ coprime to $n$ such that $lS = S + \frac{n}{2}$.
        \end{enumerate}
    \end{thm}

    This result can be restated in the context of Wilson's Conjecture.

    \begin{cor} \label{Cor: non-trivially unstable of circulants of square-free order have Wilson type (C.1) or (C.4)}
       Let $n$ be any square-free integer and let $\Gamma = \text{Cay}(\mathbb{Z}_{n},S)$ be non-trivially unstable Cayley graph. Then $\Gamma$ has Wilson type (C.1) or (C.4).
    \end{cor}
    
    Theorem \ref{THM: characterization of unstable circulants of squarefree order} and \cite[Theorem 1.5]{StabilityAndSchurRings} suggests that for arbitrary odd $m>1$ following holds.

    \begin{conj} \label{Conj: characterization of unstable circulants of order n=2m, m>1 and odd}
        Let $n=2m$ where $m>1$ is an odd integer and let $\Gamma = \text{Cay}(\mathbb{Z}_{n},S)$ be a connected and nonbipartite graph. Then $\Gamma$ is unstable if and only if 
        \begin{enumerate}[i.]
            \item there exists nonzero $h \in 2\mathbb{Z}_{2m}$ such that $S \cap 2\mathbb{Z}_{2m} + h = S \cap 2\mathbb{Z}_{2m}$;
            \item or $\text{Cay}(\mathbb{Z}_{2m},S) \cong \text{Cay}(\mathbb{Z}_{2m},S+m)$.
        \end{enumerate}
    \end{conj}


    
\section{Schur Rings and function $\alpha$} \label{SECTION: Schur Rings and function alpha}

\subsection{Two fold projections and function $\alpha$} \label{SUBSECTION: Two fold projections and function alpha} \hfill \\
For any set $X$ we denote the full permutation group of $X$ with $\text{Sym}(X)$.

\begin{defi} \label{Defi: TF-projections}
    For a  graph $\Gamma=(V,E)$ we define the group of \textit{two-fold projections} by
    $$
    \text{Aut}^{\pi}(\Gamma) = \{ \sigma_1 \in \text{Sym}(V) \mid \exists  \sigma_2 \in \text{Sym}(V) \text{ such that } (\sigma_1,\sigma_2) \in \text{Aut}^{\text{TF}}(\Gamma) \}.
    $$
\end{defi}

\begin{obs} \label{Obs: Aut^pi = Aut^TF for reduced graphs}
    Let $\Gamma$ be a reduced graph. Then $\pi_1: \text{Aut}^{\text{TF}}(\Gamma) \rightarrow \text{Aut}^{\pi}(\Gamma)$ given by $(\sigma_1,\sigma_2) \mapsto \sigma_1$ is an isomorphism.
\end{obs}
\begin{proof}
    Function $\pi_1$ is obviously a homomorphism. It is subjective by definition of $\text{Aut}^{\pi}(\Gamma)$, hence we only have to verify that its kernel is trivial. Let $(\sigma_1,\sigma_2) \in \text{ker} \pi_1$, so $\sigma_1 = id$. If $\sigma_2 \neq id$ then there exist such $v \in V$ that $\sigma_2(v) \neq v$. Since $(\sigma_1,\sigma_2)$ is a two-fold automorphism of $\Gamma$, vertices $(v,1)$ and $(\sigma_2(v),1)$ have the same neighbourhood in $\Gamma \times K_2$, and hence $N(v) = N(\sigma_2(v))$ contradicting that $\Gamma$ is reduced.
\end{proof}
    Observation \ref{Obs: Aut^pi = Aut^TF for reduced graphs} shows that for non-trivially unstable graphs difference between $\text{Aut}^{\pi}(\Gamma)$ and $\text{Aut}^{\text{TF}}(\Gamma)$ is strictly formal in group theoretic terms, however group $\text{Aut}^{\pi}(\Gamma)$ acts on the set $V$ not on $V \times \{0,1\}$ which will come in handy in the future.

    \begin{defi} \label{Defi: widetilde{gamma} and gamma for undirected reduced graphs}
        For a given graph $\Gamma$ let $\widetilde{\gamma}$ be an automorphism of $\text{Aut}^{\text{TF}}(\Gamma)$ given by $(\sigma_1,\sigma_2) \mapsto (\sigma_2,\sigma_1)$. Moreover if $\Gamma$ is reduced we define $\gamma: \text{Aut}^{\pi}(\Gamma) \mapsto \text{Aut}^{\pi}(\Gamma)$ to be the unique function making bellow diagram commute.
\[\begin{tikzcd}
	{\text{Aut}^{\text{TF}}(\Gamma)} & {\text{Aut}^{\text{TF}}(\Gamma)} \\
	{\text{Aut}^{\pi}(\Gamma)} & {\text{Aut}^{\pi}(\Gamma)}
	\arrow["{\widetilde{\gamma}}", from=1-1, to=1-2]
	\arrow["{\pi_1}", from=1-1, to=2-1]
	\arrow["{\pi_1}", from=1-2, to=2-2]
	\arrow["{\exists! \text{ }\gamma}", dotted, from=2-1, to=2-2]
\end{tikzcd}\]
If it is not clear to which graph $\widetilde{\gamma}$ or $\gamma$ refers to, we write ${\widetilde{\gamma}}_{\Gamma}$ and $\gamma_{\Gamma}$ respectively when it refers to $\Gamma$.
    \end{defi}

For a given group $G$ and automorphism $\varphi: G \rightarrow G$, $G^{\varphi}$ is the set of fixed points of $\varphi$.

    \begin{obs} \label{Obs: basic properties of gamma}
        Let $\Gamma$ be a reduced graph. Then function $\gamma$ satisfies $\gamma \in \text{Aut}(\text{Aut}^{\pi}(\Gamma))$, $\gamma^2 = id$ and ${ \left( \text{Aut}^{\pi}(\Gamma) \right) }^{\gamma} = \text{Aut}(\Gamma)$.
    \end{obs}
    \begin{proof}
        At first let us notice that $\gamma = \pi_1 \circ \widetilde{\gamma} \circ \pi_1^{-1}$ and all of the components are isomorphisms, hence $\gamma$ also is. Now let $(\sigma_1,\sigma_2) \in \text{Aut}^{\text{TF}}(\Gamma)$. Then also $(\sigma_2,\sigma_1) \in \text{Aut}^{\text{TF}}(\Gamma)$ and $\gamma(\sigma_1) = \sigma_2$, $\gamma(\sigma_2) = \sigma_1$ so indeed $\gamma^2(\sigma_1) = \sigma_1$ for arbitraty $\sigma_1 \in \text{Aut}^{\pi}(\Gamma)$. \\
        The fact that $\text{Aut}(\Gamma) \subseteq { \left( \text{Aut}^{\pi}(\Gamma) \right) }^{\gamma} $ is obvious, hence to prove the last part of this observation we have to show ${ \left( \text{Aut}^{\pi}(\Gamma) \right) }^{\gamma} \subseteq \text{Aut}(\Gamma)$. Take $\sigma_1 \in { \left( \text{Aut}^{\pi}(\Gamma) \right) }^{\gamma}$. Then $(\sigma_1,\sigma_1) = (\sigma_1,\gamma(\sigma_1)) \in \text{Aut}^{\text{TF}}(\Gamma)$. Let $\Gamma = (V,E)$. Consider a pair of vertices $v,w \in V$ such that $(v,w) \in E$. Then pair $((v,0),(v,1))$ forms an edge in $\Gamma \times K_2$ so $((\sigma_1(v),0),(\sigma_1(v),1))$ also forms an edge in $\Gamma \times K_2$. By the definition of $\Gamma \times K_2$ this means that $(\sigma_1(v),\sigma_1(w)) \in E$ and since vertices $v$ and $w$ were arbitrary we finally obtain $\sigma_1 \in \text{Aut}(\Gamma)$.
    \end{proof}

Now we can introduce the function $\alpha$ which measures ,,how unstable'' is given element of the group $\text{Aut}^{\pi}(\Gamma)$ which encodes symmetries of $\Gamma \times K_2$.

\begin{defi} \label{Defi: function alpha}
    For a reduced graph $\Gamma$, by $\alpha: \text{Aut}^{\pi}(\Gamma) \rightarrow \text{Aut}^{\pi}(\Gamma)$ we understand a function given by the formula $\alpha( \tau ) = \tau^{-1} \circ \gamma(\tau)$. Moreover, if it is not clear to which graph $\alpha$ refers to, we write $\alpha_{\Gamma}$ when it refers to the graph $\Gamma$.
\end{defi}

\begin{obs} \label{Obs: basic properties of alpha}
    Let $\Gamma$ be a reduced graph. Then for every $\sigma,\tau \in \text{Aut}^{\pi}(\Gamma)$ we have $\gamma(\alpha(\sigma)) = {\alpha(\sigma)}^{-1}$ and $\alpha(\sigma \tau) = \tau^{-1} \alpha(\sigma) \tau \circ \alpha(\tau)$.
\end{obs}
\begin{proof}
    Direct calculations and an application of the fact that $\gamma^2=id$ (Observation \ref{Obs: basic properties of gamma}).
\end{proof}

\begin{lem} \label{Lem: connection between im(alpha) and isomorphisms of graphs}
    Let $\Gamma = (V,E)$ be a reduced graph. Let $\tau_0 \in \text{Sym}(V)$ be any permutation, $E' = \{ (v,w) \in V \times V \mid (v,\tau_0^{-1}(w)) \in E \}$ and $\Gamma' = (V,E')$. Then $\Gamma \cong \Gamma' \Leftrightarrow \tau_0 \in \text{im}(\alpha_{\Gamma})$.
\end{lem}
\begin{proof}
    We will start by proving $(\Leftarrow)$, since proof of $(\Rightarrow)$ will be mimicking the first one. 
    
    $(\Leftarrow)$: Let $\sigma \in \text{Aut}^{\pi}(\Gamma)$ be such that $\alpha(\sigma) = \tau_0$. Then $\gamma(\sigma) = \sigma \circ \tau_0$. By the definition of $\gamma$ we get that $E = \{ (\sigma(v),\sigma \circ \tau_0 (w)) \mid (v,w) \in E \}$. Then $\sigma^{-1}:V \rightarrow V$ is a desired isomorphism from $\Gamma$ to $\Gamma'$.
    
    $(\Rightarrow)$: Let us suppose that $\Gamma \cong \Gamma'$. Let $\sigma: V \rightarrow V$ be the isomorphism from $\Gamma$ to $\Gamma'$. Then we have following sequence of equivalences:
    $$
    (v,w) \in E \Leftrightarrow (\sigma(v),\sigma(w)) \in E' \Leftrightarrow (\sigma(v),\tau_0^{-1} \circ \sigma(w)) \in E
    $$
    Above proves that $(\sigma,\tau_0^{-1} \circ \sigma) \in \text{Aut}^{\text{TF}}(\Gamma)$ hence $(\sigma^{-1}, \sigma^{-1} \circ \tau_0) \in \text{Aut}^{\text{TF}}(\Gamma)$. By Definition \ref{Defi: function alpha} this gives us $\alpha(\sigma^{-1}) = \tau_0$ and finally $\tau_0 \in \text{im}(\alpha_{\Gamma})$.
\end{proof}

For any group $G$ and any element $g \in G$ by ${g}_r: G \rightarrow G$ we understand a function given by the formula $h \mapsto h\cdot g$.

\begin{cor} \label{Cor: (+m)_r in im(alpha) iff. Cay(Z_2m,S) iso Cay(Z_2m,S+m)}
    Let $m$ be a positive integer and let $\text{Cay}(\mathbb{Z}_{2m},S)$ be a Cayley graph. Following conditions are equivalent:
    \begin{enumerate}[i.]
        \item $\text{there exists } \sigma \in \text{Aut}^{\pi}(\text{Cay}(\mathbb{Z}_{2m},S)) \text{ such that } \alpha(\sigma) = {m}_{r}$;
        \item $\text{Cay}(\mathbb{Z}_{2m},S) \cong \text{Cay}(\mathbb{Z}_{2m},S+m)$.
    \end{enumerate}
\end{cor}
\begin{proof}
    Apply Lemma \ref{Lem: connection between im(alpha) and isomorphisms of graphs} with $\Gamma = \text{Cay}(\mathbb{Z}_{2m},S)$ and $\tau_0 = {m}_{r}$.
\end{proof}

\subsection{Schur rings} \label{SUBSECTION: Schur rings} \hfill \\
Schur rings play are the main tool in the proof of classification of unstable circulants of order $2p^e$ for any odd prime $p$ and arbitrary $e \geq 1$ \cite{StabilityAndSchurRings}. We will apply an important partial result from this paper to obtain more classifications of unstable circulants. \\
We will denote the identity element of group $G$ by $e_G$. Moreover for a given commutative ring $R$, subring $P$ and the set $S$ we denote the smallest subset of $R$ containing $S$ and closed under addition and multiplication by elements from $P$ by $\text{Span}_P S$. For a group $G$ and commutative ring $\mathbb{K}$, by $\mathbb{K}G$ we understand the ring which consists of formal sums of elements from $G$ with coefficients from $\mathbb{K}$ and multiplication given by group action of $G$. This ring is called a group ring. Also, for a given subset $X$ of a group $G$ by $\underline{X}$ we understand an element $\sum_{x \in X} x \in \mathbb{Z}G$.
%
%
%
%
%
%

\begin{defi} \label{Defi: Schur ring} (\cite[Chapter IV]{wielandt2014finite})
    A subring $\mathcal{A}$ of the integer group ring $\mathbb{Z}G$ is called a \textit{Schur ring over $G$} if there exists a partition $\mathcal{S}(\mathcal{A})$ of $G$ such that
    \begin{enumerate}[(1)]
        \item $\{e_G\} \in \mathcal{S}(\mathcal{A})$;
        \item if $X \in \mathcal{S}(\mathcal{A})$;
        \item $\mathcal{A} = \text{Span}_{\mathbb{Z}}\{ \underline{X} \mid X \in \mathcal{S}(\mathcal{A}) \}$.
    \end{enumerate}
Elements of $\mathcal{S}(\mathcal{A})$ are called the basic sets of $\mathcal{A}$.
\end{defi}

Now let us recall, that for a given group $G$ and any $g \in G$ function $g_r : G \rightarrow G$ is given by $h \mapsto h\cdot g$. The right multiplication representation of $G$ is $G_r = \{ g_r \mid g \in G \}$. Moreover if we have a group $A$ acting on $X$, for any $x \in X$ by $A_x$ we understand the subgroup made of elements $a \in A$ which satisfy $a.x = x$. This subgroup is often called the stabilizer of $x$. By $\text{Orb}(A,X)$ we understand the set of orbits of elements of $X$ under action of the group $A$. 

\begin{prop} \label{Thm: tranzitivity modules are Schur rings} (\cite{Schur1933})
    Let $A \leq \text{Sym}(G)$. If $G_r \leq A$, then the linear subspace of 
$\mathbb{Z}G$ given by $\text{Span}_{\mathbb{Z}}\{ \underline{X} \mid X \in \text{Orb}(A_{e_G},G) \}$ forms a subring.
\end{prop}

A Schur ring described in above theorem is usually called \textit{transitivity module over $G$ induced by $A$} and denoted $V(A_{e_G},G)$. If $\Gamma \cong \text{Cay}(H,S)$ for $S \subseteq H$ such that $S^{-1}=S$, then $\Gamma \times K_2 \cong \text{Cay}(H \times \langle a \rangle, Sa)$ where $\langle a \rangle \cong \mathbb{Z}_2$. 

For any Cayley graph $\Sigma \cong \text{Cay}(H,S)$ by $\mathcal{A}(\Sigma)$ we understand the transitivity module of ${\text{Aut}(\Sigma)}_{e_H}$, that is the ring $V({\text{Aut}(\Sigma)}_{e_H},H)$. To understand the reason for instability of a circulant $\Gamma = \text{Cay}(\mathbb{Z}_n,S)$ we will use Schur ring $\mathcal{A}(\Gamma \times K_2)$. From now on we will often associate vertices of $K_2$ with elements of the group $\langle a \rangle \cong \mathbb{Z}_2$ since $K_2 \cong \text{Cay}(\langle a \rangle,\{a\})$. 

Next observation shows how being unstable relates to properties of an associated Schur ring. Similar criterion was stated in \cite[Theorem~1.1]{StabilityAndSchurRings}.

\begin{obs} \label{Obs: unstable cayley graphs means {a} is not in the associated schur ring} 
    Let $\Gamma = \text{Cay}(H,S)$ be a connected nonbipartite Cayley graph. Then $\Gamma$ is unstable if and only if $\underline{ \{ a \} } \notin \mathcal{A}(\Gamma \times K_2)$.
\end{obs}
\begin{proof}
    ($\Rightarrow$): Since $\Gamma$ is connected and nonbipartie, every automorphism of $\Gamma \times K_2$ either preserves partition $\{H \times \{e_{\langle a \rangle}\}, H \times \{a\} \}$ of vertices setwise, or permutes its elements. Since the map $h a^i \mapsto h a^{i+1}$ is an automorphism of $\Gamma \times K_2$, we conclude that there exists a two-fold automorphism $(\sigma_1,\sigma_2) \in \text{Aut}^{\text{TF}}(\Gamma)$ such that $\sigma_1 \neq \sigma_2$. Therefore there exists some $g \in H$ such that $\sigma_1(g) \neq \sigma_2(g)$. If we now consider a two-fold automorphism $( g^{-1}_r {\left({ \sigma_1(g)} \right) }^{-1}_r \sigma_1 {g}_r, g^{-1}_r {\left({ \sigma_1(g)} \right) }^{-1}_r \sigma_2 {g}_r )$, then corresponding automorphism $\tau \in \text{Aut}(\Gamma \times K_2)$ is also an element of ${\text{Aut}(\Gamma \times K_2)}_{e_{H \times \langle a \rangle}}$ and $\tau(a) \neq a$. This shows us that indeed $\underline{\{ a \}} \notin \mathcal{A}(\Gamma \times K_2)$ by definition.

    ($\Leftarrow$): Since $\underline{ \{ a \} } \notin \mathcal{A}(\Gamma \times K_2)$, the basic set of $\mathcal{A}(\Gamma \times K_2)$ which contains $a$ have at least one different element, call it $a'$. This shows that there exist $\tau \in {\text{Aut}(\Gamma \times K_2)}_{e_{H \times \langle a \rangle}}$ such that $\tau(a) = a' \neq a$. Since by definition $\tau(e_{H \times \langle a \rangle}) = e_{H \times \langle a \rangle}$, $\tau$ is an unexpected symmetry which indicates instability of $\Gamma$.
\end{proof}

Before we state the lemma showing us the potential structure of $\mathcal{A}(\Gamma \times K_2)$ we have to state following definition.

\begin{defi}
    For any subset $X$ of $G$ we define a radical of $X$ by the formula 
    $$
    \text{rad}(X) = \{ g \in G \mid gX = Xg = X \}.
    $$ 
\end{defi}

Now we are ready to state the lemma which is of our interest.

\begin{thm} \label{Thm: Classification of schur rings over Z_2m x Z_2} (\cite[Theorem~1.4]{StabilityAndSchurRings})
    Let $G = H \times \langle a \rangle$, where $H \cong \mathbb{Z}_{2m}$ for an odd number $m>1$ and $\langle a \rangle \cong \mathbb{Z}_2$. If $\mathcal{A}$ is a Schur ring over $G$ with $\underline{H} \in \mathcal{A}$ and $\underline{\{a\}} \notin \mathcal{A}$, then $\{a,ab\}$ is a basic set of $\mathcal{A}$ or
    $$
    \bigcap_{ \substack{X \in \mathcal{S}(\mathcal{A}) \\
    X \cap H_0a \neq \varnothing} } \text{rad}(X \cap H_0a) \neq \{ e_G \},
    $$
    where $b$ is the unique involution of $H$ and $H_0$ is the unique subgroup of $H$ of order $m$.
\end{thm}

Now we can apply above lemma to the case relating to instability. We should also note that for abelian group we usually denote the action as ,,$+$'' instead of ,,$\cdot$''.

\begin{lem} \label{Lem: reasons for instability in case 2m, m odd -- messy second condition}
    Let $m>1$ be an odd integer and let $\Gamma = \text{Cay}(\mathbb{Z}_{2m},S)$ be connected nonbipartite unstable circulant. Then either
    \begin{enumerate}[i.]
        \item there exists nonzero $h \in 2\mathbb{Z}_{2m}$ such that $S \cap 2\mathbb{Z}_{2m} + h = S \cap 2\mathbb{Z}_{2m}$;
        \item or $\{a,m+a\}$ is a basic set of $\mathcal{A}(\Gamma \times K_2)$.
    \end{enumerate}
\end{lem}
\begin{proof}
    By Observation \ref{Obs: unstable cayley graphs means {a} is not in the associated schur ring}, $\underline{\{a\}} \notin \mathcal{A}(\Gamma \times K_2)$. Now if we put $H = \mathbb{Z}_{2m}$ and apply Theorem \ref{Thm: Classification of schur rings over Z_2m x Z_2} we get that either $\{a,m+a\}$ is a basic set of $\mathcal{A}(\Gamma \times K_2)$, which is our second condition in the statement of the lemma, or 
    $$
    V = \bigcap_{ \substack{X \in \mathcal{S}(\mathcal{A}(\Gamma \times K_2)) \\
    X \cap (2\mathbb{Z}_{2m} + a) \neq \varnothing} } \text{rad}(X \cap (2\mathbb{Z}_{2m} + a) ) \neq \{ e_{\mathbb{Z}_{2m} \times \langle a \rangle } \}.
    $$
    Note that $\Gamma \times K_2 = \text{Cay}(\mathbb{Z}_{2m} \times \langle a \rangle, S + a )$ and hence $\underline{S+a} \in \mathcal{A}(\Gamma \times K_2)$. This shows that $S+a = \bigcup_{  \substack{X \in \mathcal{S}(\mathcal{A}(\Gamma \times K_2)) \\
    X \subseteq S+a }  } X$. Notice that $V \leq 2\mathbb{Z}_{2m}$ and take some nonzero $h \in V$.
    
    We will show that $S \cap 2\mathbb{Z}_{2m} + h = S \cap 2\mathbb{Z}_{2m}$. Take arbitrary $s \in S \cap 2\mathbb{Z}_{2m}$. Then $s+a \in (S \cap 2\mathbb{Z}_{2m}) + a \subseteq S + a$, and hence there exist a basic set $X$ of $\mathcal{A}(\Gamma \times K_2)$ such that $s + a \in X \subseteq S + a$. Notice that $s+a \in X \cap (2\mathbb{Z}_{2m} + a)$ because $s \in S \cap 2\mathbb{Z}_{2m} \subseteq 2\mathbb{Z}_{2m}$. This shows that indeed $X \cap (2\mathbb{Z}_{2m} + a) ) \neq \{ e_{\mathbb{Z}_{2m} \times \langle a \rangle } \}$ and allows us to conclude that $X + h = X$.  Finally we get $s+h+a \in X + h = X \subseteq S+a$ which leads us to conclude that $s+h \in S$. Since $s,h \in 2\mathbb{Z}_{2m}$ we obviously have $s+h \in s \cap 2\mathbb{Z}_{2m}$ as therefore $S \cap 2\mathbb{Z}_{2m} + h \subseteq S \cap 2\mathbb{Z}_{2m}$. Since both sets have the same cardinality we get desired equality.
\end{proof}

\subsection{Connection between Schur rings and function $\alpha$ for Cayley graphs} \label{SUBSECTION: Connection between Schur rings and function alpha for Cayley graphs} \hfill \\

From now on we will mainly consider connected nonbipartite graphs $\Gamma = (V,E)$. In such a case all automorphism of $\Gamma \times K_2$ permute elements of a partition $\{V \times \{0\}, V \times \{1\} \}$ of its vertices. Because there is always an expected automorphism swithcing these parts - which can be constructed from the nontrivial automorphism of $K_2$, it is easy to spot that $\Gamma$ is unstable if and only if there exists a two-fold automorphism $(\sigma_1,\sigma_2) \in \text{Aut}^{\text{TF}}(\Gamma)$ such that $\sigma_1 \neq \sigma_2$.

Cayley graphs are the one of the main focus. Will now establish a connection between action of function $\alpha$ defined for reduced graphs in Definition \ref{Defi: function alpha} and the basic set containing $a$ in the Schur ring $\mathcal{A}(\Gamma \times K_2)$ when $\Gamma$ is a Cayley graph.

It is worth noting that for some results in this subsection one could use Schur ring theory instead of applying properties of $\alpha$, which would result in proofs of similar complexity.

\begin{prop} \label{prop: X_a and action of alpha}
    If $\Gamma = \text{Cay}(G,S)$ is connected, nonbipartite and reduced graph and $X_a$ is the basic set of $\mathcal{A}(\Gamma \times K_2)$ containing $a$ then for any $g \in G$ following conditions are equivalent:
    \begin{enumerate}[i.]
        \item $ga \in X_a$;
        \item there exists $\sigma \in \text{Aut}^{\pi}(\Gamma)$ such that $\alpha(\sigma).0 = g$;
        \item there exists $g_0 \in G$ and $\sigma \in \text{Aut}^{\pi}(\Gamma)$ such that $\alpha(\sigma).g_0 = g g_0$.
        \item for any $g_0 \in G$ there exists $\sigma \in \text{Aut}^{\pi}(\Gamma)$ such that $\alpha(\sigma).g_0 = g g_0$.
    \end{enumerate}
\end{prop}
\begin{proof}
    We will start by showing equivalence of conditions i., ii., then demonstrate equivalence between ii. and iii. and finish by showing that ii. is equivalent to iv.

    (i. $\Rightarrow$ ii.): Let $(\sigma_1,\sigma_2)$ be a two-fold automorphism which corresonds to an element of ${\text{Aut}(\Gamma \times K_2)}_{e_{G \times \langle a \rangle }}$ which maps $a$ to $ga$. Then $\alpha(\sigma_2^{-1}) = \sigma_1 \sigma_2^{-1}$, so $\alpha(\sigma_1).0 = \sigma_2 \sigma_1^{-1}.0 = \sigma_2.0 = g$ which ends the proof.
    
    (ii. $\Rightarrow$ i.): Consider $\tau = {\gamma(\sigma^{-1})} \circ { (\gamma(\sigma^{-1}).0)}_{r}$. Then
    $$
    \gamma(\tau) \tau^{-1} = \sigma^{-1} \circ \gamma({ (\gamma(\sigma^{-1}).0)}_{r}) \circ { (\gamma(\sigma^{-1}).0)}_{r}^{-1} \circ \gamma(\sigma) = \sigma^{-1} \circ \gamma(\sigma) = \alpha(\sigma).
    $$
    Now let us check that $\tau.0 = {\gamma(\sigma)}^{-1} \circ { (\gamma(\sigma^{-1}).0)}_{r}.0 = {\gamma(\sigma)}^{-1}.(\gamma(\sigma^{-1}).0) = 0$ as wanted, hence 
    $g = \alpha(\sigma).0 = \gamma(\tau) \tau^{-1}.0 = \gamma(\tau).0$ which means that indeed vertices $a$ and $ga$ are in the same orbit of ${\text{Aut}(\Gamma \times K_2)}_{e_{G \times \langle a \rangle }}$.

    (ii. $\Rightarrow$ iii.): Obvious.

    (iii. $\Rightarrow$ ii.): Put $\tau = \sigma {(g_0)}_r$ Then $\alpha(\tau) = {(g_0)}_r^{-1} \sigma^{-1} \gamma(\sigma) {(g_0)}_r = {(g_0)}_r^{-1} \alpha(\sigma) {(g_0)}_r$, hence $\alpha(\tau).0 = {(g_0)}_r^{-1} \alpha(\sigma) {(g_0)}_r.0 = {(g_0)}_r^{-1} \alpha(\sigma).g_0 = {(g_0)}_r^{-1}.(g g_0) = g$ as wanted.

    (iv. $\Rightarrow$ ii.): Obvious.
    
    (ii. $\Rightarrow$ iv.): Assume $\sigma$ is such that $\alpha(\sigma).0 = g$ and take arbitrary $g_0 \in G$. Put $\tau = {(g_0)}_r^{-1} \circ \sigma$. Then $\alpha(\tau) = {(g_0)}_r \circ \alpha(\sigma) \circ {(g_0)}_r^{-1}$, hence $\alpha(\tau).g_0 = g g_0$.
\end{proof}

From Schur ring theory we already have a lot of information about $\Gamma$ when it is an connected, nonbipartite and unstable circulant of order $2m$ where $m>1$ is an odd integer. We will now focus on the circulants for which Lemma \ref{Lem: reasons for instability in case 2m, m odd -- messy second condition} does not provide explicit information. For that reason from now on we will work under the following hypothesis.


\begin{hyp} \label{HYPOTHESIS which states the hard case}
    Let $n=2m$ for some odd integer $m>1$.  Also assume $\Gamma = \text{Cay}(\mathbb{Z}_{n},S)$ is connected, nonbipartite, reduced, unstable, and $ \{a,m + a\} $ is a basic set of $\mathcal{A}(\Gamma \times K_2)$.
\end{hyp}

Our final goal is to show that in above case one have $\text{Cay}(\mathbb{Z}_{n},S) \cong \text{Cay}(\mathbb{Z}_{n},S+\frac{n}{2})$ under certain assumptions. By Corollary \ref{Cor: (+m)_r in im(alpha) iff. Cay(Z_2m,S) iso Cay(Z_2m,S+m)} to achieve this we have to show that there exists $\sigma \in \text{Aut}^{\pi}(\text{Cay}(\mathbb{Z}_{2m},S))$ such that $\alpha(\sigma) = {m}_{r}$.


\begin{cor} \label{cor: how alpha works on vertices in our case}
    Under Hypothesis \ref{HYPOTHESIS which states the hard case} for every $\sigma \in \text{Aut}^{\pi}(\Gamma)$ and every $x \in \mathbb{Z}_n$ one get $\alpha(\sigma).x \in \{x,x+m\}$.
\end{cor}
\begin{proof}
    By Hypothesis \ref{HYPOTHESIS which states the hard case} we know that $\{a,m+a\}$ is a basic set of $\mathcal{A}(\Gamma \times K_2)$. To end the proof we use equivalence between statements i. and iii. from Proposition \ref{prop: X_a and action of alpha}.
\end{proof}

To proceed further we have to define the concept of invariant partitions and block systems.

\begin{defi} \label{defi: invariant partition block system}
    Let $G$ be a group acting on set $X$. \textit{Invariant partition of $X$ with respect to the action of $G$} is a partition $\mathcal{P}$ such that for every $P \in \mathcal{P}$ and every $g \in G$ one has $g[P] = P$. If moreover $G$ acts transitively on $X$, invariant partitions are called \textit{block systems} and their elements are called \textit{blocks}.
\end{defi}

\begin{obs} \label{obs: cosets of L and widetile{L} form block systems}
    Assume Hypothesis \ref{HYPOTHESIS which states the hard case}. Then cosets of the group $L = \{0,m\}$ form a block system of $\mathbb{Z}_n$ with respect to $\text{Aut}^{\pi}(\Gamma)$ called $\mathcal{L}$. Moreover, cosets of $\widetilde{L} = L\langle a \rangle$ form a block system with respect to $\text{Aut}(\Gamma \times K_2)$ called $\widetilde{\mathcal{L}}$.
\end{obs}
\begin{proof}
Choose arbitrary $x \in \mathbb{Z}_n$. Since $\{a,m+a\}$ is a basic set of $\mathcal{A}(\Gamma \times K_2)$, by equivalence of criterion i. and iv. in Proposition \ref{prop: X_a and action of alpha} there exists such $\sigma \in \text{Aut}^{\pi}(\Gamma)$ that $\alpha(\sigma).x = x + m$.

Let $\tau$ be an arbitrary element of $\text{Aut}^{\pi}(\Gamma)$. Putting $\tau^{-1}$ instead of $\tau$ in the equation from Observation \ref{Obs: basic properties of alpha} shows that $\alpha(\sigma \tau^{-1}) {\alpha(\tau^{-1})}^{-1} = \tau \alpha(\sigma) \tau^{-1}$. Therefore we know that
$$
\tau.(x+m) = \tau \alpha(\sigma) \tau^{-1}.(\tau.x) = \alpha(\sigma \tau^{-1}) {\alpha(\tau^{-1})}^{-1}.(\tau.x) \in \{\tau.x, \tau.x + m\}.
$$
Last inclusion is due to Corollary \ref{cor: how alpha works on vertices in our case}. Since $\tau.(x+m) \neq \tau.x$, $\tau.(x+m) = \tau.x + m$ which proves that $\mathcal{L}$ indeed forms a block system.

To end the proof we have to moreover prove that $\widetilde{\mathcal{L}}$ is a block system with respect to $\text{Aut}(\Gamma \times K_2)$. First of all let us notice that cosets of subgroup $\mathbb{Z}_n$ of the group $\mathbb{Z}_n \times \langle a \rangle$ form a block system by Hypothesis \ref{HYPOTHESIS which states the hard case}. One can easily see that a function $a_r$ preserves cosets of $\widetilde{L}$ setwise, hence we can only have to check if automorphism induced by two-fold automorphism also preserve partition $\widetilde{\mathcal{L}}$.

By definition, for every $\sigma \in \text{Aut}^{\pi}(\Gamma)$, $(\sigma, \sigma \circ \alpha(\sigma))$ is a two-fold automorphism. Let us define permutations
$$
    \tau_1.x = \begin{cases} 
      x & \text{ if } x \in \mathbb{Z}_n \\
    \alpha(\sigma).(xa) a & \text{ if } x \in \mathbb{Z}_n a
   \end{cases} \quad \text{and} \quad
    \tau_2.x = \begin{cases} 
      \sigma.x & \text{ if } x \in \mathbb{Z}_n \\
    \sigma.(xa) a & \text{ if } x \in \mathbb{Z}_n a
   \end{cases}.
$$
Notice that automorphism created by $(\sigma, \sigma \circ \alpha(\sigma))$ is exactly $\tau_2 \circ \tau_1$. $\tau_1$ preserves all of the cosets of $L$ by Corollary \ref{cor: how alpha works on vertices in our case}, hence it also preserves cosets of $\widetilde{L} = L\langle a \rangle$ setwise. Since we know that cosets of $L$ form a block system with respect to $\text{Aut}^{\pi}(\Gamma)$ and $\tau_2$ acts uniformly on both $\mathbb{Z}_n$ and $\mathbb{Z}_n a$, we conclude that $\tau_2$ permutes cosets of $L\langle a \rangle$. Combination of above facts shows that indeed cosets of $\widetilde{L}$ form a block system with respect to $\text{Aut}(\Gamma \times K_2)$.
\end{proof}



Now we can divide edges of $\Gamma = \text{Cay}(\mathbb{Z}_n,S)$ into two disjoint parts, which cannot be mixed by any automorphism of $\Gamma \times K_2$. Let us define the set $S_r = \{s \in S \mid s +m \in S \}$ of \textit{reflective connection set} and the complementary set $S_a = S \backslash S_r$ of \textit{anti-reflective connection set}. We call an edge $(x,y)$ of $\Gamma$ \textit{reflective} if $y-x \in S_r$ and \textit{anti-reflective} otherwise. Before stating next observation we need a couple more definitions.

\begin{defi} \label{defi: reflective and areflective edges and colored quotient by L}
Assume $\Gamma$ satisfies Hypothesis\ref{HYPOTHESIS which states the hard case}. We define a colored graph $\bigslant{\Gamma}{\mathcal{L}}$ to have vertex set $\mathcal{L}$ and edges in the first (reflective) color between vertices $L+x$ and $L+y$ when $y-x \in S_r$ and edges in second (anti-reflective) color between vertices $L+x$ and $L+y$ when $\{ y-x, y-x + m\} \cap S_a \neq \varnothing$. We will often refer to the edges in the first color as \textit{reflective} and to the ones in the second color as \textit{anti-reflective}.
\end{defi}

\begin{defi} \label{defi: induced action on the block system}
    For any set $X$, $\sigma \in \text{Sym}(X)$ and a partition $\mathcal{P}$ which is invariant with respect to $\langle \sigma \rangle$, we define $\text{ind}_{\mathcal{P}} (\sigma)$ to be a permutation of $\mathcal{P}$ given by the formula $\mathfrak{p} \mapsto \sigma[\mathfrak{p}]$. We call it \textit{the permutation of $\mathcal{P}$ induced by $\sigma$}.
\end{defi}

\begin{obs} \label{obs: tf-projections induce automorphisms of Gamma / mathcal{L}}
    Assume Hypothesis \ref{HYPOTHESIS which states the hard case}. For any element $\sigma \in \text{Aut}^{\pi}(\Gamma)$ we obtain that $\text{ind}_{\mathcal{L}} (\sigma) \in \text{Aut}(\bigslant{\Gamma}{\mathcal{L}})$.
\end{obs}
\begin{proof}
    Assume that vertices $L+x$ and $L+y$ are connected by the reflective edge in $\bigslant{\Gamma}{\mathcal{L}}$. Then there are exactly $4$ edges between sets $\widetilde{L}+x$ and $\widetilde{L}+y$ in the graph $\Gamma \times K_2$. Similarly when vertices $L+x$ and $L+y$ are connected by the anti-reflective edge in $\bigslant{\Gamma}{\mathcal{L}}$, there are exactly $2$ edges between sets $\widetilde{L}+x$ and $\widetilde{L}+y$ in the graph $\Gamma \times K_2$. Now note that the function $L+x \mapsto \widetilde{L}+x$ is a bijection between $\mathcal{L}$ and $\widetilde{\mathcal{L}}$ such that reflective edges of $\bigslant{\Gamma}{\mathcal{L}}$ are mapped onto pairs with exactly $4$ edges between them and anti-reflective edges are mapped onto pairs with exactly $2$ edges between them. We also know that $(\sigma,\gamma(\sigma))$ corresponds to the automorphism of $\Gamma \times K_2$ which by Observation \ref{obs: cosets of L and widetile{L} form block systems} permutes cosets of $\widetilde{L}$. Therefore this automorphism preserves number of edges between cosets of $\widetilde{L}$ and we indeed get $\text{ind}_{\mathcal{L}} (\sigma) \in \text{Aut}(\bigslant{\Gamma}{\mathcal{L}})$.
\end{proof}


\begin{lem}
    Assume Hypothesis \ref{HYPOTHESIS which states the hard case} and let $C = L \langle S_a \rangle$. Then cosets of $C$ form a block system of $\mathbb{Z}_n$ with respect to $\text{Aut}^{\pi}(\Gamma)$ called $\mathcal{C}$.
\end{lem}
\begin{proof}
    By definition $L \leq C$. Since elements of the set $S_a$ give rise to anti-reflective edges in $\Gamma$ and later to anti-reflective edges of $\bigslant{\Gamma}{\mathcal{L}}$, partition of $\mathcal{L}$ into cosets of $\bigslant{C}{L}$ is just the partition of vertices of $\bigslant{\Gamma}{\mathcal{L}}$ into connected components with respect to edges of the second color (anti-reflective ones). Take arbitrary $\sigma \in \text{Aut}^{\pi}(\Gamma)$. By Observation \ref{obs: tf-projections induce automorphisms of Gamma / mathcal{L}} we know that $\text{ind}_{\mathcal{L}} (\sigma)$ permutes cosets of $\bigslant{C}{L}$, hence $\sigma$ permutes cosets of $C$ as wanted.
\end{proof}

Another crucial concept in our proof are so called $\alpha$-homogeneous partitions.

\begin{defi}
    For a graph $\Gamma$ satisfying Hypothesis \ref{HYPOTHESIS which states the hard case} we call a partition $\mathcal{P}$ of the set
$\mathbb{Z}_n$ \textit{$\alpha$-homogeneous} when for each $P \in \mathcal{P}$ and each $\sigma \in \text{Aut}^{\pi}(\Gamma)$ it holds that $\alpha(\sigma).x - x$ is constant over all $x \in P$.
\end{defi}


\begin{lem}
    Under Hypothesis \ref{HYPOTHESIS which states the hard case} block system $\mathcal{C}$ is an $\alpha$-homogeneous partition.
\end{lem}
\begin{proof}
    Observe that a partition into cosets of $L$ is $\alpha$-homogeneous, since for any $\sigma \in \text{Aut}^{\pi}(\Gamma)$ we obtain $\alpha(\sigma).x , \alpha(\sigma).(x+m) \in \{x,x+m\}$, hence either $\alpha(\sigma).x =x$ and $\alpha(\sigma).(x+m)=x+m$ or $\alpha(\sigma).x =x+m$ and $\alpha(\sigma).(x+m)=x$. To end the proof it is enough to show that for any $s \in S_a$ cosets $L+x$ and $L+x+s$ act alike by $\alpha(\sigma)$ for arbitrary $\sigma \in \text{Aut}^{\pi}(\Gamma)$.

    Recall, that by Observation \ref{Obs: basic properties of alpha} $\gamma(\alpha(\sigma)) = {\alpha(\sigma)}^{-1}$, and by Corollary \ref{cor: how alpha works on vertices in our case} ${\alpha(\sigma)}^{-1} = \alpha(\sigma)$, hence $\gamma(\alpha(\sigma)) = {\alpha(\sigma)}$. Now applying Observation \ref{Obs: basic properties of gamma} yields that $\alpha(\sigma) \in \text{Aut}(\Gamma)$. Now since $s \in S$ and $s+m \notin S$ by definition of $S_a$, the fact that $\alpha(\sigma)$ is an automorphism of $\Gamma$ forces the fact that it acts alike on both $L+x$ and $L+x+s$.
\end{proof}


Next observation alerts us that value of function $\alpha$ does not necessarily depend on the full permutation $\sigma \in \text{Aut}^{\pi}(\Gamma)$ but only on the potentially small part of it.

\begin{lem} \label{lem: alpha |C depends only on action on C}
    \textbf{(Local behavior of function $\alpha$)} Assume Hypothesis \ref{HYPOTHESIS which states the hard case}. Take any $\mathfrak{c} \in \mathcal{C}$. If $\sigma_1,\sigma_2 \in \text{Aut}^{\pi}(\Gamma)$ are such that $\sigma_1 {|}_{\mathfrak{c}} = \sigma_2 {|}_{\mathfrak{c}}$, then $\alpha(\sigma_1) {|}_{\mathfrak{c}} = \alpha(\sigma_2) {|}_{\mathfrak{c}}$.
\end{lem}
\begin{proof}
    First of all let us take some $x \in \mathfrak{c}$. Now put $\tau_i = {(\sigma_i(x) - x)}_r^{-1} \circ \sigma_i$ for $i \in \{1,2\}$.
    Now calculate that $\alpha(\tau_i) = \sigma_i^{-1} \circ {(\sigma_i(x) - x)}_r^{-1} \circ {(\sigma_i(x) - x)}_r^{-1} \circ \gamma(\sigma_i) = \sigma_i^{-1} \circ \gamma(\sigma_i) = \alpha(\sigma_i)$ for $i \in \{1,2\}$. Since $\sigma_1(x) = \sigma_2(x)$, we see that it is enough to show $\alpha(\tau_1) {|}_{\mathfrak{c}} = \alpha(\tau_2) {|}_{\mathfrak{c}}$.
    
    Let $\Sigma$ be the subgraph of $\Gamma$ induced on $\mathfrak{c}$. Since $\mathfrak{c} = C + x$, $\Sigma$ is isomorphic to $\text{Cay}(C,S \cap C)$ and this isomorphism is given by $y \mapsto y-x$ for any $y \in C + x$.
    
    Now notice that $\tau_i[\mathfrak{c}] = \mathfrak{c}$, and since $L \leq C$, we know that indeed symmetries of $\Gamma \times K_2$ induced by two-fold automorphisms $(\tau_i,\gamma(\tau_i))$ also give a symmetry on the induced subgraph $\Sigma \times K_2$ on vertex set $\mathfrak{c} +\{0,a\}$. This leads us to conclude that $( \tau_i {|}_{\mathfrak{c}},\gamma(\tau_i) {|}_{\mathfrak{c}} )$ are two-fold automorphisms of $\Sigma$ for $i \in \{1,2\}$. Since $\tau_1 {|}_{\mathfrak{c}} = \tau_2 {|}_{\mathfrak{c}}$ and $\tau_1[\mathfrak{c}] = \mathfrak{c}$, we obtain $\gamma(\tau_i) {|}_{\mathfrak{c}} = \alpha(\tau_i) {|}_{\mathfrak{c}} \circ \tau_1 {|}_{\mathfrak{c}}$. We can therefore conclude, that
    $$
    (id_{\mathfrak{c}},\alpha(\tau_1) {|}_{\mathfrak{c}} \circ  \alpha(\tau_2) {|}_{\mathfrak{c}}^{-1} ) = ( \tau_1 {|}_{\mathfrak{c}} \circ \tau_2 {|}_{\mathfrak{c}}^{-1}, \alpha(\tau_1) {|}_{\mathfrak{c}} \circ \tau_1 {|}_{\mathfrak{c}}  \circ \tau_2 {|}_{\mathfrak{c}}^{-1} \circ \alpha(\tau_2) {|}_{\mathfrak{c}}^{-1} ) \in \text{Aut}^{\text{TF}}(\Sigma).
    $$ 
    Now assume by contradiction that $\alpha(\tau_1) {|}_{\mathfrak{c}} \neq  \alpha(\tau_2) {|}_{\mathfrak{c}}$. Then there exists such $y \in \mathfrak{c}$ that $\alpha(\tau_1) {|}_{\mathfrak{c}} \circ  \alpha(\tau_2) {|}_{\mathfrak{c}}^{-1}.y \neq y$. By Corollary \ref{cor: how alpha works on vertices in our case} this means that $\alpha(\tau_1) {|}_{\mathfrak{c}} \circ  \alpha(\tau_2) {|}_{\mathfrak{c}}^{-1}.y = y+m$. Since $(id_{\mathfrak{c}},\alpha(\tau_1) {|}_{\mathfrak{c}} \circ  \alpha(\tau_2) {|}_{\mathfrak{c}}^{-1} ) \in \text{Aut}^{\text{TF}}(\Sigma)$ that would mean that vertices $y$ and $y+m$ have the same neighbourhoods in $\Sigma$. This yields a contradiction, since that would mean $S \cap C + m = S \cap C$. On the other hand we know that $S_a \subseteq S \cap C$ and $S_a \neq \varnothing$ because by Hypothesis \ref{HYPOTHESIS which states the hard case} $\Gamma$ is reduced hence there is an element $s \in S_a \subseteq S \cap C$ such that $s+m \notin S$, which proves that actually $s+m \notin S \cap C$ and demonstrates that $S \cap C + m \neq S \cap C$.
\end{proof}

We end this section with definition of a partition $\mathcal{B}$ which will turn out to be a block system. This block system will play a central role in the proof of the main theorem of this paper. Before introducing it we have to define certain partial order on the family of partitions of a given set.

\begin{defi}
    For a given set $X$ and its partitions $\mathcal{P}$, $\mathcal{Q}$ we say that \textit{partition $\mathcal{P}$ is a fragmentation of $\mathcal{Q}$} or equivalently \textit{partition $\mathcal{Q}$ is a thickening of $\mathcal{P}$} when for each $\mathfrak{p} \in \mathcal{P}$ there exists $\mathfrak{q} \in \mathcal{Q}$ such that $\mathfrak{p} \subseteq \mathfrak{q}$. We denote this partial order by $\mathcal{P} \prec \mathcal{Q}$.
\end{defi}

\begin{lem} \label{lem: mathcal{B} forms a block system made of cosets of B}
    Assume Hypothesis \ref{HYPOTHESIS which states the hard case}. Define $\mathcal{B}$ to be the thickest $\alpha$-homogeneous partition (maximal one with respect to $\prec$ among $\alpha$-homogeneous partitions). Then $\mathcal{B}$ is a block system and its elements are cosets of a certain subgroup $B \leq \mathbb{Z}_n$.
\end{lem}
\begin{proof}
    Let $\sim$ be a relation on $\mathbb{Z}_n$ such that $x \sim y$ if and only if for every $\sigma \in \text{Aut}^{\pi}(\Gamma)$, $\alpha(\sigma).x - x = \alpha(\sigma).y - y$. Once can easily see that $\sim$ is an equivalence relation, its equivalence classes form an $\alpha$-homogeneous partition called $\mathcal{B}$. Now notice, that for any $\alpha$-homogeneous partition $\mathcal{P}$ and any $\mathfrak{p} \in \mathcal{P}$, any pair ov vertices $x,y \in \mathfrak{p}$ satisfies $x \sim y$ by definition of $\alpha$-homogeneous partitions. This proves that indeed $\mathcal{P} \prec \mathcal{B}$ so $\mathcal{B}$ is the thickest among $\alpha$-homogeneous partitions.

    Now we will prove that $\mathcal{B}$ is a block system. Assume otherwise, that is there are vertices $x \sim y$ and $\tau \in \text{Aut}^{\pi}(\Gamma)$ such that $\tau.x \nsim \tau.y$. This means that there exist $\sigma \in \text{Aut}^{\pi}(\Gamma)$ such that $\alpha(\sigma).(\tau.x) - \tau.x \neq \alpha(\sigma).(\tau.y) - \tau.y$. Without loss of generality, because of Corollary \ref{cor: how alpha works on vertices in our case} we can assume that $\alpha(\sigma).(\tau.x) = \tau.x$ and $\alpha(\sigma).(\tau.y) = \tau.y + m$.
    
    By Observation \ref{Obs: basic properties of alpha} we get that $\alpha(\sigma \tau) \circ {\alpha(\tau)}^{-1} = \tau^{-1} \alpha(\sigma) \tau$. Therefore, on one hand $\left( \alpha(\sigma \tau) \circ {\alpha(\tau)}^{-1} \right).x - x = \left( \alpha(\sigma \tau) \circ {\alpha(\tau)}^{-1} \right).y - y$ and on the other hand $\tau^{-1} \alpha(\sigma) \tau.x = \tau^{-1} \alpha(\sigma).(\tau.x) = \tau^{-1}.(\tau.x) = x$ and $\tau^{-1} \alpha(\sigma) \tau.y = \tau^{-1} \alpha(\sigma).(\tau.y) = \tau^{-1}.(\tau.(y) + m) = y + m$, where last equality follows from Observation \ref{obs: cosets of L and widetile{L} form block systems}. Now we verify that $\left( \alpha(\sigma \tau) \circ {\alpha(\tau)}^{-1} \right).x - x = \tau^{-1} \alpha(\sigma) \tau.x - x = 0$ and $\left( \alpha(\sigma \tau) \circ {\alpha(\tau)}^{-1} \right).y - y = \tau^{-1} \alpha(\sigma) \tau.y - y = m$ which yields a desired contradiction, proving that $\mathcal{B}$ indeed is a block system of $\mathbb{Z}_n$ with respect to the group $\text{Aut}^{\pi}(\Gamma)$.

    Since ${(\mathbb{Z}_n)}_{r} \leq \text{Aut}^{\pi}(\Gamma)$, subset $\mathfrak{b} \in \mathcal{B}$ such that $0 \in \mathfrak{b}$ has to be a subgroup and other elements of $\mathcal{B}$ need to be its cosets. From now on we will refer to this subgroup as $B$.
\end{proof}

\section{Chain automorphisms} \label{SECTION: Chain automorphisms}


\begin{defi}
    A \textit{colored directed graph} is a pair $(V,(E_1,\ldots E_k))$ where $V$ is a given finite set, and each of $E_i$-s are binary relations on $V$. Unlike in simple graph, we do not add any restrictions on these relations. We say that there is a directed edge in color $i$ between vertices $v$ and $w$ when $v E_i w$.

    Moreover an \textit{in-neighbourhood in color $j$ of vertex $v$} is a set $N_{\text{in},j}v = \{ w \in V \mid w E_j v \}$ and \textit{out-neighbourhood in color $j$ of vertex $v$} is a set $N_{\text{out},j}v = \{ w \in V \mid v E_j w \}$. We will usually refer to the \textit{in-neighbourhood of vertex $v$}, which is a $k$-tuple $N_{\text{in}}v = (N_{\text{in},1}v,\ldots,N_{\text{in},k}v)$ usually understood as a colored set. Similarly we define the \textit{out-neighbourhood of vertex $v$} as $N_{\text{out}}v = (N_{\text{in},1}v,\ldots,N_{\text{in},k}v)$.

    We also call a colored digraph \textit{reduced} when each pair of different vertices have different both colored sets of in- and out-neighbourhoods.
\end{defi}

\begin{defi}
    For a colored directed graph $\Gamma = (V,(E_1,\ldots,E_k))$, \textit{chain of $\Gamma$} is an infinite colored digraph with vertex set $V \times \mathbb{Z}$ and edge relations $E_{j,\text{Ch}} = \{ ((v,i),(w,i+1)) \mid (v,w) \in E_j \text{ and } i \in \mathbb{Z} \}$. This infinite colored digraph is denoted $\text{Ch}(\Gamma)$.
\end{defi}

\begin{defi}
    Let define a group $\text{Aut}_{\text{Ch}}(\Gamma)$ to be the group of sequences ${\{\sigma_i\}}_{i \in \mathbb{Z}}$ made of permutations of $V$ such that a function $\widetilde{\sigma}:V \times \mathbb{Z} \rightarrow V \times \mathbb{Z}$ given by formula $(v,i) \mapsto (\sigma_i(v),i)$ is an automorphism of $\text{Ch}(\Gamma)$. Moreover we define $\text{Aut}^{\text{Ch}}(\Gamma) = \{\sigma \in \text{Sym}(V) \mid \exists \text{ } {\{\sigma_i\}}_{i \in \mathbb{Z}} \in \text{Aut}_{\text{Ch}}(\Gamma) \text{ such that } \sigma = \sigma_0 \}$ and by $\pi_j:\text{Aut}_{\text{Ch}}(\Gamma) \rightarrow \text{Aut}^{\text{Ch}}(\Gamma)$ we define a function given by ${\{\sigma_i\}}_{i \in \mathbb{Z}} \mapsto \sigma_j$. If the graph to which $\pi_j$ refers to is not clear from the context we write $\pi_{j,\Gamma}$. We call elements of $\text{Aut}_{\text{Ch}}(\Gamma)$ \textit{chain automorphisms of $\Gamma$} and elements of $\text{Aut}^{\text{Ch}}(\Gamma)$ \textit{chain projections of $\Gamma$}.
\end{defi}

\begin{defi}
    We call a (directed and colored) graph reduced iff. each pair of different vertices have different sets of both colored in-neighbourhoods and colored out-neighbourhoods.
\end{defi}

\begin{obs} \label{obs: Gamma reduced implies Ch(Gamma) reduced}
    Let $\Gamma$ be reduced. Then for every $j \in \mathbb{Z}$, each pair of vertices in the set $V \times \{j\}$ have distinct in- and out-neighbourhoods in $\text{Ch}(\Gamma)$.
\end{obs}
\begin{proof}
    Let $\Gamma$ be the colored digraph on vertex set $V$. For a vertex $(v,j)$ of $\text{Ch}(\Gamma)$ we have $N_{\text{in}}(v,j) = N_{\text{in}}v \times \{j-1\}$. Therefore if vertices $(v,j)$ and $(w,j)$ have the same in-neighbourhoods then $N_{\text{in}}v = N_{\text{in}}w$ which contradicts the fact that $\Gamma$ was reduced.
\end{proof}

\begin{defi}
    A \textit{colored directed Cayley graph over the group $G$ with colored connection sets $S_1,\ldots,S_k$} is a colored digraph with vertex set $V=G$ and colored edges given by formulas $E_i = \{  (x,y) \in G \times G \mid yx^{-1} \in S_i \}$ for each $i \in [k]$. We will refer to this colored digraph by $\text{DiCay}(G,S_1,\ldots,S_k)$.
\end{defi}

\begin{obs}
    Let $\Gamma = \text{DiCay}(\mathbb{Z}_m,S_1\ldots S_k)$ for some $m \in \mathbb{Z}_{>0}$ and $S_1\ldots S_k \subseteq \mathbb{Z}_m$. Then $\Gamma$ is reduced if and only if there does not exist nonzero $h \in \mathbb{Z}_m$ such that for all colors $i \in [k]$ one have $S_i + h = S_i$.
\end{obs}
\begin{proof}
    We will prove both directions by contradiction.
    
    $(\Rightarrow)$: Assume $h \in \mathbb{Z}_m$ is a nonzero element satisfying $S_i + h = S_i$ for all colors $i \in [k]$. Then $N_{\text{out}} (-h) = N_{\text{out}} (0)$, hence contradiction.

    $(\Leftarrow)$: Assume at first that vertices $x$ and $y$ have the same out-neighbourhoods. By definition of a directed Cayley graph this means that for every $i \in [k]$ we have $S_i + x = S_i + y$, hence $S_i + (x-y) = S_i$. Since $x \neq y$ putting $h=x-y$ ends the proof. 

    Second case to consider is that vertices $x$ and $y$ have the same in-neighbourhoods. Then for every $i \in [k]$ we have $x - S_i = y - S_i$, hence $S_i - x = S_i - y$ and finally $S_i = S_i + (x-y)$ so we can once again put $h = x-y$.
\end{proof}

\begin{obs} \label{OBS: Chain automorphisms are determined by action on V x {0} for reduced graphs}
    Let $\Gamma$ be a directed colored and reduced graph. Then for every $j \in \mathbb{Z}$ homomorphism $\pi_j$ is an isomorphism.
\end{obs}
\begin{proof}
    Let study the kernel of $\pi_j$. Assume $\sigma_i = id$ for arbitraty $i \in \mathbb{Z}$. Then if $\widetilde{\sigma}$ was not identity on $V \times \{i-1\}$ or on $V \times \{i+1\}$, it would contradict Observation \ref{obs: Gamma reduced implies Ch(Gamma) reduced}. Therefore if $\sigma_j = id$, it follows that $\forall \text{ } i \in \mathbb{Z} \text{ } \sigma_i = id$, hence kernel is trivial an our thesis follows as $\pi_j$ is surjective by definition.
\end{proof}

\begin{defi} \label{Defi: widetilde{gamma} and gamma for colored directed reduced graphs}
     Let $\Gamma$ be a directed colored graph. Then $\widetilde{\gamma}: \text{Aut}_{\text{Ch}}(\Gamma) \rightarrow \text{Aut}_{\text{Ch}}(\Gamma)$ is a function given by a formula ${\{\sigma_i\}}_{i \in \mathbb{Z}} \mapsto {\{\sigma_{i+1}\}}_{i \in \mathbb{Z}}$. Moreover, if this graph is reduced, then $\gamma: \text{Aut}^{\text{Ch}}(\Gamma) \rightarrow \text{Aut}^{\text{Ch}}(\Gamma)$ is the unique function satisfying $\gamma \circ \pi_j = \pi_j \circ \widetilde{\gamma}$ for all $j \in \mathbb{Z}$, that is making the bellow diagram commute.
\[\begin{tikzcd}
	{\text{Aut}_{\text{Ch}}(\Gamma)} & {\text{Aut}_{\text{Ch}}(\Gamma)} \\
	{\text{Aut}^{\text{Ch}}(\Gamma)} & {\text{Aut}^{\text{Ch}}(\Gamma)}
	\arrow["{\widetilde{\gamma}}", from=1-1, to=1-2]
	\arrow["{\pi_j}", from=1-1, to=2-1]
	\arrow["{\pi_j}", from=1-2, to=2-2]
	\arrow["{ \exists ! \text{ }\gamma}", dotted, from=2-1, to=2-2]
\end{tikzcd}\]
    If it is not clear to which graph $\widetilde{\gamma}$ or $\gamma$ refers to, we write ${\widetilde{\gamma}}_{\Gamma}$ and $\gamma_{\Gamma}$ respectively when it refers to $\Gamma$.
\end{defi}

\begin{prop}
    Automorphism $\gamma$ is well defined for reduced directed colored graphs.
\end{prop}
\begin{proof}
    Follows easily from Observation \ref{OBS: Chain automorphisms are determined by action on V x {0} for reduced graphs}.
\end{proof}

\begin{prop} \label{prop: equivalence of both definitions of gamma}
    Let $\Gamma$ be a reduced graph. Then function $(\pi_0,\pi_1): \text{Aut}_{\text{Ch}}(\Gamma) \mapsto \text{Aut}^{\text{TF}}(\Gamma)$ given by formula ${\{ \sigma_i \}}_{i \in \mathbb{Z}} \mapsto (\sigma_0, \sigma_1)$ is an isomorphism which makes both definitions of $\widetilde{\gamma}$ (Definition \ref{Defi: widetilde{gamma} and gamma for undirected reduced graphs} and Definition \ref{Defi: widetilde{gamma} and gamma for colored directed reduced graphs}) equivalent. Moreover $\text{Aut}^{\text{Ch}}(\Gamma) = \text{Aut}^{\pi}(\Gamma)$ and both definitions of $\gamma$ are also equivalent.
\end{prop}
\begin{proof}
    Let $\Gamma = (V,E)$. First of all note that if ${\{ \sigma_i \}}_{i \in \mathbb{Z}}$ gives an automorphism of $\text{Ch}(\Gamma)$, then $(\sigma_{j},\sigma_{j+1})$ gives an automorphism of a subgraph of $\text{Ch}(\Gamma)$ induced by $V \times \{j,j+1\}$ and moreover since sets $V \times \{i\}$ are always fixed setwise, it also gives an automorphism of the graph which forgets orientation of edges, which is isomorphic to $\Gamma \times K_2$. This proves that indeed $(\sigma_j,\sigma_{j+1}) \in \text{Aut}^{\text{TF}}(\Gamma)$. It is easily visible now, that since $(\sigma_j,\sigma_{j+1})$, $(\sigma_{j+1},\sigma_{j+2})$ and hence $(\sigma_{j+2},\sigma_{j+1})$ are two-fold automorphisms of a reduced graph $\Gamma$, we can conclude $\sigma_j = \sigma_{j+2}$ for any $j \in \mathbb{Z}$. This proves that $(\pi_0,\pi_1)$ indeed is an isomorphism and the rest of the proposition follows by definitions.
\end{proof}

    We denote that two natural numbers $x$ and $y$ are coprime by  writing $x \perp y$. Now we will establish basic facts about $\gamma$ when $\Gamma = \text{DiCay}(\mathbb{Z}_m,S_1\ldots S_k)$. For that we will need the following definitions. For any $l \perp m$ let $\psi_l: \mathbb{Z}_m \rightarrow \mathbb{Z}_m$ be a function defined by the formula $x \mapsto lx$. Also take any $\psi \in \text{Sym}(X)$. By $\iota(\psi): \text{Sym}(X) \rightarrow \text{Sym}(X)$ we understand a function given by the formula $\varphi \mapsto \psi \circ \varphi \circ \psi^{-1}$. Also for any positive integers $n,l$ and any set $S \subseteq \mathbb{Z}_n$ by $l S$ we understand the set $\{l s \mid s\in S\}$. To make notation easier, for any colored directed circulant $\Gamma = \text{DiCay}(\mathbb{Z}_m,S_1\ldots S_k) $, we denote digraph $\text{DiCay}(\mathbb{Z}_m,lS_1\ldots lS_k) $ by $\Gamma^{(l)}$.

    Next proposition and its proof is just \cite[Lemma 2.2]{MorrisOddAbelianGroups} adapted to the language of graph chains.

\begin{prop} \label{Prop: { sigma_{li} } is in Aut_Ch(Gamma^{(l)}) }
    Let ${\{\sigma_i\}}_{i \in \mathbb{Z}} \in \text{Aut}^{\text{Ch}}( \text{DiCay}(\mathbb{Z}_m,S_1\ldots S_k) )$ and let $l \perp m$. Then ${\{\sigma_{l i}\}}_{i \in \mathbb{Z}} \in \text{Aut}^{\text{Ch}}( \text{DiCay}(\mathbb{Z}_m,l S_1\ldots l S_k) )$.
\end{prop}
\begin{proof}
    Let $l = p_1 \cdot \ldots \cdot p_r$ where $p_i$-s are all prime numbers. It is enough to prove that the thesis holds for each prime number $p$ and then our thesis easily follows by applying result for primes for each $p_i$-s in order. 

    We will now show that for any prime number $p$, ${\{\sigma_{p i}\}}_{i \in \mathbb{Z}} \in \text{Aut}^{\text{Ch}}( \text{DiCay}(\mathbb{Z}_m,p S_1\ldots p S_k) )$. \textit{A proper path between $x$ and $y$ in color $j$} is a sequence ${\{x_r\}}_{0 \leq r \leq p} \subset \mathbb{Z}_m$ such that $x_0=x$, $x_p=y$ and for all $0 \leq r \leq p-1$ we have $x_{r+1} - x_r \in S_j$. By $\mathcal{P}_{x,y}^j$ we denote the set of all proper paths between $x$ and $y$ in color $j$.
    
    From now on we will use $\Gamma = \text{DiCay}(\mathbb{Z}_m, S_1\ldots S_k)$ and $\Gamma^{(p)} = \text{DiCay}(\mathbb{Z}_m,p S_1\ldots p S_k)$. Take any $i \in \mathbb{Z}$. If ${\{x_r\}}_{0 \leq r \leq p}$ is a proper sequence in color $j$ then the sequence ${\{(x_r,pi+r)\}}_{0 \leq r \leq p}$ of vertices of $\text{Ch}(\Gamma)$ forms a path made of edges in color $j$. Since $\widetilde{\sigma}$ is an automorphism of $\text{Ch}(\Gamma)$, ${\{(\sigma_{pi+r}(x_r),pi+r)\}}_{0 \leq r \leq p}$ also forms such a path, and hence ${\{ \sigma_{i+r}(x_r) \}}_{0 \leq r \leq p}$ is a proper sequence in color $j$. This transformation gives a bijection from $\mathcal{P}_{x,y}^j$ to $\mathcal{P}_{ \sigma_{pi}(x),\sigma_{p(i+1)}(y) }^j$, so for any $i \in \mathbb{Z}$ we obtain equality $$\# \mathcal{P}_{x,y}^j = \# \mathcal{P}_{ \sigma_{pi}(x),\sigma_{p(i+1)}(y) }^j.$$

    Now let us determine $\# \mathcal{P}_{x,y}^j \text{ mod } p$ depending on $x$ and $y$. To achieve this let first define $\mathcal{S}_{d}^j = \{ {\{d_r\}}_{1 \leq r \leq p} \mid d_r \in S_j \text{ for all } 1 \leq r \leq p \text{ and } \sum_{r=1}^{p} d_r = d \}$. Note that each element ${\{x_r\}}_{0 \leq r \leq p} \in \mathcal{P}_{x,y}^j$ can be mapped to the sequence ${\{x_r - x_{r-1}\}}_{1 \leq r \leq p} \in \mathcal{S}_{y-x}^j$ therefore bijecting these sets and showing that $\# \mathcal{P}_{x,y}^j = \# \mathcal{S}_{y-x}^j$. Now consider the action of $\mathbb{Z}_p$ on $\mathcal{S}_{y-x}^j$ defined by $t.{\{d_r\}}_{1 \leq r \leq p} = {\{ d_{r+t \text{ mod }p} \}}_{1 \leq r \leq p}$. Now take any ${\{d_r\}}_{1 \leq r \leq p} \in \mathcal{S}_{y-x}^j$ and assume that there exists nonzero $t \in \mathbb{Z}$ such that $t.{\{d_r\}}_{1 \leq r \leq p} = {\{d_r\}}_{1 \leq r \leq p} $. This shows that for any $1 \leq r \leq p$ we have $d_r = d_{r+t \text{ mod }p}$, hence all $d_r$-s are equal to some $s \in S_j$, as $t \perp p$. Then $y-x = p s \in \mathbb{Z}_m$. Moreover note that if $p  s_1 = p s_2$ then $s_1 = s_2$ as elements of $\mathbb{Z}_m$ as $p \perp m$. Above reasoning shows that
    $$
    \# \mathcal{P}_{x,y}^j \text{ mod } p = \# \mathcal{S}_{y-x}^j \text{ mod }p = 
    \begin{cases} 
      1 & \text{ if } y-x = p s \text{ for some } s \in S_j, \\
      0 & \text{otherwise}.
   \end{cases}
    $$
    Now let consider a graph $\Gamma^{(p)}$. Take its vertices $(x,i)$ and $(y,i+1)$ which are connected by an edge in color $j$. This means that $y-x \in pS_j$, so $\# \mathcal{P}_{x,y}^j \text{ mod } p = \# \mathcal{S}_{y-x}^j \text{ mod }p = 1$. Then since $\# \mathcal{P}_{x,y}^j = \# \mathcal{P}_{ \sigma_{pi}(x),\sigma_{p(i+1)}(y) }^j$, we get $\# \mathcal{P}_{\sigma_{pi}(x),\sigma_{p(i+1)}(y)}^j \text{ mod } p = 1$ so $\sigma_{p(i+1)}(y) - \sigma_{pi}(x) \in p S_j$. This demonstrates that if $\sigma_{(p)} = {\{\sigma_{p i}\}}_{i \in \mathbb{Z}}$ then $\widetilde{\sigma}_{(p)}$ is an automorphism of $\text{Ch}(\Gamma^{(p)})$, so indeed ${\{\sigma_{p i}\}}_{i \in \mathbb{Z}} \in \text{Aut}_{\text{Ch}}(\Gamma^{(p)})$.
\end{proof}

\begin{lem} \label{Lem: Aut^Ch(Gamma) =  Aut^Ch(Gamma)^{(l)} and iota(pi_l) is an automorphism of it}
    For every reduced colored direcrted circulant $\Gamma = \text{DiCay}(\mathbb{Z}_m,S_1\ldots S_k)$ and every $l \perp m$, $\text{Aut}^{\text{Ch}}(\Gamma) = \text{Aut}^{\text{Ch}}( \Gamma^{(l)} )$ and $\iota(\psi_l) \in \text{Aut}(\text{Aut}^{\text{Ch}}(\Gamma))$.
\end{lem}
\begin{proof}
    Let us take any $\sigma \in \text{Aut}^{\text{Ch}}(\Gamma)$. Then there exists some ${\{\sigma_i\}}_{i \in \mathbb{Z}} \in \text{Aut}_{\text{Ch}}(\Gamma)$ such that $\sigma = \sigma_0$. Then by Proposition \ref{Prop: { sigma_{li} } is in Aut_Ch(Gamma^{(l)}) } we get that ${\{\sigma_{l i}\}}_{i \in \mathbb{Z}} \in \text{Aut}_{\text{Ch}}( \Gamma^{(l)} )$, so $\sigma = \sigma_0 \in \text{Aut}^{\text{Ch}}(\Gamma^{(l)})$ and hence $\text{Aut}^{\text{Ch}}(\Gamma) \subseteq \text{Aut}^{\text{Ch}}(\Gamma^{(l)})$. Moreover, the function $\psi_{l}: \mathbb{Z}_m \rightarrow \mathbb{Z}_m$ gives an isomorphism between $\Gamma$ and $\text{DiCay}(\mathbb{Z}_m,l S_1\ldots l S_k) )$, hence function $\widetilde{\iota(\psi_{l})}: \text{Aut}^{\text{Ch}}( \text{DiCay}(\Gamma) \rightarrow \text{Aut}_{\text{Ch}}( \text{DiCay}(\Gamma^{(l)} ) $ given by a formula ${\{\sigma_i\}}_{i \in \mathbb{Z}} \mapsto {\{ \iota(\psi_{l}).\sigma_i \}}_{i \in \mathbb{Z}}$ is an isomorphism. Since $\Gamma$ is reduced, we have 
    $$
    \text{Aut}^{\text{Ch}}( \Gamma ) \overset{\pi_{0,\Gamma}^{-1}}{\cong} \text{Aut}_{\text{Ch}}( \Gamma ) \overset{ \widetilde{\iota(\psi_{l})} }{\cong} \text{Aut}_{\text{Ch}}( \Gamma^{(l)} )\overset{ \pi_{0,\Gamma^{(l)} }}{\cong} \text{Aut}^{\text{Ch}}( \Gamma^{(l)} ), 
    $$
    hence $|\text{Aut}^{\text{Ch}}( \Gamma )| = |\text{Aut}^{\text{Ch}}( \Gamma^{(l)} )|$, so indeed $\text{Aut}^{\text{Ch}}( \Gamma ) = \text{Aut}^{\text{Ch}}( \Gamma^{(l)} )$. Also above sequence of isomoprhisms shows that $\pi_{0,\Gamma^{(l)}} \circ \widetilde{\iota(\psi_{l})} \circ \pi_{0,\Gamma}^{-1} = \iota(\psi_{l})$ is an isomorphism between $\text{Aut}^{\text{Ch}}( \Gamma )$ and $\text{Aut}^{\text{Ch}}( \Gamma^{(l)} ) = \text{Aut}^{\text{Ch}}( \Gamma )$, hence an automorphism of $\text{Aut}^{\text{Ch}}( \Gamma )$.
\end{proof}

\begin{thm} \label{THM: how multiplication by l change gamma in contect of Chains}
    For every reduced colored direcrted circulant $\Gamma = \text{DiCay}(\mathbb{Z}_m,S_1\ldots S_k)$ and every $l \perp m$, $\iota(\psi_l) \circ \gamma \circ {\iota(\psi_l)}^{-1} = \gamma^l$.
\end{thm}
\begin{proof}
    Let us take any $\tau \in \text{Aut}^{\text{Ch}}(\Gamma) = \text{Aut}^{\text{Ch}}(\Gamma^{(l)})$ and let ${\{\tau_i\}}_{i \in \mathbb{Z}} \in \text{Aut}_{\text{Ch}}(\Gamma^{(l)})$ be such that $\tau = \tau_0$. By Lemma \ref{Lem: Aut^Ch(Gamma) =  Aut^Ch(Gamma)^{(l)} and iota(pi_l) is an automorphism of it} there exists ${\{\sigma_i\}}_{i \in \mathbb{Z}} \in \text{Aut}_{\text{Ch}}(\Gamma)$ such that $\tau = \sigma_0$. Proposition \ref{Prop: { sigma_{li} } is in Aut_Ch(Gamma^{(l)}) } tells us that ${\{\sigma_{li}\}}_{i \in \mathbb{Z}} \in \text{Aut}_{\text{Ch}}(\Gamma^{(l)})$ hence by Observation \ref{OBS: Chain automorphisms are determined by action on V x {0} for reduced graphs} we get ${\{\tau_i\}}_{i \in \mathbb{Z}} = {\{\sigma_{li}\}}_{i \in \mathbb{Z}}$.
    
    On the other hand, since $\psi_l$ is an isomorphism between $\Gamma$ and $\Gamma^{(l)}$, function $\widetilde{\iota(\psi_{l})}: \text{Aut}_{\text{Ch}}(\Gamma) \rightarrow \text{Aut}_{\text{Ch}}(\Gamma^{(l)})$, defined during the proof of Lemma \ref{Lem: Aut^Ch(Gamma) =  Aut^Ch(Gamma)^{(l)} and iota(pi_l) is an automorphism of it}, is an isomorphism. This enables us to say that ${ \{ {\iota(\psi_l)}^{-1}.\tau_i \} }_{i \in \mathbb{Z}} \in \text{Aut}_{\text{Ch}}(\Gamma)$. Combining all acquired information leads us to conclude that for any $\tau \in \text{Aut}^{\text{Ch}}(\Gamma)$ we get
    $$
    \gamma_{\Gamma}^{l}(\tau) = \gamma_{\Gamma}^{l}(\sigma_{0}) = \sigma_{l} = \tau_1 =
    \left( {\iota(\psi_l)} \circ {\iota(\psi_l)}^{-1} \right).\tau_1 =
    {\iota(\psi_l)}.{\left( 
    {\iota(\psi_l)}^{-1}.\tau_1
    \right)}
    $$
    $$
    = {\iota(\psi_l)}.{ \left(
    \gamma_{\Gamma}
    \left(
    {\iota(\psi_l)}^{-1}.\tau_0 
    \right) 
    \right) } 
    =
    \left({\iota(\psi_l)} \circ \gamma_{\Gamma} \circ {\iota(\psi_l)}^{-1} \right)(\tau_0) = \left({\iota(\psi_l)} \circ \gamma_{\Gamma} \circ {\iota(\psi_l)}^{-1} \right)(\tau).
    $$
    These calculations can be visualized by a commuting diagram:
\[\begin{tikzcd}
	{\text{Aut}^{\text{Ch}}(\Gamma)} & {\text{Aut}^{\text{Ch}}(\Gamma^{(l)})} & {\text{Aut}^{\text{Ch}}(\Gamma)} \\
	{\text{Aut}^{\text{Ch}}(\Gamma)} & {\text{Aut}^{\text{Ch}}(\Gamma^{(l)})} & {\text{Aut}^{\text{Ch}}(\Gamma)}
	\arrow["{\iota(\psi_l)}", from=1-1, to=1-2]
	\arrow["{\gamma_{\Gamma}}", from=1-1, to=2-1]
	\arrow["id", from=1-2, to=1-3]
	\arrow["{\gamma_{\Gamma^{(l)}}}", from=1-2, to=2-2]
	\arrow["{\gamma_{\Gamma}^{l}}", from=1-3, to=2-3]
	\arrow["{\iota(\psi_l)}", from=2-1, to=2-2]
	\arrow["id", from=2-2, to=2-3]
\end{tikzcd}\]
    Above calculations show that $\iota(\psi_l) \circ \gamma \circ {\iota(\psi_l)}^{-1} = \gamma^l$ for any colored directed circulant.
\end{proof}

\begin{cor} \label{COR: order of gamma for directed colored circulants}
    If $\text{DiCay}(\mathbb{Z}_m,S_1\ldots S_k)$ is a reduced digraph, then $\gamma^m = id$.
\end{cor}
\begin{proof}
    It is enough to put $l = m+1$ in Theorem \ref{THM: how multiplication by l change gamma in contect of Chains} to get $\gamma = \iota(\psi_{1}) \circ \gamma \circ {\iota(\psi_{1})}^{-1} = \iota(\psi_{m+1}) \circ \gamma \circ {\iota(\psi_{m+1})}^{-1} = \gamma^{m+1}$. Since automorphisms of the group $\text{Aut}^{\text{Ch}}(\text{DiCay}(\mathbb{Z}_m,S_1\ldots S_k))$ form a group, multiplying by $\gamma^{-1}$ gives us desired conclusion.
\end{proof}

Bellow we demonstrate how above methods reprove a known classification of unstable circulants of odd order (cf. \cite{FERNANDEZ202249} and \cite{MorrisOddAbelianGroups}).

\begin{cor}
    Let $m$ be an odd integer. If $\Gamma = \text{DiCay}(\mathbb{Z}_m,S_1\ldots S_k)$ is not directed, connected and reduced then $\Gamma$ is stable.
\end{cor}
\begin{proof}
    Since $\Gamma$ is connected and is a circulant on odd number ov verticies, it is also nonbipartite, hence $\Gamma \times K_2$ is a connected and bipartite graph. Since it also is reduced, it is unstable if and only if $\gamma \neq id$.

    By Corollary \ref{COR: order of gamma for directed colored circulants} $\gamma^m = id$. By Proposition \ref{prop: equivalence of both definitions of gamma} and Observation \ref{Obs: basic properties of gamma} we also get $\gamma^2 = id$. Since $m \perp 2$, combining these statements yields $\gamma = id$, hence $\Gamma$ indeed is stable.
\end{proof}

\section{Replacement Property} \label{SECTION: Replacement Property}

Before we define what replacement property is, we have to first define what is \textit{the kernel of the partition $\mathcal{P}$ in a group $G$}. 

\begin{defi}
    Let $X$ be a set and let $G \leq \text{Sym}(X)$ be a group. If $\mathcal{P}$ is a partition of $X$, by $G_{\mathcal{P}}$ we understand the subgroup $\{ g \in G \mid \forall \mathfrak{p} \in \mathcal{P} \text{ } g[\mathfrak{p}]=\mathfrak{p} \}$. We will refer to this subgroup as \textit{the kernel of the partition $\mathcal{P}$ in a group $G$}.
\end{defi}

When it comes to automorphism group of the graph $\Gamma$, we usually denote the kernel of the partition $\mathcal{P}$ by $\text{Aut}_{\mathcal{P}}(\Gamma)$. Now we will define a radical of a Cayley graph and the quotient of the graph by its radical. This definition will become usefull in the proof of Theorem \ref{Thm: if n is squarefree, all pairs (Gamma,H) satisfy replacement property}.

\begin{defi}
    Let $\Gamma = \text{Cay}(G,S_1,\ldots,S_k)$ be a colored Cayley digraph over abelian group $G$. Then we define \textit{the radical of $\Gamma$} as $\bigcap\limits_{i = 1}^{k} \text{rad}(S_i)$ and denote it $\text{rad}(\Gamma)$. We also define $\bigslant{\Gamma}{\text{rad}(\Gamma)}$ to be the colored Cayley digraph $\text{Cay}\left( \bigslant{G}{\text{rad}(\Gamma)},\bigslant{S_1}{\text{rad}(\Gamma)},\ldots,\bigslant{S_k}{\text{rad}(\Gamma)} \right)$.
\end{defi}

\begin{defi} \label{defi: replacement property of a pair (Gamma,H)}
        Let $\Gamma = \text{Cay}(G,S)$ be a Cayley graph and let $H \leq G$. Now define a partition $\mathcal{H} = \{ Hg \mid g \in G\}$ of the set $G$ of verticies of $\Gamma$. We say that a pair $(\Gamma,H)$ satisfy \textit{replacement property} when 
        $$
        \exists f: \mathcal{H} \rightarrow G \quad \forall \sigma \in \text{Aut}_{\mathcal{H}}(\Gamma) \quad \exists \widetilde{\sigma} \in \text{Aut}_{\mathcal{H}}(\Gamma) \quad \left( \text{ }
        \forall \mathfrak{h} \in \mathcal{H} \quad \widetilde{\sigma} {|}_{\mathfrak{h}} = {f(\mathfrak{h})}_r \circ \sigma {|}_{H} \circ {f(\mathfrak{h})}_r^{-1} 
        \text{ } \right)
        $$
        Function $f$, existence of which is claimed, additionally needs to satisfy $f(H)=e_G$.
\end{defi}

    We now defined all concepts and gathered all tools needed to prove Theorem \ref{Thm: if n is squarefree, all pairs (Gamma,H) satisfy replacement property}. \\

\textit{Proof of Theorem \ref{Thm: if n is squarefree, all pairs (Gamma,H) satisfy replacement property}:} \\
Put $|H| = \ell_1$ and $\ell_2 = \frac{n}{|H|}$. By our assumption $\ell_1 \perp \ell_2$ and by Chinese remainder theorem we have an isomorphism $\mathbb{Z}_{\ell_1} \times \mathbb{Z}_{\ell_2} \cong \mathbb{Z}_n$ which is given by $\varphi: (x,y) \mapsto x \ell_2 + y \ell_1$. From now on we will use this isomorphism frequently. Subgroup $H$ is mapped by it to $\mathbb{Z}_{\ell_1} \times \{0\}$. We will show that if one defines $f: \mathcal{H} \rightarrow \mathbb{Z}_n$ by the formula $\mathbb{Z}_{\ell_1} \times \{y\} \mapsto (0,y)$ then it works. Formally speaking one should conjugate $f$ by isomorphism $\varphi$ to get the function required in Definition \ref{defi: replacement property of a pair (Gamma,H)}, however as mentioned above, from now one we will interchange $\mathbb{Z}_n$ with $\mathbb{Z}_{\ell_1} \times \mathbb{Z}_{\ell_2}$ freely.

    Take arbitrary $\sigma \in \text{Aut}_{\mathcal{H}}(\Gamma)$. Define $\widetilde{\sigma}: \mathbb{Z}_n \rightarrow \mathbb{Z}_n$ by the formula $(x,y) \mapsto 
\sigma.(x,0) + (0,y)$. Function $\widetilde{\sigma}$ defined as above satisfies
        $$
        \forall \mathfrak{h} \in \mathcal{H} \quad \widetilde{\sigma} {|}_{\mathfrak{h}} = {f(\mathfrak{h})}_r \circ \sigma {|}_{H} \circ {f(\mathfrak{h})}_r^{-1}.
        $$
    To end the proof we have to show that $\widetilde{\sigma} \in \text{Aut}_{\mathcal{H}}(\Gamma)$. It is obvious that $\widetilde{\sigma}$ fixes elements of $\mathcal{H}$ setwise, hence we only have to show that $\widetilde{\sigma} \in \text{Aut}(\Gamma)$. Let $T = \varphi^{-1}[S]$. Then $\varphi$ gives an isomorphism $\text{Cay}(\mathbb{Z}_{\ell_1} \times \mathbb{Z}_{\ell_2}, T) \cong \Gamma$. Define $T_i = T \cap (\mathbb{Z}_{\ell_1} \times \{i\})$. We will show that $\widetilde{\sigma}$ is an automorphism of ${\Gamma}_i = \text{Cay}(\mathbb{Z}_{\ell_1} \times \mathbb{Z}_{\ell_2}, T_i \cup T_{(-i)})$ for each $i \in [\ell_2]$.

    Observe that since $\sigma \in \text{Aut}_{\mathcal{H}}(\Gamma)$, it is an automorphism of ${\widetilde{\Gamma}}_i$ for each $i \in \mathbb{Z}_{\ell_2}$. If $i = 0$, then $\Gamma_0$ consists of $\ell_2$ copies of the graph $\text{Cay}(\mathbb{Z}_{\ell_1} \times \{0\},T_0)$. Since $\sigma {|}_{\mathbb{Z}_{\ell_1} \times \{0\}} = \widetilde{\sigma} {|}_{\mathbb{Z}_{\ell_1} \times \{0\}}$, we know that $\widetilde{\sigma} {|}_{\mathbb{Z}_{\ell_1} \times \{0\}}$ gives an automorphism of $\text{Cay}(\mathbb{Z}_{\ell_1} \times \{0\},T_0)$. Since $\widetilde{\sigma}$ acts in the same way on each coset of $\mathbb{Z}_{\ell_1} \times \{0\}$, we deduce that it is an automorphism of $\Gamma_0$.

    Now choose arbitrary nonzero $i \in \mathbb{Z}_{\ell_2}$. Define $\widetilde{T}_i = T_i + (0,-i) \subseteq \mathbb{Z}_{\ell_1} \times \{0\}$ and $\widetilde{\Gamma}_i = \text{DiCay}(\mathbb{Z}_{\ell_1},\widetilde{T}_i)$. We will now define a function from vertices of the infinite digraph $\text{Ch}(\widetilde{\Gamma}_i) = \text{DiCay}(\mathbb{Z}_{\ell_1} \times \mathbb{Z}, \widetilde{T}_i \times \{1\})$ to vertices of $\Gamma_i$. Let $\eta: \mathbb{Z}_{\ell_1} \times \mathbb{Z} \rightarrow \mathbb{Z}_{\ell_1} \times \mathbb{Z}_{\ell_2}$ by the function given by the formula $(x,j) \mapsto (x,j \cdot i)$. Notice that for any $j\in \mathbb{Z}$ and any $x,y \in \mathbb{Z}_{\ell_1}$,  $((x,j),(y,j+1))$ is a directed edge in $\text{Ch}(\widetilde{\Gamma}_i)$ if and only if $(\eta.(x,j),\eta.(y,j+1))$ is an edge in $\Gamma_i$.

    We will create a certain chain automorphism of $\widetilde{\Gamma}_i$ based on $\sigma$. Let us start by putting $\pi_{\mathbb{Z}_{\ell_1}}:\mathbb{Z}_{\ell_1} \times \mathbb{Z}_{\ell_2} \rightarrow \mathbb{Z}_{\ell_1}$ to be a function defined by $(x,y) \mapsto x$. For any $j \in \mathbb{Z}$ let function $\sigma_j: \mathbb{Z}_{\ell_1} \rightarrow \mathbb{Z}_{\ell_1}$ be given by the formula $x \mapsto \pi_{\mathbb{Z}_{\ell_1}}(\sigma.(x,j\cdot i\text{ }(\text{mod } \ell_2) ))$. We claim that ${\{ \sigma_{j} \}}_{j \in \mathbb{Z}}$ belongs to $\text{Aut}_{\text{Ch}}(\widetilde{\Gamma}_i)$, which can be visualized on the bellow diagram.
\[\begin{tikzcd}
	{\text{Ch}(\widetilde{\Gamma}_i)} & {(x,j)} && {\text{Aut}_{\text{Ch}}(\widetilde{\Gamma}_i)} & {{\{ \sigma_{j} \}}_{j \in \mathbb{Z}}} \\
	{\Gamma_i} & {(x,j\cdot i)} && {\text{Aut}(\Gamma_i)} & \sigma
	\arrow["\eta", from=1-1, to=2-1]
	\arrow[maps to, from=1-2, to=2-2]
	\arrow[dashed, from=2-4, to=1-4]
	\arrow[dashed, maps to, from=2-5, to=1-5]
\end{tikzcd}\]

\begin{lem} \label{lem: help in the proof of Theorem 1.1 -- Lemma 1}
    Sequence ${\{ \sigma_{j} \}}_{j \in \mathbb{Z}}$ belongs to the group $\text{Aut}_{\text{Ch}}(\widetilde{\Gamma}_i)$.
\end{lem}
\textit{Proof of the lemma.}
Let $j \in \mathbb{Z}$ be arbitrary and let $\widetilde{\text{B}}_j\widetilde{\Gamma}_i$ be the directed subgraph of $\text{Ch}(\widetilde{\Gamma}_i)$ induced by the subset of vertices $\mathbb{Z}_{\ell_1} \times \{j,j+1\}$. To prove the lemma it is sufficient to check that for every $j \in \mathbb{Z}$ function $\tau_j: \mathbb{Z}_{\ell_1} \times \{j,j+1\} \rightarrow \mathbb{Z}_{\ell_1} \times \{j,j+1\}$ given by formulas
$$
(x,j) \mapsto (\sigma_{j}.x,j ) \quad \text{ } \quad (x,j+1) \mapsto (\sigma_{j+1}.x,j+1)
$$
is an automorphism of $\widetilde{\text{B}}_j\widetilde{\Gamma}_i$. 

By $\text{B}_j\widetilde{\Gamma}_i$ we understand a graph with vertex set $\mathbb{Z}_{\ell_1} \times \{j,j+1\}$ and an edge between $(x,j)$ and $(y,j+1)$ if and only if there was a directed edge from $(x,j)$ to $(y,j+1)$ in $\widetilde{\text{B}}_j\widetilde{\Gamma}_i$. Since all directed edges in $\widetilde{\text{B}}_j\widetilde{\Gamma}_i$ were going from $\mathbb{Z}_{\ell_1} \times \{j\}$ to $\mathbb{Z}_{\ell_1} \times \{j+1\}$ and $\tau_j$ fixes both of these setwise, we only have to show that $\tau_j$ is an automorphism of $\text{B}_j\widetilde{\Gamma}_i$.

Define $\Gamma_{i,j}$ to be the subgraph of $\Gamma_i$ induced by subset of vertices $V_j$ given by the formula
$$
V_j = \mathbb{Z}_{\ell_1} \times \{j\cdot i,(j+1)\cdot i\}.
$$

To show that indeed $\tau_j$ is an automorphism of $\text{B}_j\widetilde{\Gamma}_i$ let us first notice that function 
$$
 \eta_j = \eta {|}_{\mathbb{Z}_{\ell_1}  \times \{j,j+1\}}: \mathbb{Z}_{\ell_1}  \times \{j,j+1\} \rightarrow V_j
$$
gives an isomorphism of graphs $\text{B}_j\widetilde{\Gamma}_i \overset{\eta_j}{\cong} \Gamma_{i,j}$. 
\[\begin{tikzcd}
	{\text{B}_j\widetilde{\Gamma}_i} & {\text{B}_j\widetilde{\Gamma}_i} \\
	{\Gamma_{i,j}} & {\Gamma_{i,j}}
	\arrow["{\tau_j}", dashed, from=1-1, to=1-2]
	\arrow["{\eta_j}", from=1-1, to=2-1]
	\arrow["{\eta_j}", from=1-2, to=2-2]
	\arrow["{\sigma { | }_{V_j}}", from=2-1, to=2-2]
\end{tikzcd}\]
Since $\sigma$ was an automorphism of $\Gamma_i$, $\sigma {|}_{V_j}$ is an automoprhism of $\Gamma_{i,j}$. Commutativity of the above diagram shows that $\tau_j$ indeed is an automorphism of $\text{B}_j\widetilde{\Gamma}_i$.
\hfill \qedsymbol{}
\newline

Note that two vertices $x$ and $y$ of $\widetilde{\Gamma}_i = \text{DiCay}(\mathbb{Z}_{\ell_1},\widetilde{T}_i)$ have the same in-neighbourhoods if and only if $x-y \in \text{rad}(\widetilde{T}_i)$ and similarly have the same in-neighbourhoods if and only if $x-y \in \text{rad}(\widetilde{T}_i)$. Therefore for any $j \in \mathbb{Z}$, $\sigma_{j}$ permutes cosets of $\text{rad}(\widetilde{\Gamma}_i)= \text{rad}(\widetilde{T}_i)$. Let $\mathcal{R}$ be the partition of $\mathbb{Z}_{\ell_1}$ into cosets of $\text{rad}(\widetilde{\Gamma}_i)$. Define $\nu_j = \text{ind}_{\mathcal{R}}(\sigma_{j})$ for any $j \in \mathbb{Z}$. Then ${\{\nu_j\}}_{j \in \mathbb{Z}} \in \text{Aut}_{\text{Ch}}\left( \bigslant{\widetilde{\Gamma}_i}{\text{rad}(\widetilde{\Gamma}_i)} \right)$. Note that by definition digraph $\bigslant{\widetilde{\Gamma}_i}{\text{rad}(\widetilde{\Gamma}_i)}$ is reduced. Since $\sigma_{0} = \sigma_{\ell_2}$, we get $\nu_{0} = \nu_{\ell_2}$ and finally $\gamma^{\ell_2}(\nu_0) = \nu_0$, where $\gamma$ is the automorphism of $\text{Aut}^{\text{Ch}}\left( \bigslant{\widetilde{\Gamma}_i}{\text{rad}(\widetilde{\Gamma}_i)} \right)$ described in Definition \ref{Defi: widetilde{gamma} and gamma for colored directed reduced graphs}. On the other hand, by Corollary \ref{COR: order of gamma for directed colored circulants} $\gamma^{|\mathcal{R}|}(\nu_0) = \nu_0$. Note that $|\mathcal{R}|$ is a divisor of $\ell_1$ which is coprime to $\ell_2$, hence there exist such integers $a,b \in \mathbb{Z}$ that $a |\mathcal{R}| + b \ell_2 = 1$. This leads us to the fact that
$$
\gamma(\nu_0) = \gamma^{a |\mathcal{R}| + b \ell_2}(\nu_0) = \gamma^{a |\mathcal{R}|}(\nu_0) = \nu_0.
$$
Above equality shows that for each $j \in \mathbb{Z}$ we have $\nu_j = \nu_0$. Since $\nu_0 = \text{ind}_{\mathcal{R}}(\sigma_{0})$, we get 
$$
{\{ \sigma_{0} \}}_{j \in \mathbb{Z}} \in \text{Aut}_{\text{Ch}}(\widetilde{\Gamma}_i).
$$
To end the proof of the whole theorem we only need to show the following lemma.

\begin{lem} \label{lem: help in the proof of Theorem 1.1 -- Lemma 2}
    Permutation $\widetilde{\sigma}$ is an automorphism of $\Gamma_i$.
\end{lem}
\textit{Proof of the lemma.} We will mainly reverse the proces from the proof of Lemma \ref{lem: help in the proof of Theorem 1.1 -- Lemma 1}. Take arbitrary $j \in \mathbb{Z}$ and let $\widetilde{\tau}_j: \mathbb{Z}_{\ell_1} \times \{j,j+1\} \rightarrow \mathbb{Z}_{\ell_1} \times \{j,j+1\}$ be given by formulas
$$
(x,j) \mapsto (\sigma_{0}.x,j ) \quad \text{ } \quad (x,j+1) \mapsto (\sigma_{0}.x,j+1).
$$
Define graphs $\text{B}_j\widetilde{\Gamma}_i$ and $\Gamma_{i,j}$ as in the proof of Lemma \ref{lem: help in the proof of Theorem 1.1 -- Lemma 1} and do the same for $\eta_j$ and $V_j$. Similarly to the situation before, following diagram commutes.
\[\begin{tikzcd}
	{\text{B}\widetilde{\Gamma}_i} & {\text{B}\widetilde{\Gamma}_i} \\
	{\Gamma_{i,j}} & {\Gamma_{i,j}}
	\arrow["{\widetilde{\tau}_j}", from=1-1, to=1-2]
	\arrow["{\eta_j}", from=1-1, to=2-1]
	\arrow["{\eta_j}", from=1-2, to=2-2]
	\arrow["{\widetilde{\sigma} { | }_{V_j}}", dashed, from=2-1, to=2-2]
\end{tikzcd}\]

From the fact that ${\{ \sigma_{0} \}}_{j \in \mathbb{Z}} \in \text{Aut}_{\text{Ch}}(\widetilde{\Gamma}_i)$ it follows that $\widetilde{\tau}_j \in \text{Aut}(\text{B}\widetilde{\Gamma}_i)$ and we finally get $\widetilde{\sigma} { | }_{V_j} \in \text{Aut}(\Gamma_{i,j})$. Now put $K_i = \mathbb{Z}_{\ell_1} \times i\mathbb{Z}_{\ell_2}$. We will refer to the subgraph of $\Gamma_i$ induced on vertices from the set $K_i$ as $\Gamma_{i,K_i}$. Since $j \in \mathbb{Z}$ was arbitrary, we obtain the information that $\widetilde{\sigma} {|}_{K_i}$ is an automorphism of $\Gamma_{i,K_i}$. Note that $T_i \cup T_{(-i)} \subseteq K_i$, hence $\Gamma_{i,K_i}$ is the union of connected components of $\Gamma_i$. Since $\widetilde{\sigma}$ acts uniformly on cosets of $K_i$, we get $\widetilde{\sigma} \in \text{Aut}(\Gamma_i)$.
\hfill \qedsymbol{}
\newline

Since $i \in \mathbb{Z}_{\ell_2}$ was an arbitrary nonzero element, above argument shows that indeed $\widetilde{\sigma}$ maps all edges of $\Gamma$ onto edges of $\Gamma$, hence $\widetilde{\sigma} \in \text{Aut}_{\mathcal{H}}(\Gamma)$ as wanted.
\hfill \qedsymbol{}





\section{Group theoretical results} \label{SECTION: Group theoretical results}

\subsection{Proof of Theorem \ref{Thm characterization of primitive group actions with regular cyclic subgroup and additional assumptions}} \hfill \\

Before we prove Theorem \ref{Thm characterization of primitive group actions with regular cyclic subgroup and additional assumptions} we need to state some definitions and cite some results.

\begin{defi}
    Let $G$ be a group acting transitively on $X$. If only block systems (cf. Definition \ref{defi: invariant partition block system}) of this action are partition of $X$ into singletons and partition containing one element, then action of $G$ on $X$ is called \textit{primitive}. Group action of $G$ on $X$ is called \textit{regular} if for each pair $x,y \in X$ there exists unique $g \in G$ such that $g.x = y$.
\end{defi}

\begin{defi}
    we say that group action of $G$ on $X$ is isomorphic to the group action of $H$ on $Y$ if there exists an isomorphism $\varphi: G \rightarrow H$ and a bijection $f: X \rightarrow Y$ such that for any $g \in G$ following diagram commutes.
\[\begin{tikzcd}
	X & X \\
	Y & Y
	\arrow["{g.}", from=1-1, to=1-2]
	\arrow["f", from=1-1, to=2-1]
	\arrow["f", from=1-2, to=2-2]
	\arrow["{\varphi(g).}", from=2-1, to=2-2]
\end{tikzcd}\]
\end{defi}


Before we cite next important theorem, let us recall notation for certain groups. For the start, for any given $k$ the dihedral group is defined by $D_{2k} = \langle x,y \mid x^n = y^2 = xyxy =e \rangle$. $\text{Aff}(\mathbb{F}_p)$ is the set of affine functions from $\mathbb{F}_p$ onto itself, that is functions given by formulas $x \mapsto ax + b$ for some $a\neq0, b \in \mathbb{F}_p$. $A_k$ and $S_k$ denote alternating and symmetric groups on $k$ elements respectively. For given $d>1$ and $q$ which is a power of a prime $p$, by $\Gamma\text{L}_{d}(\mathbb{F}_{q})$ we understand the semi-simple product $\text{GL}_{d}(\mathbb{F}_{q}) \rtimes_{\theta} \text{Gal}\left( {\mathbb{F}_q} \text{/} {\mathbb{F}_p} \right)$, where for any element $\phi \in \text{Gal}\left( {\mathbb{F}_q} \text{/} {\mathbb{F}_p} \right)$, $\theta(\phi)$ is an automorphism of $\text{GL}_{d}(\mathbb{F}_{q})$ which sends any matrix $M = {\{m_{i,j}\}}_{i,j \in [d]}$ onto ${\{ \phi(m_{i,j}) \}}_{i,j \in [d]}$. By $\text{PSL}_{d}(\mathbb{F}_{q})$ we denote the quotient of $\text{SL}_{d}(\mathbb{F}_{q})$ by subgroup made from multiplicities of identity matrix. Similarly we define $\text{PGL}_{d}(\mathbb{F}_{q})$ and $\text{P}\Gamma\text{L}_{d}(\mathbb{F}_{q})$ as quotients of $\text{GL}_{d}(\mathbb{F}_{q})$ and $\Gamma\text{L}_{d}(\mathbb{F}_{q})$ respectively. Finally, by $M_{11}$ and $M_{23}$ we denote Mathieu groups with corresponding indexes.

\begin{thm} \label{thm: characterization of prim perm groups (citation)}
(\cite[Corollary 1.2]{FinPrimPermGroups}) Let $k$ be any positive integer and let  $X=\mathbb{Z}_k$. If $G \leq \text{Sym}(X)$ acts primitively on $X$ and ${\left(\mathbb{Z}_k\right)}_{r} \leq G$, then up to an isomorphism of group action one of the following holds:
    \begin{enumerate}[i.]
         \item $\mathbb{F}_p \leq G \leq \text{Aff}(\mathbb{F}_p)$, $X = \mathbb{F}_p$ where $k=p$ is a prime;
         \item $A_{k} \leq G \leq S_{k}$ with $k \geq 4$ and standard action of permutation groups on elements;
         \item $\text{PGL}_{d}(\mathbb{F}_{q}) \leq G \leq \text{P}\Gamma\text{L}_{d}(\mathbb{F}_{q})$, $X = \mathbb{P}^{d-1}\mathbb{F}_{q}$ where $d>1$ is an arbitrary integer, $q$ is a power of a prime and action of $G$ on $X$ is the action of projective group on lines;
         \item $(G,k) \in \{ (\text{PSL}_2(\mathbb{F}_{11}), 11), (M_{11}, 11), (M_{23}, 23) \}$.
    \end{enumerate}
\end{thm}


\begin{obs} \label{obs: no copy of D_22 in PSL(2,11)}
    Group $\text{PSL}_2(\mathbb{F}_{11})$ has order $660$ and does not contain any isomorphic copy of the group $D_{22}$.
\end{obs}
\begin{proof}
    The fact that $|\text{PSL}_2(\mathbb{F}_{11})|= 660$ can be checked in \cite{WebsitePSL211}, just as the fact that any maximal subgroup of $\text{PSL}_2(\mathbb{F}_{11})$ is of order $12$, $55$ or $60$. Since $D_{22}$ has order $22$, if there was a subgroup $E \leq \text{PSL}_2(\mathbb{F}_{11})$ isomorphic to this dihedral group, it would have to be contained in some maximal subgroup, hence its order should be divisible by $12$, $55$ or $60$. Clearly $22$ is not divisible by any of those which ends the proof. The fact that
\end{proof}


\begin{obs} \label{obs: no copy of D_22 in M11}
    Group $M_{11}$ has order $7920$ and does not contain any isomorphic copy of the group $D_{22}$.
\end{obs}
\begin{proof}
    The fact that $|M_{11}|= 7920$ can be checked in \cite{WebsiteM11}, just as the fact that each maximal subgroup of $M_{11}$ either have one of orders $48$, $120$, $144$, $720$ or is isomorphic to $\text{PSL}_2(\mathbb{F}_{11})$. If there was a subgroup $E \leq M_{11}$ isomorphic to $D_{22}$, it would be contained in a maximal subgroup and hence $22$ would divide its order. $22$ does not divide any of $48$, $120$, $144$, $720$, hence if there is an isomorphic copy of $D_{22}$ in $M_{11}$, so is one in $\text{PSL}_2(\mathbb{F}_{11})$, but by Lemma \ref{obs: no copy of D_22 in PSL(2,11)} the later does not hold.
\end{proof}


\begin{obs} \label{obs: no copy of D_46 in M23}
    Group $M_{23}$ does not contain any isomorphic copy of the group $D_{46}$.
\end{obs}
\begin{proof}
    By \cite{WebsiteM23} each maximal subgroup of $M_{23}$ have one of orders $443520$, $40320$, $20160$, $7920$, $5760$ or $253$. If there was a subgroup $E \leq M_{23}$ isomorphic to $D_{46}$, it would be contained in a maximal subgroup and hence $46$ would divide its order, but none of the numbers listed above is divisible by $46$.
\end{proof}


Now we will list a couple of lemmas which will be helpful during the proof of Theorem \ref{Thm characterization of primitive group actions with regular cyclic subgroup and additional assumptions}.

\begin{lem} \label{lem: Aut(Mi) = Mi for i 11 or 23}
(\cite[Chapter XII, Remark 1.15]{FinGroupsIII}) For $i \in \{11,23\}$ $\text{Aut}(M_i) \cong M_i$.
\end{lem}


\begin{lem} \label{lem: S < G < Aut(G) < Aut(S)}
    Let $S$ be a non-abelian simple group which is a socle of a group $G$. If moreover group $G$ satisfies $S \leq G \leq \text{Aut}(S)$, then $G \leq \text{Aut}(G) \leq \text{Aut}(S)$.
\end{lem}
\begin{proof}
    Let us define a function $\iota: S \rightarrow \text{Aut}(S)$ such that $\iota(s).x = sxs^{-1}$ for any $s,x \in S$. It is easy to notice that such a function is a monomorfism since $S$ is non-abelian. Define $\widetilde{S} = \text{im}(\iota) \leq \text{Aut}(S)$. We will consider $G$ as the subgroup of $\text{Aut}(S)$ such that $\widetilde{S} \leq G \leq \text{Aut}(S)$. We will show that for any $\psi \in \text{Aut}(G)$ there exists $\varphi \in \text{Aut}(S)$ such that for any $g \in G$ it holds that $\psi(g) = \varphi \circ g \circ \varphi^{-1}$.

    Since $\widetilde{S}$ is a socle of $G$, $\psi[\widetilde{S}] = \widetilde{S}$. This shows that $\psi {|}_{\widetilde{S}} \in \text{Aut}(\widetilde{S})$. Let $\varphi \in \text{Aut}(S)$ be such that the below diagram commutes.
\[\begin{tikzcd}
	S & S \\
	{\widetilde{S}} & {\widetilde{S}}
	\arrow["\varphi", dotted, from=1-1, to=1-2]
	\arrow["\iota", from=1-1, to=2-1]
	\arrow["\iota", from=1-2, to=2-2]
	\arrow["{\psi {|}_{\widetilde{S}}}", from=2-1, to=2-2]
\end{tikzcd}\]
Now let us define $\widetilde{\varphi}: \text{Aut}(S) \rightarrow \text{Aut}(S)$ by the formula $\nu \mapsto \varphi \circ \nu \circ \varphi^{-1}$. It is obvious by definition of $\varphi$ that $\widetilde{\varphi} {|}_{\widetilde{S}} = \psi {|}_{\widetilde{S}}$. We will now demonstrate that indeed $\widetilde{\varphi} {|}_{G} = \psi$. Take some $g\in G$. Let $\psi(g) = \widetilde{g} \circ \widetilde{\varphi}(g)$. Then for any $s \in S$ it holds that 
$$
\iota(\varphi \circ g(s)) = \widetilde{\varphi}(g\iota(s) g^{-1}) = \psi(g\iota(s) g^{-1}) = \psi(g) \circ \psi(\iota(s)) \circ \psi(g)^{-1} = 
$$
$$
= \widetilde{g} \circ \widetilde{\varphi}(g) \circ \widetilde{\varphi}(\iota(s))  \circ \widetilde{\varphi}(g)^{-1} \circ \widetilde{g}^{-1} = \widetilde{g} \circ \widetilde{\varphi}(g\iota(s) g^{-1}) \circ \widetilde{g}^{-1} =
$$
$$
= \widetilde{g} \circ \iota(\varphi \circ g(s)) \circ \widetilde{g}^{-1} = \iota(\widetilde{g}(\varphi \circ g(s))).
$$
Since $s \in S$ was arbitrary, on can put $\widetilde{s} = \varphi \circ g(s)$ and $\widetilde{s}$ also can be choose to be an arbitrary element of $S$. Since $\iota$ is a monomorphism we get that $\widetilde{g}(\widetilde{s}) = \widetilde{s}$ for any $\widetilde{s} \in S$, hence $\widetilde{g} = e_{\text{Aut}(S)}$, hence indeed $\psi(g) = \widetilde{\varphi}(g)$. Since $g \in G$ was arbitrary we get $\widetilde{\varphi} {|}_{G} = \psi$, which shows that homomorphism $\text{res}_{S}: \text{Aut}(G) \rightarrow \text{Aut}(S)$ given by formula $\psi \mapsto \psi {|}_{\widetilde{S}}$ indeed is a monomorphism, hence $\text{Aut}(G)$ can be understood as the subgroup of $\text{Aut}(S)$. 


To end the proof notice that since $G \leq \text{Aut}(S)$, for any nontrivial $g \in G$ function $\iota(g): G \mapsto G$ given by the formula $\nu \mapsto g \circ \nu \circ g^{-1}$ is not identity, hence $\text{Aut}(G)$ naturally contains a copy of $G$.
\end{proof}


\begin{lem} \label{lem: FinPrimPermGroups lemma about action of groups with socle PSL(d,q)}
    (\cite[Lemma 2.3]{FinPrimPermGroups}) Let $G$ be a group such that $\text{PSL}_{d}(\mathbb{F}_{q}) \leq G \leq \text{Aut}(\text{PSL}_{d}(\mathbb{F}_{q}))$ for some $d>1$ and $q$ being a prime power. If $G$ acts on the set $X$ of the size $(q^d -1) \text{/}(q-1)$ and $G$ contains a cyclic subgroup acting transitively on $X$, then $\text{PGL}_{d}(\mathbb{F}_{q}) \leq G \leq \text{P}\Gamma\text{L}_{d}(\mathbb{F}_{q})$.
\end{lem}

\begin{lem} \label{lem: from Dobson paper}
    (\cite[Lemma 19]{DOBSON200579}) Let $k = (q^d -1) \text{/}(q-1)$ with $d \geq 3$, write $q = p^r$ with $p$ prime, and let $k' = k \text{/} \text{gcd}(r,k)$. Let $\rho'$ be an element of order $k'$ in $\text{PGL}_{d}(\mathbb{F}_{q})$ that acts semi-regularly on $\mathbb{P}^{d-1}\mathbb{F}_q$. Then $\rho'$ is not conjugate to ${(\rho')}^{-1}$ in $\text{P}\Gamma\text{L}_{d}(\mathbb{F}_{q})$.
\end{lem}

Let us recall the situation presented in assumptions of Theorem \ref{Thm characterization of primitive group actions with regular cyclic subgroup and additional assumptions}. We consider $G$ to be the group of permutations of the set $X = \mathbb{Z}_k$ with odd number of elements. Moreover it contains a cyclic regular subgroup ${(\mathbb{Z}_k)}_r \leq G$ and its automorphism group contains $\iota(\mathfrak{i})$, where $\mathfrak{i}: \mathbb{Z}_k \rightarrow \mathbb{Z}_k$ is given by the formula $x \mapsto (-x)$. Note that $\langle {(\mathbb{Z}_k)}_r , \mathfrak{i} \rangle$ is the standard permutation presentation of the dihedral group $D_{2k}$. We are now ready for the proof of Theorem \ref{Thm characterization of primitive group actions with regular cyclic subgroup and additional assumptions}. \\

\textit{Proof of Theorem \ref{Thm characterization of primitive group actions with regular cyclic subgroup and additional assumptions}:}\\
Recall Theorem \ref{thm: characterization of prim perm groups (citation)}. Since action of $G$ on $X$ satisfies its assumptions, to prove our theorem we only have to eliminate option (iv) and restrict parameters $d$ and $q$ in option (iii) to $d=2$ and $q=2^{\ell}$ for some $\ell \geq 2$. Let us start by eliminating all possibilities form option (iv).\\

\textit{Case 1.1:} Assume that $G \cong \text{PSL}_2(\mathbb{F}_{11})$ and $k = 11$. Assume that $\mathfrak{i} \in G$. Then since $\langle {(\mathbb{Z}_11)}_r , \mathfrak{i} \rangle \cong D_{22}$, $G$ contains an isomorphic copy of $D_{22}$, which contradicts the conclusion of Observation \ref{obs: no copy of D_22 in PSL(2,11)}. This shows that $\mathfrak{i} \notin G$, so $G \neq \langle G, \mathfrak{i} \rangle$. Notice that $H = \langle G, \mathfrak{i} \rangle$ also acts primitively on the set $X$ and ${(\mathbb{Z}_{11})}_r \leq G \leq H$, hence Theorem \ref{thm: characterization of prim perm groups (citation)} applies, and to show a contradiction we will now eliminate each of the possibles i. - iv.

Before we do that notice that $\iota(\mathfrak{i}) \in \text{Aut}(G)$, hence $|H| = 2|G| = 2 \cdot 660 = 1320$ (cf. Observation \ref{obs: no copy of D_22 in PSL(2,11)}). Condition i. cannot hold because then $|H| \leq |\text{Aff}(\mathbb{F}_{11})| = 110 < 1320 = |H|$. Condition ii. cannot hold because $|H| = 1320 < 19 958 400 = |A_{11}| \leq |H|$. Condition iii. cannot hold because there does not exist such $d>1$ and $q$ being a prime power for which $|\mathbb{P}^{d-1}\mathbb{F}_{q}|=11$. We are now left with condition iv. Since $k=11$ and $H$ contains a subgroup of index $2$ isomorphic to $\text{PSL}_2(\mathbb{F}_{11})$, only other possibility is that $H \cong M_{11}$, however $|H| = 1320 \neq 7920 = |M_{11}|$ (cf. Observation \ref{obs: no copy of D_22 in M11}).

Above shows that indeed $G \cong \text{PSL}_2(\mathbb{F}_{11})$ and $k = 11$ cannot hold. \\

\textit{Case 1.2:} Assume that $G \cong M_{11}$ and $k = 11$. By Lemma \ref{lem: Aut(Mi) = Mi for i 11 or 23} $\text{Aut}(M_{11}) \cong M_{11}$. Since $D_{22} \cong \langle {(\mathbb{Z}_{11})}_r , \mathfrak{i} \rangle \leq \text{Aut}(G)$, it would mean that $M_{11}$ contains some isomorphic copy of $D_{22}$. Observation \ref{obs: no copy of D_22 in M11} shows that later of the above is false, hence $G \cong M_{11}$ and $k = 11$ cannot hold. \\

\textit{Case 1.3:} Assume that $G \cong M_{23}$ and $k = 23$. By Lemma \ref{lem: Aut(Mi) = Mi for i 11 or 23} $\text{Aut}(M_{23}) \cong M_{23}$. Since $D_{46} \cong \langle {(\mathbb{Z}_{23})}_r , \mathfrak{i} \rangle \leq \text{Aut}(G)$, it would mean that $M_{23}$ contains some isomorphic copy of $D_{46}$. Observation \ref{obs: no copy of D_46 in M23} shows that later of the above is false, hence $G \cong M_{23}$ and $k = 23$ cannot hold. \\

We now succeeded in showing that indeed iv. does not hold. Now we will show that in iii. one needs $d=2$ and later that $q$ is indeed a power of $2$. \\

\textit{Case 2.1:} Assume that up to an isomorphism of group actions $\text{PGL}_{d}(\mathbb{F}_{q}) \leq G \leq \text{P}\Gamma\text{L}_{d}(\mathbb{F}_{q})$ and $X = \mathbb{P}^{d-1}\mathbb{F}_{q}$. Moreover $d>2$, $q=p^r$ for some prime $p$ and action of $G$ on $X$ is the action of projective group on lines. 

In above case $S = \text{PSL}_{d}(\mathbb{F}_{q})$ is a socle of $G$ and it is a simple non-abelian group. Moreover $\text{PSL}_{d}(\mathbb{F}_{q}) \leq \text{PGL}_{d}(\mathbb{F}_{q})$ and $\text{P}\Gamma\text{L}_{d}(\mathbb{F}_{q}) \leq \text{Aut}(\text{PSL}_{d}(\mathbb{F}_{q}))$, hence assumptions of Lemma \ref{lem: S < G < Aut(G) < Aut(S)} are fulfilled. Applying this lemma shows that $G \leq \text{Aut}(G) \leq \text{Aut}(\text{PSL}_{d}(\mathbb{F}_{q}))$. Now define $H = \langle G,\mathfrak{i} \rangle$. Since $\mathfrak{i} \in \text{Aut}(G)$, $H \leq \text{Aut}(G)$. 

Notice that $H$ acts on the set $X = \mathbb{P}^{d-1}\mathbb{F}_{q}$ of size $(q^d -1) \text{/}(q-1)$, since $G$ contained a cyclic subgroup acting regularly on $X$, so does $H$ and by above reasoning $\text{PSL}_{d}(\mathbb{F}_{q}) \leq H \leq \text{Aut}(\text{PSL}_{d}(\mathbb{F}_{q}))$. Combining all of the above allows us to apply Lemma \ref{lem: FinPrimPermGroups lemma about action of groups with socle PSL(d,q)} which tells us that $H \leq \text{P}\Gamma\text{L}_{d}(\mathbb{F}_{q})$, hence $\mathfrak{i} \in \text{P}\Gamma\text{L}_{d}(\mathbb{F}_{q})$.

Before the isomorphism of group actions we had $X = \mathbb{Z}_k$ and ${(\mathbb{Z}_k)}_r \leq G$. Let $\rho: \mathbb{Z}_k \rightarrow \mathbb{Z}_k$ be given by the formula $x \mapsto x+1$. Put $\rho' = \rho^{r}$. Since $\rho \in G \leq \text{P}\Gamma\text{L}_{d}(\mathbb{F}_{q})$, $\text{P}\Gamma\text{L}_{d}(\mathbb{F}_{q}) = \text{PGL}_{d}(\mathbb{F}_{q}) \rtimes \text{Gal}\left( {\mathbb{F}_q} \text{/} {\mathbb{F}_p}\right)$ and $\text{Gal}\left( {\mathbb{F}_q} \text{/} {\mathbb{F}_p}\right) \cong \mathbb{Z}_r$, we conclude that $\rho' = \rho^{r} \in \text{PGL}_{d}(\mathbb{F}_{q})$. It is also easy to notice that since $\rho$ was an element of order $k$ which acts regularly on $X$, $\rho'$ is an element of order $k' = k \text{/} \text{gcd}(r,k)$ which acts semi-regularly on $X=\mathbb{P}^{d-1}\mathbb{F}_q$. 

Since assumptions of Lemma \ref{lem: from Dobson paper} are fulfiled, it implies that $\rho'$ is not conjugate to ${(\rho')}^{-1}$ in $\text{P}\Gamma\text{L}_{d}(\mathbb{F}_{q})$. However based on previously proved fact that $\mathfrak{i} \in \text{P}\Gamma\text{L}_{d}(\mathbb{F}_{q})$ and equality $\mathfrak{i} \text{ } \rho' \text{ }\mathfrak{i}^{-1} = (\rho')^{-1}$ we deduce that the contrary is true, which yields a contradiction, hence shows that considered case is impossible. \\

\textit{Case 2.2:} Assume that up to an isomorphism of group actions $\text{PGL}_{2}(\mathbb{F}_{q}) \leq G \leq \text{P}\Gamma\text{L}_{2}(\mathbb{F}_{q})$ and $X = \mathbb{P}^{1}\mathbb{F}_{q}$ for some prime power $q$. Now just compute $k = |X| = |\mathbb{P}^{1}\mathbb{F}_{q}| = q+1$ and notice that we assume $k$ to be odd, hence $q$ has to be even. Since $q$ is a power of a prime, it has to be a power of $2$. \\

To summarize: we eliminated the case iv. completely and reduced the range of possible parameters $d$ and $q$ in case iii. to $d=2$ and $q = 2^{\ell}$ for some $\ell \geq 1$. 

To end the proof we shall show that one can additionally restrict to $\ell \geq 2$ in the case iii. Notice however, that $\text{PGL}_{2}(\mathbb{F}_{2}) = \text{P}\Gamma\text{L}_{2}(\mathbb{F}_{2}) \cong S_3$ and $|\mathbb{P}^{1}\mathbb{F}_{2}|=3$, hence the case $\ell = 1$ is covered by the first family from Theorem \ref{thm: characterization of prim perm groups (citation)} with $k = p = 3$ and $G \cong \text{Aff}(\mathbb{F}_3) \cong S_3$.
\hfill \qedsymbol{}

\subsection{Cohomology of group modules} \hfill \\

In this section we will introduce concepts from cohomology theory of group modules. Objects introduced in this subsection will be the main tool in the proof of Theorem \ref{Thm: action of a nontrivial cocycle from group G acting primitively on X with additional assumptions}.

\begin{defi}
    Let $G$ be a group acting on the abelian group $M$. Then $M$ is called \textit{a group module}. If one wants to address role of the group $G$, they can say that \textit{$M$ is a $G$-module}. Let us also define the submodule of invariants as $M^G = \{m\in M \mid \forall g \in G \text{ } g.m = m \}$.
\end{defi}

In the next chapter we will mainly consider group modules of the following type.
\begin{defi} \label{defi: G-module A[X]}
    Let $G$ be a group acting on the set $X$ and let $\mathbb{A}$ by an abelian group. Then $\mathbb{A}[X]$ is a $G$-module made of a group $\bigoplus_{x \in X} \mathbb{A} \vec{e}_x$. Action of an element $g \in G$ is given on the basis ${\{\vec{e}_x\}}_{x \in X}$ by $\vec{e}_x \mapsto \vec{e}_{g.x}$ and extended in a unique way which forms an isomorphism of the group $\bigoplus_{x \in X} \mathbb{A} \vec{e}_x$.
\end{defi}

\begin{defi}
    Let $G$ be a group and let $M$ be a $G$-module. Then function $\omega: G \rightarrow M$ is called a \textit{co-cycle} if for every $g,h \in G$ it satisfies $\omega(gh) = g.\omega(h) + \omega(g)$. Function $\omega: G \rightarrow M$ is called a \textit{co-boundary} if there exists $m \in M$ such that $\omega(g) = g.m - m$. We will denote the group of co-cycles with point-wise addition as $\text{Cocyc}(G,M)$ and the group of co-boundaries by $\text{Cobund}(G,M)$. \textit{First cohomology group of the pair $(G,M)$} is defined as 
    $$
    H^1(G,M) = \bigslant{\text{Cocyc}(G,M)}{\text{Cobund}(G,M)}.
    $$
    If $\omega$ is a co-cycle, by $[\omega]$ we denote an element $\omega + \text{Cobund}(G,M)$ of $H^1(G,M)$.
\end{defi}

It is worth noting that usually one does not define just the first cohomology group, but the whole sequence of them, which gives a broader view but is unnecessary for our purposes. Usually one does not define $H^1(G,M)$ as above, but rather do it more abstractly and only later show that indeed above definition is the correct one up to an isomorphism. Definitions stated bellow will also vary from the usual ones, but these are most explicit and useful for our purposes.

Before we continue let us stress two trivial, yet important observations. Firstly, each co-boundary is a co-cycle, hence definition of the group $H^1(G,M)$ makes sense. Second one is that if $\omega:G \rightarrow M$ is a cocycle, then $\omega(e_G) = e_M$ and if $G$ acts trivially on $M$, that is if $M = M^G$, co-cycles are just regular homomorphisms from $G$ to $M$.

%
%

%
%
\begin{defi}
    Let $G$ be a finite group, $H \leq G$, $K \unlhd G$ and let $M$ be a $G$-module. We will now define maps called \textit{restriction}, \textit{corestriction} and \textit{inflation} denoted $\text{res}$, $\text{cores}$ and $\text{inf}$ respectively.
    $$
    \text{res}: H^1(G,M) \rightarrow H^1(H,M)
    $$
    is given by the formula $[\omega] \mapsto [\omega {|}_{H}]$.
    $$
    \text{cores}: H^1(H,M) \rightarrow H^1(G,M)
    $$
    is defined in a following way. Let $\chi:G \rightarrow G$ be such that for any $g \in G$, $\chi(g) \in gH$ and if $g_1 H = g_2 H$, then $\chi(g_1) = \chi(g_2)$. Define also $X = \text{im}(\chi)$. Natural way to explain it is to say that function $\chi$ chooses a representatives of the left cosets of $H$ in $G$ and $X$ is this set of representatives. For any cocycle $\omega:H \rightarrow M$ define $\omega_{\text{cores}}$ by the formula $g \mapsto \sum_{x \in X} {\left( \chi(g.x) \right)}.\omega({\chi(g.x)}^{-1} \cdot g.x)$ for any $g \in G$. Corestriction can now be defined by the formula $[\omega] \mapsto [\omega_{\text{cores}}]$.
    $$
    \text{inf}: H^1(G/K,M^{K}) \rightarrow H^1(G,M)
    $$
    is define as follows. Let $\varphi:G \rightarrow G/K$ be the natural quotient map, that is the one given by $g \mapsto gK$. Then inflation can be defined by $[\omega] \mapsto [\omega \circ \varphi]$.
\end{defi}
%
We are now ready to state following known results.

\begin{prop} \label{prop: res cor = multiplication by [G:H]}
    (\cite[Chapter VII, Proposition 6]{LocalFields}) (Restriction-corestriction sequence) Let $G$ be a finite group, $H \leq G$ and let $M$ be a $G$-module. Then following diagram commutes.
\[\begin{tikzcd}
	{H^{1}(G,M)} && {H^{1}(G,M)} \\
	& {H^{1}(H,M)}
	\arrow["{\cdot [G:H]}", from=1-1, to=1-3]
	\arrow["{\text{res}}"', from=1-1, to=2-2]
	\arrow["{\text{cores}}"', from=2-2, to=1-3]
\end{tikzcd}\]
\end{prop}

\begin{prop} \label{prop: multiplication by |G| zeros H^1(G,M)}
    Let $G$ be a finite group and let $M$ be a $G$-module $M$. If one put $n = |G|$, then the function $\cdot n: H^{1}(G,M) \rightarrow H^{1}(G,M)$ is the zero map.
\end{prop}
\begin{proof}
    Apply Proposition \ref{prop: res cor = multiplication by [G:H]} with $H = \{e_G\}$. Since $H^1( \{e_G\} ,M) = 0$, $\text{res}$ is a zero map, hence $\cdot n = \text{cores} \circ \text{res}$ also is.
\end{proof}

%
%
We call a sequence $A \overset{\varphi}{\rightarrow} B \overset{\psi}{\rightarrow} C$ \textit{exact} if $\text{im}(\varphi) = \text{ker}(\psi)$. We call longer sequences like $A \overset{\varphi}{\rightarrow} B \overset{\psi}{\rightarrow} C \overset{\chi}{\rightarrow} D$ exact when each of sequences $A \overset{\varphi}{\rightarrow} B \overset{\psi}{\rightarrow} C$ and $B \overset{\psi}{\rightarrow} C \overset{\chi}{\rightarrow} D$ are exact. When we have a $G$-module $M$ and $K \unlhd G$. For $\omega \in \text{Cocyc}(K,M)$ and any $g \in G$ by $g.\omega$ we understand the unique function making the following diagram commute.
\[\begin{tikzcd}
	K & M \\
	K & M
	\arrow["\omega", from=1-1, to=1-2]
	\arrow["{\iota(g)}", from=1-1, to=2-1]
	\arrow["{g.}", from=1-2, to=2-2]
	\arrow["{\exists!\text{ } g.\omega}", dotted, from=2-1, to=2-2]
\end{tikzcd}\]
Just to make above definition complete, $\iota(g):G \rightarrow G$ is given by the formula $k \mapsto g k g^{-1}$. In fact it happens that for any $k \in K$ $[k.\omega] = [\omega]$, hence above action of $G$ on $\text{Cocyc}(K,M)$ induces an action of $G / K$ on $H^1(K,M)$.

\begin{prop} \label{prop: inf-res exact sequence}
    (\cite[Chapter VII, Proposition 6]{LocalFields}) Let $G$ be finite and let $K \unlhd G$. Then the following sequence is exact.
\[\begin{tikzcd}
	0 & {H^1(G/K,M^K)} & {H^1(G,M)} & {{H^1(K,M)}^{G/K}}
	\arrow[from=1-1, to=1-2]
	\arrow["{\text{inf}}", from=1-2, to=1-3]
	\arrow["{\text{res}}", from=1-3, to=1-4]
\end{tikzcd}\]
\end{prop}

\subsection{Proof of Theorem \ref{Thm: action of a nontrivial cocycle from group G acting primitively on X with additional assumptions}} \hfill \\

Before we prove the main result of this subsection we will prove a couple of lemmas which will make the proof of Theorem \ref{Thm: action of a nontrivial cocycle from group G acting primitively on X with additional assumptions} much clearer.
\begin{lem} \label{lem: H^1({Z}_m,{F}_2[X]) = 0}
    Let $m$ be an odd integer and let $\mathbb{Z}_m$ act on the set $X=\mathbb{Z}_m$ by right addition. Then 
    $$
    H^1(\mathbb{Z}_m,\mathbb{F}_2[X]) = 0.
    $$
\end{lem}
\begin{proof}
    Since $\mathbb{F}_2[X]$ is a linear space over $\mathbb{F}_2$, $\cdot 2: H^1(\mathbb{Z}_m,\mathbb{F}_2[X]) \rightarrow H^1(\mathbb{Z}_m,\mathbb{F}_2[X])$ is a zero map. On the other hand by Proposition \ref{prop: multiplication by |G| zeros H^1(G,M)} $\cdot m: H^1(\mathbb{Z}_m,\mathbb{F}_2[X]) \rightarrow H^1(\mathbb{Z}_m,\mathbb{F}_2[X])$ is a zero map. Since $2$ is coprime to $m$, identity map is also a zero map, hence $H^1(\mathbb{Z}_m,\mathbb{F}_2[X]) = 0$.
\end{proof}

Before we prove the next lemma we need to demonstrate that the following holds.

\begin{obs} \label{obs: (x,x+1,x+2) generates A_m}
    Let $m \leq 3$ be any integer and let $S = \{ (x+1,x+2,x+3) \mid x \in \mathbb{Z}_m \}$. Then $\langle S \rangle = A_m$.
\end{obs}
\begin{proof}
    Let $S_m$ act on the set $\mathbb{Z}_m$. It is known that $S_m$ can be generated by transpositions given by $(x,x+1)$ for some $x \in \mathbb{Z}_m$. By definition $A_m$ is the set of permutations which are composed of an even number of such transpositions. To prove the observation it is enough to show that for any $x,y \in \mathbb{Z}_k$, permutation $(x,x+1)(y,y+1)$ can be generated by elements from the set $S$. Take such positive integer $z$ that $x+z = y$ (mod $k$). If we put $\sigma(x) = (x,x+1)(x+1,x+2)$, then we can write
    $$
    (x,x+1) (y,y+1) = (x,x+1)(x+1,x+2)(x+1,x+2)(y,y+1) =
    $$
    $$
    = \sigma(x) \circ (x+1,x+2)(y,y+1) = \ldots = \sigma(x) \circ \ldots \circ \sigma(x+z-1).
    $$
    Since $\sigma(x) = (x,x+1,x+2) \in S$, we conclude that indeed $\langle S \rangle = A_k$.
\end{proof}

\begin{lem} \label{lem: H^1(A_k,F_2[X]) = 0}
    Let $k \geq 5$ be an odd integer and let $A_k$ act on the set $X$ such that $|X|=k$ in a standard way. Then 
    $$
    H^1(A_k,\mathbb{F}_2[X]) = 0.
    $$
\end{lem}
\begin{proof}
    We will identify $X$ with $\mathbb{Z}_k$. Let $\omega:A_k \rightarrow \mathbb{F}_2[X]$ be a co-cycle and let $\mathfrak{r}: \mathbb{Z}_k \rightarrow \mathbb{Z}_k$ be a permutation given by $x \mapsto x+1$. Note that $[\omega {|}_{\langle \mathfrak{r} \rangle}] \in H^1(\langle \mathfrak{r} \rangle ,\mathbb{F}_2[X])$ and $\langle \mathfrak{r} \rangle \cong \mathbb{Z}_k$, hence by Lemma \ref{lem: H^1({Z}_m,{F}_2[X]) = 0}, $\omega {|}_{\langle \mathfrak{r} \rangle}$ is a co-boundary, hence there exists $\nu \in \mathbb{F}_2[X]$ such that $\omega(\mathfrak{r}^{\ell}) = \mathfrak{r}^{\ell}.\nu - \nu$. Now let $\omega'$ be given by the formula $\omega'(\sigma) = \omega(\sigma) - \sigma.\nu + \nu$. Since function given by $\sigma \mapsto \sigma.\nu - \nu$ is a co-boundary by definition, we get $[\omega] = [\omega']$. We will now show that $[\omega'] \equiv \vec{0}$.

    Let $a,b,c \in \mathbb{Z}_k$ be a pairwise different triple. Notice that
    $$
    \vec{0} = \omega'(id) = \omega'({(abc)^{3}}) = {(abc)}^{2}.\omega'((abc)) + {(abc)}.\omega'((abc)) + \omega'((abc)).
    $$
    Let $\omega'((abc)) = \sum_{x \in \mathbb{Z}_k} \varepsilon_{x} \vec{e}_x$. Then for each $x \notin \{a,b,c\}$ we get $0 = 3 \varepsilon_x$ by comparing coefficients in front of $\vec{e}_x$, and it implies $\varepsilon_x = 0$. Additionally by comparing coefficients in front of basis vector $\vec{e}_a$ we get $0 = \varepsilon_a + \varepsilon_b + \varepsilon_3$. This shows that
    $$
    \omega'((abc)) \in \{\vec{0},\vec{e}_a+\vec{e}_b,\vec{e}_b+\vec{e}_c,\vec{e}_c+\vec{e}_a\}.
    $$
    Now notice that for any $x \in \mathbb{Z}_k$ following equality occurs
    $$
    \omega'((x+1,x+2,x+3)) = \omega'(\mathfrak{r}^x (1,2,3) \mathfrak{r}^{-x}) = \mathfrak{r}^x.\omega'( (1,2,3) \mathfrak{r}^{-x}) + \omega'(\mathfrak{r}^x) =
    $$
    $$
    = \mathfrak{r}^x.{ \left( (1,2,3).\omega'(\mathfrak{r}^x) + \omega'((1,2,3))  \right)} = \mathfrak{r}^x.\omega'((1,2,3)).
    $$
    If $\omega'((1,2,3)) = \sum_{x \in \mathbb{Z}_k} \varepsilon_{x} \vec{e}_x \neq \vec{0}$, then either $\varepsilon_2 = 1$ or $\varepsilon_3=1$. Now we will eliminate both of these cases. For each $1 \leq i \leq (k-1)/2$ define elements $\delta_x^i \in \mathbb{F}_2$ to satisfy
    $$
    \omega'((1,2,\ldots,2i+1)) = \sum_{x \in \mathbb{Z}_k} \delta_{x}^i \vec{e}_x.
    $$

    \textit{Case 1:} Assume $\varepsilon_2 = 1$. We will prove inductively that for each $i$, $\delta_{2i}^i = 1$ and $\delta_{x}^i=0$ for $x \geq 2i+2$. Case $i=1$ holds by our assumption. Now assume it holds for $i-1$.
    $$
    \omega'((1,2,\ldots,2i+1)) = \omega'((1,2,\ldots,2i-1)(2i-1,2i,2i+1)) = 
    $$
    $$
    = (1,2,\ldots,2i-1).\omega'((2i-1,2i,2i+1)) + \omega'((1,2,\ldots,2i-1)).
    $$
    Since $\omega'((2i-1,2i,2i+1)) = \mathfrak{r}^{2i-2}.\omega'((1,2,3))$ we conclude that $\delta_{2i}^i = 1 + \delta_{2i}^{i-1} = 1$ and $\delta_{x}^i=0$ for $x \geq 2i+2$ as wanted. Now inductive argument is complete.

    Note that $\omega'(\mathfrak{r}) = \omega'(1,2,\ldots,k) = \sum_{x \in \mathbb{Z}_k} \delta_{x}^{(k-1)/2} \vec{e}_x$ and since $\delta_{k-1}^{(k-1)/2}=1$, we get a contradiction with the fact that $\omega'(\mathfrak{r}) = \vec{0}$. \\

    \textit{Case 2:} Assume $\varepsilon_3 = 1$. Proof that it is impossible will be similar to the one in the first case. We will prove inductively that for each $i$, $\delta_{2i+1}^i = 1$ and $\delta_{x}^i=0$ for $x \geq 2i+2$. Case $i=1$ holds by our assumption. Now assume it holds for $i-1$.
    $$
    \omega'((1,2,\ldots,2i+1)) = \omega'((1,2,\ldots,2i-1)(2i-1,2i,2i+1)) = 
    $$
    $$
    = (1,2,\ldots,2i-1).\omega'((2i-1,2i,2i+1)) + \omega'((1,2,\ldots,2i-1)).
    $$
    Since $\omega'((2i-1,2i,2i+1)) = \mathfrak{r}^{2i-2}.\omega'((1,2,3))$ we conclude that $\delta_{2i+1}^i = 1 + \delta_{2i+1}^{i-1} = 1$ and $\delta_{x}^i=0$ for $x \geq 2i+2$ as wanted. Now inductive argument is complete.

    Note that $\omega'(\mathfrak{r}) = \omega'(1,2,\ldots,k) = \sum_{x \in \mathbb{Z}_k} \delta_{x}^{(k-1)/2} \vec{e}_x$ and since $\delta_{k}^{(k-1)/2}= \delta_{(k-1) + 1}^{(k-1)/2} = 1$, we get a contradiction with the fact that $\omega'(\mathfrak{r}) = \vec{0}$. \\

    Above argument shows that $\omega'((1,2,3)) = \vec{0}$. Since for any $x \in \mathbb{Z}_k$ we have equality $\omega'((x+1,x+2,x+3)) = \mathfrak{r}^x.\omega'((1,2,3))$, it follows that $\omega'((x+1,x+2,x+3)) = \vec{0}$. By Observation \ref{obs: (x,x+1,x+2) generates A_m} set $S = \{ (x+1,x+2,x+3) \mid x \in \mathbb{Z}_m \}$ generates $A_m$, hence we obtain $\omega' \cong \vec{0}$, which means $[\omega] = [\omega'] = 0$ for any $[\omega] \in H^1(A_k,\mathbb{F}_2[X])$.
\end{proof}
%
%
Before we state the next lemma we have to define certain subgroups of $\text{PSL}_2(\mathbb{F}_{2^{\ell}})$. At first notice that $a^2 = \text{det}\left(\begin{bmatrix}
a & 0\\
0 & a 
\end{bmatrix} \right)$ equals $1$ if and only if $a=1$ because $\text{char}(\mathbb{F}_{2^{\ell}}) = 2$, hence function $\phi: \mathbb{F}_{2^{\ell}}^{*} \rightarrow \mathbb{F}_{2^{\ell}}^{*}$ given by $a \mapsto a^2$ is the Frobenius automorphism. Above shows that $\text{PSL}_2(\mathbb{F}_{2^{\ell}}) = \text{SL}_2(\mathbb{F}_{2^{\ell}})$ and $\text{SL}_2(\mathbb{F}_{2^{\ell}}) \cong \text{PGL}_2(\mathbb{F}_{2^{\ell}})$. Now we define
$$
T_{\ell} = \Biggl\{ \begin{bmatrix}
1 & a\\
0 & 1 
\end{bmatrix} \mid a \in \mathbb{F}_{2^{\ell}} \Biggr\} \quad \text{and} \quad {U}_{\ell} = \Biggl\{ \begin{bmatrix}
a & b\\
0 & a^{-1} 
\end{bmatrix} \mid a \in \mathbb{F}_{2^{\ell}}^{*}, b \in \mathbb{F}_{2^{\ell}} \Biggr\}.
$$


\begin{lem} \label{lem: S_{ell}/T_{ell} invariants of H^1(T_{ell},F_2[X]) is the trivial group}
    Let $\ell \geq 2$ be an arbitrary integer and let $T_{\ell}$ act on $X = \mathbb{P}^{1} \mathbb{F}_{2^{\ell}}$ in a standrad way. Then
    $$
    H^1(T_{\ell},\mathbb{F}_2[X]) \cong \mathbb{F}_{2^{\ell}} \quad \text{and} \quad {H^1(T_{\ell},\mathbb{F}_2[X])}^{U_{\ell} \text{/} T_{\ell}} = 0.
    $$
\end{lem}
\begin{proof}
    For any $a \in \mathbb{F}_{2^{\ell}}$ by $\ell(a)$ we denote the line $\mathbb{F}_{2^{\ell}} \cdot {[a,1]}^{T}$. By $\ell(\infty)$ we denote the line $\mathbb{F}_{2^{\ell}} \cdot {[1,0]}^{T}$. Let us define $F_{\ell} = \bigoplus_{a \in \mathbb{F}_{2^{\ell}}} \mathbb{F}_{2^{\ell}} \cdot \vec{e}_{ \ell(a) }$. Notice that action of $T_{\ell}$ fixes both submodules $F_{\ell}$ and $\mathbb{F}_{2^{\ell}} \cdot \vec{e}_{ \ell(\infty)}$ setwise. Therefore $\mathbb{F}_2[X] \cong F_{\ell} \oplus \mathbb{F}_{2^{\ell}} \cdot \vec{e}_{ \ell(\infty)}$ as an $T_{\ell}$-module. Above observation allows us to state that 
    $$
    H^1(T_{\ell},\mathbb{F}_2[X]) \cong H^1(T_{\ell},F_{\ell}) \oplus H^1(T_{\ell},\mathbb{F}_{2^{\ell}} \cdot \vec{e}_{ \ell(\infty)}), \text{ and}
    $$
    $$
    {H^1(T_{\ell},\mathbb{F}_2[X])}^{U_{\ell} \text{/} T_{\ell}} \cong {H^1(T_{\ell},F_{\ell})}^{U_{\ell} \text{/} T_{\ell}} \oplus {H^1(T_{\ell},\mathbb{F}_{2^{\ell}} \cdot \vec{e}_{ \ell(\infty)})}^{U_{\ell} \text{/} T_{\ell}}.
    $$
    Now we will calculate each of the cohomologies over submodules $F_{\ell}$ and $\mathbb{F}_{2^{\ell}} \cdot \vec{e}_{ \ell(\infty)}$. \\
    
    \textit{Calculation of $H^1(T_{\ell},F_{\ell})$:} Observe that $\varphi:\mathbb{F}_{2^{\ell}} \rightarrow T_{\ell}$ given by $a \rightarrow \begin{bmatrix}
1 & a\\
0 & 1
\end{bmatrix}$ is an isomorphism. Moreover, if one define the action of the group $\mathbb{F}_{2^{\ell}}$ on the set $\mathbb{F}_{2^{\ell}}$ by right addition, that is for any $a,b \in \mathbb{F}_{2^{\ell}}$ we put $a.b = b + a$. Define the map $\widetilde{\varphi}: \mathbb{F}_2[\mathbb{F}_{2^{\ell}}] \rightarrow F_{\ell}$ by $\vec{e}_{a} \mapsto \vec{e}_{\ell(a)}$ on the basis ${ \{ \vec{e}_{a} \} }_{a \in \mathbb{F}_{2^{\ell}}}$ and extend linearly. Notice that for any $a \in \mathbb{F}_{2^{\ell}}$ following diagram commutes.
\[\begin{tikzcd}
	{\mathbb{F}_2[\mathbb{F}_{2^{\ell}}]} & {\mathbb{F}_2[\mathbb{F}_{2^{\ell}}]} \\
	{F_{\ell}} & {F_{\ell}}
	\arrow["{a.}", from=1-1, to=1-2]
	\arrow["{\widetilde{\varphi}}", from=1-1, to=2-1]
	\arrow["{\widetilde{\varphi}}", from=1-2, to=2-2]
	\arrow["{\varphi(a).}", from=2-1, to=2-2]
\end{tikzcd}\]
This demonstrates that in fact $H^1(T_{\ell},F_{\ell}) \cong H^1(\mathbb{F}_{2^{\ell}},\mathbb{F}_2[\mathbb{F}_{2^{\ell}}])$. We will show that any co-cycle $\omega:\mathbb{F}_{2^{\ell}} \rightarrow \mathbb{F}_2[\mathbb{F}_{2^{\ell}}]$ is actually a co-boundary. Let
$$
\omega(a) = \sum_{b \in \mathbb{F}_{2^{\ell}}} \delta_a^b \cdot \vec{e}_b,
$$
for arbitrary $a,b \in \mathbb{F}_{2^{\ell}}$ and appropriate $\delta_a^b \in \mathbb{F}_2$. Take any $c \in \mathbb{F}_{2^{\ell}}$. Extracting coefficient in front of $\vec{e}_c$ from the co-cycle equation $\omega(a+b) = a.\omega(b) + \omega(a)$ shows that
$$
\delta_{a+b}^{c} = \delta_{b}^{c+a} + \delta_a^c.
$$
If we put $a=b=c$ and account for the fact that $\omega(0) = \vec{0}$, we get $0 = \delta_0^a = \delta_a^0 + \delta_a^a$, hence $\delta_a^0 = \delta_a^a$. If in the same equation we just put $c=0$ we get $\delta_{a+b}^a = \delta_b^0 + \delta_a^a = \delta_b^0 + \delta_a^0$. If we now substitute $x= a+b$ and $y=a$ we get 
$$
\delta_x^y = \delta_{x+y}^0 + \delta_y^0.
$$
Above equation shows that co-cycle $\omega$ is fully determined by the sequence ${\{\delta_a^0\}}_{a \in \mathbb{F}_{2^{\ell}}}$. Additionally, since $\omega(0) = \vec{0}$, we have $\delta_0^0 = 0$. To show that $\omega$ indeed is a co-boundary we will construct an element $\nu \in \mathbb{F}_2[\mathbb{F}_{2^{\ell}}]$ which fulfill any given sequence ${\{\delta_a^0\}}_{a \in \mathbb{F}_{2^{\ell}}}$ with $\delta_0^0 = 0$.

Put $\nu = \sum_{a \in \mathbb{F}_{2^{\ell}}} \delta_a^0 \cdot \vec{e}_{a}$ and let $\widetilde{\omega}: \mathbb{F}_{2^{\ell}} \rightarrow \mathbb{F}_2[\mathbb{F}_{2^{\ell}}]$ be given by $a \mapsto a.\nu - \nu$. Indeed the coefficient of $\widetilde{\omega}(a)$ next to $\vec{e}_{0}$ equals $\delta_{(-a)}^0 - \delta_0^0 = \delta_a^0$ as intended, which shows that
$$
H^1(T_{\ell},F_{\ell}) \cong H^1(\mathbb{F}_{2^{\ell}},\mathbb{F}_2[\mathbb{F}_{2^{\ell}}]) = 0.
$$
\\
    \text{ }\text{ }\text{ }\textit{Calculation of $H^1(T_{\ell},\mathbb{F}_{2^{\ell}} \cdot \vec{e}_{ \ell(\infty)})$:} Notice that each element of $T_{\ell}$ fixes $\mathbb{F}_{2^{\ell}} \cdot \vec{e}_{ \ell(\infty)}$ point-wise. Since $T_{\ell} \cong \mathbb{F}_{2^{\ell}} \cong \mathbb{F}_{2}^{\ell}$ as established earlier and $\mathbb{F}_{2^{\ell}} \cdot \vec{e}_{ \ell(\infty)} \cong \mathbb{F}_2$, we have an isomorphism $H^1(T_{\ell},\mathbb{F}_{2^{\ell}} \cdot \vec{e}_{ \ell(\infty)}) \cong H^1(\mathbb{F}_{2}^{\ell},\mathbb{F}_2)$ where $\mathbb{F}_2$ is considered as $\mathbb{F}_{2}^{\ell}$-module with trivial action.

    Triviality of action of $\mathbb{F}_{2}^{\ell}$ on $\mathbb{F}_2$ means that co-cycles are just homomorphisms from $\mathbb{F}_{2}^{\ell}$ to $\mathbb{F}_2$ and the only co-boundary is the $0$ function, hence $H^1(T_{\ell},\mathbb{F}_{2^{\ell}} \cdot \vec{e}_{ \ell(\infty)}) \cong {\left( \mathbb{F}_{2}^{\ell} \right)}^{*} \cong \mathbb{F}_{2}^{\ell}$. \\

    \textit{Calculation of ${H^1(T_{\ell},\mathbb{F}_{2^{\ell}} \cdot \vec{e}_{ \ell(\infty)})}^{U_{\ell} \text{/} T_{\ell}}$:} For arbitrary $a \in \mathbb{F}_{2^{\ell}}^{*}$ and $b \in \mathbb{F}_{2^{\ell}}$ take $\eta(a) = \begin{bmatrix}
a & 0\\
0 & a^{-1}
\end{bmatrix}$ and $\tau(b) = \begin{bmatrix}
1 & b\\
0 & 1
\end{bmatrix}$. Then $\eta(a) \tau(b) {\eta(a)}^{-1} = \tau(b\cdot a^2)$. Note that $\mathbb{F}_{2^{\ell}} \cdot \vec{e}_{ \ell(\infty)}$ is fixed point-wise by any element of $U_{\ell}$. Let $\omega: T_{\ell} \rightarrow \mathbb{F}_2 [\mathbb{F}_{2^{\ell}} \cdot \vec{e}_{ \ell(\infty)}]$ be a co-cycle such that $[\omega] \in {H^1(T_{\ell},\mathbb{F}_{2^{\ell}} \cdot \vec{e}_{ \ell(\infty)})}^{U_{\ell} \text{/} T_{\ell}}$.

We already showed that there are no nonzero co-boundaries in $H^1(T_{\ell},\mathbb{F}_{2^{\ell}} \cdot \vec{e}_{ \ell(\infty)})$, hence from $[\tau(a).\omega] = [\omega]$ we deduce $\tau(a).\omega = \omega$. For any $b \in \mathbb{F}_{2^{\ell}}$ by definition we have $$
{\omega}(\eta(b)) = {\left( \tau(a).\omega \right)}(\eta(b)) = \tau(a).{\left( \omega({\tau(a)}^{-1} \eta(b) \tau(a)) \right)} = \omega( \eta(b\cdot a^{-2}) ).
$$
Since field $\mathbb{F}_{2^{\ell}}$ have characteristic $2$, function $\phi: \mathbb{F}_{2^{\ell}}^{*} \rightarrow \mathbb{F}_{2^{\ell}}^{*}$ given by the formula $a \mapsto a^{-2}$ is a bijection, hence we can take arbitrary $\widetilde{a} \in \mathbb{F}_{2^{\ell}}^{*} $ and we can put $a = \phi^{-1}(\widetilde{a})$ to obtain ${\omega}(\eta(b)) = \omega( \eta(b\cdot \widetilde{a}) )$. This shows that $\omega {|}_{\mathbb{F}_{2^{\ell}}^{*}}$ is constant. Since we showed before that $\omega$ is linear, for any two different elements $x,y \in \mathbb{F}_{2^{\ell}}^{*}$ we have $x+y \in \mathbb{F}_{2^{\ell}}^{*}$ and
$$
2\omega(x) = \omega(x) + \omega(y) = \omega(x+y) = \omega(x).
$$
Above equation demonstrates that $\omega {|}_{\mathbb{F}_{2^{\ell}}^{*}} \equiv \vec{0}$, hence $\omega \equiv \vec{0}$ and ${H^1(T_{\ell},\mathbb{F}_{2^{\ell}} \cdot \vec{e}_{ \ell(\infty)})}^{U_{\ell} \text{/} T_{\ell}} = 0$. \\

Combining all of the above results gives us
$$
    H^1(T_{\ell},\mathbb{F}_2[X]) \cong H^1(T_{\ell},F_{\ell}) \oplus H^1(T_{\ell},\mathbb{F}_{2^{\ell}} \cdot \vec{e}_{ \ell(\infty)}) \cong 0 \oplus \mathbb{F}_{2^{\ell}} \cong \mathbb{F}_{2^{\ell}}, \text{ and}
$$
$$
    {H^1(T_{\ell},\mathbb{F}_2[X])}^{U_{\ell} \text{/} T_{\ell}} \cong {H^1(T_{\ell},F_{\ell})}^{U_{\ell} \text{/} T_{\ell}} \oplus {H^1(T_{\ell},\mathbb{F}_{2^{\ell}} \cdot \vec{e}_{ \ell(\infty)})}^{U_{\ell} \text{/} T_{\ell}} \cong 0 \oplus 0 \cong 0.
$$
\end{proof}

Now we are prepared to prove of Theorem \ref{Thm: action of a nontrivial cocycle from group G acting primitively on X with additional assumptions}. \\

\textit{Proof of Theorem \ref{Thm: action of a nontrivial cocycle from group G acting primitively on X with additional assumptions}:} \\
Let us start by recognizing that the group $G$ acting on the set $X$ fulfils all assumptions of Theorem \ref{Thm characterization of primitive group actions with regular cyclic subgroup and additional assumptions}, hence we only have to show the conclusion of Theorem \ref{Thm: action of a nontrivial cocycle from group G acting primitively on X with additional assumptions} for group actions which belong to the families i. - iii. 

Observe that for any co-cycle $\omega': G \rightarrow \mathbb{F}_2[X]$, $\omega' {|}_{ {(\mathbb{Z}_k)}_r }$ is a co-boundary by Lemma \ref{lem: H^1({Z}_m,{F}_2[X]) = 0}. Let $\nu \in \mathbb{F}_2[X]$ be such that $\omega' {|}_{ {(\mathbb{Z}_k)}_r }$ is given by $\sigma \mapsto \sigma.\nu - \nu$. If we define $\omega'': G \rightarrow \mathbb{F}_2[X]$ by the formula $ g \mapsto \omega'(g) - g.\nu + \nu$ then $\omega'' {|}_{ {(\mathbb{Z}_k)}_r } \equiv \vec{0}$. This $\omega''$ is unique. Otherwise there would exists a nonzero co-cycle $\varepsilon: G \rightarrow \mathbb{F}_2[X]$ such that $\varepsilon {|}_{ {(\mathbb{Z}_k)}_r } \equiv \vec{0}$. Let such $\varepsilon$ be given by the formula $g \mapsto g.\upsilon - \upsilon$ for $\upsilon = \sum_{x \in \mathbb{Z}_k } \delta_x \vec{e}_x$. Notice however, that then 
$$
\vec{0} = \varepsilon(1_r) = \sum_{x \in \mathbb{Z}_k } (\delta_{x-1} - \delta_x) \cdot \vec{e}_x,
$$
so it would imply that $\delta_x = \text{const.}$ for all $x \in \mathbb{Z}_k$, however it automatically implies $\varepsilon \equiv \vec{0}$. For convenience from now on by $\mathbbm{1}$ we will denote the vector $\sum_{x \in X} \vec{e}_x$. While keeping all of the above in mind we will now gone through three cases indicated by Theorem \ref{Thm characterization of primitive group actions with regular cyclic subgroup and additional assumptions}. \\

\textit{Case i.} We have $\mathbb{F}_p \leq G \leq \text{Aff}(\mathbb{F}_p)$ and $X = \mathbb{F}_p$ where $k=p$ is an odd prime. By Lemma \ref{lem: H^1({Z}_m,{F}_2[X]) = 0} we get $H^1(\mathbb{Z}_p,\mathbb{F}_2[X]) = 0$. Since $\mathbb{Z}_p \unlhd \text{Aff}(\mathbb{F}_p)$, $\mathbb{Z}_p \unlhd G$ and hence we can apply Proposition \ref{prop: inf-res exact sequence} to obtain that bellow sequence is exact.
\[\begin{tikzcd}
	0 & {H^1(G/\mathbb{Z}_p,{\mathbb{F}_2[X]}^{\mathbb{Z}_p})} & {H^1(G,\mathbb{F}_2[X])} & {H^1(\mathbb{Z}_p,\mathbb{F}_2[X])= 0}
	\arrow[from=1-1, to=1-2]
	\arrow["{\text{inf}}", from=1-2, to=1-3]
	\arrow["{\text{res}}", from=1-3, to=1-4]
\end{tikzcd}\]
This shows that $H^1(G/\mathbb{Z}_p,{\mathbb{F}_2[X]}^{\mathbb{Z}_p}) \overset{\text{inf}}{\cong} H^1(G,\mathbb{F}_2[X])$. Notice that ${\mathbb{F}_2[X]}^{\mathbb{Z}_p} = \mathbb{F}_2 \cdot \mathbbm{1}$, hence $G/\mathbb{Z}_p$ acts trivially on it. This shows that $H^1(G/\mathbb{Z}_p,{\mathbb{F}_2[X]}^{\mathbb{Z}_p}) \cong \text{Hom}(G/\mathbb{Z}_p,\mathbb{F}_2 \cdot \mathbbm{1})$. Note that since $G/\mathbb{Z}_p \leq \mathbb{Z}_{p-1}$, it is cyclic, so only nontrivial automorphism is the one which maps the subgroup $H \leq G/\mathbb{Z}_p$ of index $2$ onto $\vec{0}$ and other elements are mapped to $\mathbbm{1}$. Such automorphism, hence the subgroup $H$ need to exist since we are given a nonzero cocycle $\omega$ which vanishes on ${(\mathbb{Z}_p)}_r$.

Checking definition of inflation map shows that indeed only nontrivial element of $H^1(G,\mathbb{F}_2[X])$ is represented by a co-cycle which maps elements of some subgroup $G_0$ of index two onto $\vec{0}$ and others onto $\mathbbm{1}$. Since subgroup ${(\mathbb{Z}_k)}_r \leq G$ has odd order, we have ${(\mathbb{Z}_k)}_r \leq G_0$ so this co-cycle needs to equal $\omega$, hence conclusion of Theorem \ref{Thm: action of a nontrivial cocycle from group G acting primitively on X with additional assumptions} follows. \\

\textit{Case ii.} We have $A_{k} \leq G \leq S_{k}$ for some odd $k \geq 5$ and $X = \mathbb{Z}_k$. By Lemma \ref{lem: H^1(A_k,F_2[X]) = 0} we get $H^1(A_k,\mathbb{F}_2[X]) = 0$, hence $G = S_k$ because we are given a nontrivial co-cycle $\omega$ with additional properties discussed above. Since $A_k \unlhd S_k$ we can apply Proposition \ref{prop: inf-res exact sequence}.
\[\begin{tikzcd}
	0 & {H^1(S_k/A_k,{\mathbb{F}_2[X]}^{A_k})} & {H^1(S_k,\mathbb{F}_2[X])} & {H^1(A_k,\mathbb{F}_2[X]) = 0}
	\arrow[from=1-1, to=1-2]
	\arrow["{\text{inf}}", from=1-2, to=1-3]
	\arrow["{\text{res}}", from=1-3, to=1-4]
\end{tikzcd}\]
Since ${(\mathbb{Z}_k)}_r \leq A_k$ we obtain ${\mathbb{F}_2[X]}^{A_k} = \mathbb{F}_2 \cdot \mathbbm{1}$, hence $S_k/A_k$ acts trivially on it. We now get
$$
\mathbb{F}_2^{*} \cong \text{Hom}(S_k/A_k,\mathbb{F}_2 \cdot \mathbbm{1}) = H^1(S_k/A_k,{\mathbb{F}_2[X]}^{A_k}) \overset{\text{ind}}{\cong} H^1(S_k,\mathbb{F}_2[X]).
$$
It is now clear that the only nontrivial co-cycle from  $S_k$ to $\mathbb{F}_2[X]$ which vanishes on ${(\mathbb{Z}_k)}_r$ it the one which maps elements of $A_k$ onto $\vec{0}$ and the rest onto $\mathbbm{1}$. $\omega$ satisfies above assumptions, so it has to be equal to the co-cycle described above and conclusion of Theorem \ref{Thm: action of a nontrivial cocycle from group G acting primitively on X with additional assumptions} holds.\\

\textit{Case iii.} Now we have $\text{PGL}_{2}(\mathbb{F}_{2^\ell}) \leq G \leq \text{P}\Gamma\text{L}_{2}(\mathbb{F}_{2^\ell})$ and $X = \mathbb{P}^1\mathbb{F}_{2^\ell}$ for some positive $\ell \geq 2$. Action of $G$ on $X$ is the standard action of projective group on lines. As recognized earlier, $\text{PSL}_{2}(\mathbb{F}_{2^\ell}) \cong 
 \text{SL}_{2}(\mathbb{F}_{2^\ell}) \cong \text{PGL}_{2}(\mathbb{F}_{2^\ell})$. We will start by showing that $H^{1}(\text{SL}_{2}(\mathbb{F}_{2^\ell}),\mathbb{F}_2[X]) = 0$.

 By Lemma \ref{lem: S_{ell}/T_{ell} invariants of H^1(T_{ell},F_2[X]) is the trivial group} we have ${H^1(T_{\ell},\mathbb{F}_2[X])}^{U_{\ell} \text{/} T_{\ell}} = 0$. Since $T_{\ell} \unlhd U_{\ell}$, we can apply Proposition \ref{prop: inf-res exact sequence}. We obtain the following exact sequence
\[\begin{tikzcd}
	0 & {H^1(U_{\ell}/T_{\ell},{\mathbb{F}_2[X]}^{T_{\ell}})} & {H^1(U_{\ell},\mathbb{F}_2[X])} & {{H^1(T_{\ell},\mathbb{F}_2[X])}^{U_{\ell} / T_{\ell}} = 0}
	\arrow[from=1-1, to=1-2]
	\arrow["{\text{inf}}", from=1-2, to=1-3]
	\arrow["{\text{res}}", from=1-3, to=1-4]
\end{tikzcd}\]
It shows that $H^1(U_{\ell}/T_{\ell},{\mathbb{F}_2[X]}^{T_{\ell}}) \overset{\text{inf}}{\cong} H^1(U_{\ell},\mathbb{F}_2[X])$. By Proposition \ref{prop: multiplication by |G| zeros H^1(G,M)},
$$
\cdot (2^\ell - 1): H^1(U_{\ell}/T_{\ell},{\mathbb{F}_2[X]}^{T_{\ell}}) \rightarrow H^1(U_{\ell}/T_{\ell},{\mathbb{F}_2[X]}^{T_{\ell}})
$$
is a zero map. On the other hand ${\mathbb{F}_2[X]}^{T_{\ell}}$ is a linear space over $\mathbb{F}_2$, so 
$$
\cdot 2: H^1(U_{\ell}/T_{\ell},{\mathbb{F}_2[X]}^{T_{\ell}}) \rightarrow H^1(U_{\ell}/T_{\ell},{\mathbb{F}_2[X]}^{T_{\ell}})
$$
also is a zero map. Clearly $2$ is coprime to $2^{\ell} - 1$, so identity on $H^1(U_{\ell}/T_{\ell},{\mathbb{F}_2[X]}^{T_{\ell}})$ also is a zero map, hence $H^1(U_{\ell},\mathbb{F}_2[X]) \cong H^1(U_{\ell}/T_{\ell},{\mathbb{F}_2[X]}^{T_{\ell}}) = 0$. Applying Proposition \ref{prop: res cor = multiplication by [G:H]} for $\text{SL}_{2}(\mathbb{F}_{2^\ell})$ and $U_{\ell}$ gives us a commutative diagram
\[\begin{tikzcd}
	{H^{1}(\text{SL}_{2}(\mathbb{F}_{2^\ell}),\mathbb{F}_2[X])} && {H^{1}(\text{SL}_{2}(\mathbb{F}_{2^\ell}),\mathbb{F}_2[X])} \\
	& {H^{1}(U_{\ell},\mathbb{F}_2[X])}
	\arrow["{\cdot (2^{\ell}+1)}", from=1-1, to=1-3]
	\arrow["{\text{res}}"', from=1-1, to=2-2]
	\arrow["{\text{cores}}"', from=2-2, to=1-3]
\end{tikzcd}\]
Since $\mathbb{F}_2[X]$ is a linear space over $\mathbb{F}_2$, $\cdot (2^{\ell}+1)$ is just an identity. On the other hand since $H^{1}(U_{\ell},\mathbb{F}_2[X]) = 0$, $\text{cores} \circ \text{res}$ is a zero map. Commutativity of the above diagram implies 
$$
H^{1}(\text{SL}_{2}(\mathbb{F}_{2^\ell}),\mathbb{F}_2[X]) = 0.
$$

Now notice that $\text{P}\Gamma\text{L}_{2}(\mathbb{F}_{2^{\ell}}) / \text{PGL}_{2}(\mathbb{F}_{2^{\ell}}) \cong \text{Gal}\left( {\mathbb{F}_{2^{\ell}}} \text{/} {\mathbb{F}_2} \right) \cong \mathbb{Z}_{\ell}$, hence $G / \text{PGL}_{2}(\mathbb{F}_{2^{\ell}})$ as its subgroup also has to be cyclic. Putting $K = \text{PGL}_{2}(\mathbb{F}_{2^{\ell}})$ in Proposition \ref{prop: inf-res exact sequence} yields
\[\begin{tikzcd}
	0 & {H^1 \left( G/ \text{PGL}_{2}(\mathbb{F}_{2^{\ell}}) ,{\mathbb{F}_2[X]}^{  \text{PGL}_{2}(\mathbb{F}_{2^{\ell}}) } \right)} & {H^1(G,\mathbb{F}_2[X])} & 0
	\arrow[from=1-1, to=1-2]
	\arrow["{\text{inf}}", from=1-2, to=1-3]
	\arrow["{\text{res}}", from=1-3, to=1-4]
\end{tikzcd}\]
 since ${H^1( \text{PGL}_{2}(\mathbb{F}_{2^{\ell}}) ,\mathbb{F}_2[X])} = 0$. Since $\text{PGL}_{2}(\mathbb{F}_{2^{\ell}})$ acts transitively on $X = \mathbb{P}^1 \mathbb{F}_{2^{\ell}}$ we conclude that ${\mathbb{F}_2[X]}^{  \text{PGL}_{2}(\mathbb{F}_{2^{\ell}}) } = \mathbb{F}_2 \cdot \mathbbm{1}$. It is easy to notice that $G$ acts trivially on $\mathbb{F}_2 \cdot \mathbbm{1}$, hence it is a $G/ \text{PGL}_{2}(\mathbb{F}_{2^{\ell}})$-module with trivial action, so
 $$
 \text{Hom} \bigl( G/ \text{PGL}_{2}(\mathbb{F}_{2^{\ell}}),\mathbb{F}_2 \cdot \mathbbm{1} \bigr) = H^1 \left( G/ \text{PGL}_{2}(\mathbb{F}_{2^{\ell}}) ,{\mathbb{F}_2[X]}^{  \text{PGL}_{2}(\mathbb{F}_{2^{\ell}}) } \right) \overset{\text{inf}}{\cong} H^1(G,\mathbb{F}_2[X]).
 $$
Since we are given a nontrivial co-cycle $\omega: G \rightarrow \mathbb{F}_2[X]$, $\text{Hom}(G/ \text{PGL}_{2}(\mathbb{F}_{2^{\ell}},\mathbb{F}_2 \cdot \mathbbm{1})$ cannot be trivial, hence there exists a nontrivial homomorphism from $G/ \text{PGL}_{2}(\mathbb{F}_{2^{\ell}})$ to $\mathbb{F}_2 \cdot \mathbbm{1}$. Existence of this homomorphism implies existence of subgroup $H \leq G/ \text{PGL}_{2}(\mathbb{F}_{2^{\ell}})$, which is its kernel. Definition of inflation map shows that there exists $G_0 \leq G$ and a co-cycle which sends elements of $G_0$ onto $\vec{0}$ and others onto $\mathbbm{1}$. Noticing that ${(\mathbb{Z}_k)}_r \leq G$ has odd order implies ${(\mathbb{Z}_k)}_r \leq G_0$. We now see that $\omega$ satisfies all of the above, hence demonstrated uniqueness shows that co-cycle described above is $\omega$, hence conclusion of Theorem \ref{Thm: action of a nontrivial cocycle from group G acting primitively on X with additional assumptions} holds.\\

We now settled last of cases indicated by Theorem \ref{Thm characterization of primitive group actions with regular cyclic subgroup and additional assumptions}, hence completed the proof.
\hfill \qedsymbol{}

%
%

\section{Main results} \label{SECTION: Main results}

In Section \ref{SECTION: Schur Rings and function alpha} a lot of effort was put into understanding the case when $n=2m$ for some $m>1$ which is odd and $\Gamma = \text{Cay}(\mathbb{Z}_n,S)$ is some connected, nonbipartite and unstable Cayley graph for which $ \{a,m + a\} $ is a basic set of $\mathcal{A}(\Gamma \times K_2)$. Recall that if above holds, $\mathcal{B}$ is the thickest $\alpha$-homogeneous partition and by Lemma \ref{lem: mathcal{B} forms a block system made of cosets of B} it forms a block system and is made of cosets of a group $B$. For $\Gamma = \text{Cay}(\mathbb{Z}_n,S)$ we additionally define $\Gamma_{\text{reflective}} = \text{Cay}(\mathbb{Z}_n,S_r)$ (cf. Definition \ref{defi: reflective and areflective edges and colored quotient by L}).
We start this section by solving the case described above completely by the following theorem.

\begin{thm} \label{thm: handeling the hard case for 2m with assumptions about replacement property}
    Assume Hypothesis \ref{HYPOTHESIS which states the hard case}. If additionally there exists such a block system $\widetilde{\mathcal{B}}$ made of cosets of $\widetilde{B} \leq \mathbb{Z}_n$ which is minimal among block systems which are thicker than $\mathcal{B}$ and pairs $(\Gamma_{\text{reflective}},B)$, $(\Gamma_{\text{reflective}},\widetilde{B})$ satisfy replacement property, then
    $$
    \text{Cay}(\mathbb{Z}_{2m},S) \cong \text{Cay}(\mathbb{Z}_{2m},S+m).
    $$
\end{thm}
\begin{proof}
    Notice that by Corollary \ref{Cor: (+m)_r in im(alpha) iff. Cay(Z_2m,S) iso Cay(Z_2m,S+m)} we only have to show that $\sigma \in \text{Aut}^{\pi}(\Gamma)$ such that $\alpha(\sigma) = {m}_{r}$. In the proof we will consider two separate cases depending on the form of function $\alpha$ on the subgroup $\text{Aut}_{\mathcal{B}}^{\pi}(\Gamma)$. Also note that by Observation \ref{obs: tf-projections induce automorphisms of Gamma / mathcal{L}} we have $\text{Aut}^{\pi}(\Gamma) \leq \text{Aut}(\Gamma_{\text{reflective}})$. \\

    \textit{Case 1:} Assume $\alpha {|}_{\text{Aut}_{\mathcal{B}}^{\pi}(\Gamma)}$ does not constantly equal ${id}_{\mathbb{Z}_n}$. Then we have an element $\sigma$ such that $\alpha(\sigma) \neq {id}_{\mathbb{Z}_n}$. Since $\mathcal{B}$ is $\alpha$-homogeneous, there exists ${\mathfrak{b}} \in \mathcal{B}$ such that for any $x \in {\mathfrak{b}}$ one has $\alpha(\sigma).x = x+m$. Notice that $(\sigma,\gamma(\sigma))$ induces an automorphism of $\Gamma \times K_2$ which fixes each coset of $B\langle a \rangle$ setwise. 
    
    Let $f: \mathcal{B} \rightarrow \mathbb{Z}_n$ be the function which existence is provided by replacement property of the pair $(\Gamma_{\text{reflective}},B)$ (cf. Definition \ref{defi: replacement property of a pair (Gamma,H)}). Notice that we can conjugate $\sigma$ by ${( f(\mathfrak{b}) )}_r^{-1}$ if needed and without loss of generality assume that $\alpha(\sigma).b = b+m$ for each $b \in B$. We will now prove that $\widetilde{\sigma} \in \text{Aut}_{\mathcal{B}}(\Gamma_{\text{reflective}})$ which existence is postulated by replacement property of $(\Gamma_{\text{reflective}},B)$ is in fact an element of $\text{Aut}_{\mathcal{B}}^{\pi}(\Gamma)$ which satisfy $\alpha(\widetilde{\sigma}) = m_r$.
    
    Consider a function $\tau: \mathbb{Z}_n \langle a \rangle \rightarrow \mathbb{Z}_n \langle a \rangle$ given by formula bellow.
    $$
    \tau(x) = \begin{cases}
        \widetilde{\sigma}(x) & \text{ if } x \in \mathbb{Z}_n \\
     \widetilde{\sigma}(x+m)+m+a & \text{ otherwise }
    \end{cases}
    $$
    Consider edges inside a fixed coset of $B\langle a \rangle$. Note, that for any $b \in B\langle a \rangle$ we get the following formula.
    $$
    \tau(b) = \begin{cases}
        \sigma(b) & \text{ if } x \in \mathbb{Z}_n \\
     \sigma(x+b)+m+b & \text{ otherwise }
    \end{cases}
    $$
    This formula ensures that each edge of $\Gamma \times K_2$ inside $B\langle a \rangle$ is mapped onto some edge since $\alpha(\sigma).b_0 = b_0+m$ for every $b \in B_0$. Replacement property ensures that action of $\tau$ on other cosets of $B\langle a \rangle$ is the same up to a permutation by a rotation, hence edges inside other cosets are also mapped onto edges by the same argument.

    Notice that the set $S_r$ is such that $S_r + m = S_r$, so indeed $\sigma \in \text{Aut}_{\mathcal{B}}(\Gamma_{\text{reflective}})$ and additional rotation of vertices from $\mathbb{Z}_n\langle a \rangle \backslash \mathbb{Z}_n$ by $m$ does not change the fact that edges between different cosets of $B\langle a \rangle$ are mapped onto edges.

    Combining the conclusions from two of the above paragraphs we get $\widetilde{\sigma} \in \text{Aut}_{\mathcal{B}}^{\pi}(\Gamma)$ and $\alpha(\widetilde{\sigma}) = m_r$ whcich ends the proof in the first case. \\

    \textit{Case 2:} Assume $\alpha {|}_{\text{Aut}_{\mathcal{B}}^{\pi}(\Gamma)} 
\equiv {id}_{\mathbb{Z}_n}$. The proof will be based on the following diagram.
\[\begin{tikzcd}
	{\text{Aut}^{\pi}(\Gamma)} & {\text{Aut}_{\widetilde{\mathcal{B}}}^{\pi}(\Gamma) = P} & {Q = P_{f_{\widetilde{\mathcal{B}}}}} \\
	& R \\
	& G
	\arrow[hook', from=1-2, to=1-1]
	\arrow["{\text{res}_{\widetilde{B}}}"', two heads, from=1-2, to=2-2]
	\arrow[hook', from=1-3, to=1-2]
	\arrow[dashed, hook', two heads, from=1-3, to=2-2]
	\arrow["{\text{ind}_{\widetilde{B} \cap \mathcal{B}}}"', two heads, from=2-2, to=3-2]
\end{tikzcd}\]

We have to start by defining some of maps and groups presented on the above diagram. We define $P = \text{Aut}_{\widetilde{\mathcal{B}}}^{\pi}(\Gamma)$ as seen above. Let $f_{\widetilde{B}}: \widetilde{\mathcal{B}} \rightarrow \mathbb{Z}_n$ be the function which existence is ensured by replacement property of pair $(\Gamma_{\text{reflective}},\widetilde{B})$. By $P_{f_{\widetilde{\mathcal{B}}}}$ we understand the subgroup of $P$ made of such $\sigma \in P$, that $\widetilde{\sigma} = \sigma$, where $\widetilde{\sigma}$ it the the element of $\text{Aut}_{\mathcal{B}}(\Gamma_{\text{reflective}})$ postulated by the replacement property of the pair $(\Gamma_{\text{reflective}},\widetilde{B})$ understood with respect to previously mentioned function $f_{\widetilde{B}}$. 

Function $\text{res}_{\widetilde{B}}: P \rightarrow \text{Sym}(\widetilde{B})$ is given by the formula $\sigma \rightarrow \sigma {|}_{\widetilde{B}}$. Now we can pit $R = \text{im}(\text{res}_{\widetilde{B}})$. Define $\text{ind}_{\widetilde{B} / B}: R \rightarrow \text{Sym}(\widetilde{B} / B)$ to be the unique function such that for every $\sigma \in R$ following diagram commutes.
\[\begin{tikzcd}
	{\widetilde{B}} && {\widetilde{B}} \\
	{\widetilde{B} / B } && {\widetilde{B} / B }
	\arrow["\sigma", from=1-1, to=1-3]
	\arrow["\kappa", from=1-1, to=2-1]
	\arrow["\kappa", from=1-3, to=2-3]
	\arrow["{\exists ! \text{ } \text{ind}_{ \widetilde{B} / B }(\sigma)}", dotted, from=2-1, to=2-3]
\end{tikzcd}\]
By $\kappa: \widetilde{B} \rightarrow \widetilde{B} / B$ we understand the standard quotient map, that is the one given by $\widetilde{b} \mapsto \widetilde{b} + B$. Now we put $G = \text{im}(\text{ind}_{ \widetilde{B} / B })$.

%

Now we will define a couple of functions which will turn out to be co-cycles. At first let us define $\mathbb{A} = \bigl\{ \sigma \in \text{Sym}(\mathbb{Z}_n) \mid \forall x \in \mathbb{Z}_n \text{ }  \sigma.x \in \{x,x+m\} \bigr\}$. Now let $\varphi: \mathbb{A} \rightarrow \mathbb{F}_2[\mathcal{L}]$ (cf. Observation \ref{obs: cosets of L and widetile{L} form block systems} and Definition \ref{defi: G-module A[X]}) be given by $(x,x+m) \mapsto \vec{e}_{L+x}$ and extended in a unique way which makes $\varphi$ a homomorphism of abelian groups. Notice that such $\varphi$ is an isomorphism. Notice that by 
Corollary \ref{cor: how alpha works on vertices in our case} $\text{im}(\alpha) \subseteq \mathbb{A}$. Let us now define $\widetilde{\alpha}: \text{Aut}^{\pi}(\Gamma) \rightarrow {\mathbb{F}_2[\mathcal{L}]}$ to be the unique function making the following diagram commute.
\[\begin{tikzcd}
	{\text{Aut}^{\pi}(\Gamma)} & {\mathbb{A}} \\
	{\text{Aut}^{\pi}(\Gamma)} & {\mathbb{F}_2[\mathcal{L}]}
	\arrow["{\alpha({[\text{ }\cdot\text{ }]}^{-1})}", from=1-1, to=1-2]
	\arrow["id"', from=1-1, to=2-1]
	\arrow["\varphi", from=1-2, to=2-2]
	\arrow["{\exists ! \text{ } \widetilde{\alpha}}", dotted, from=2-1, to=2-2]
\end{tikzcd}\]

Define also $\widetilde{\alpha}_P = \widetilde{\alpha} {|}_{P}$ and $\widetilde{\alpha}_Q = (\widetilde{\alpha}_P) {|}_{Q} = \widetilde{\alpha} {|}_{Q}$. Before we continue we need state a definition of $\pi_{\widetilde{B} / L}$ and prove the following lemma. Function $\pi_{\widetilde{B} / L}: \mathbb{F}_2[\mathcal{L}] \rightarrow \mathbb{F}_2[\widetilde{B} / L]$ is defined on a basis by formula
$$
\vec{e}_{L+x} \mapsto \begin{cases}
    \vec{e}_{L+x} & \text{ if } L+x \in \widetilde{B} / L \\
    \vec{0} & \text{ otherwise }
\end{cases}
$$
and extended linearly.

\begin{lem} \label{lem: text{res}_{widetilde{B}} {|}_{Q} is an iso}
    For each $\sigma \in P$, permutation $\widetilde{\sigma}$ defined by the construction from replacement property of $(\Gamma_{\text{reflective}},\widetilde{B})$ is an element of $Q$. Moreover function $\text{res}_{\widetilde{B}} {|}_{Q}: Q \rightarrow R$ is an isomorphism.
\end{lem}
\textit{Proof of the lemma.} Let us take arbitrary $\sigma \in P$. We will start by showing that the function $\widetilde{\sigma} \in \text{Aut}_{\widetilde{\mathcal{B}}}(\Gamma_{\text{reflective}})$ defined by the construction from replacement property of $(\Gamma_{\text{reflective}},\widetilde{B})$ is actually an element of of $\text{Aut}^{\pi}(\Gamma)$ and hence of its subgroup $Q$. Consider a function $\tau(\sigma): \mathbb{Z}_n \times \langle a \rangle \rightarrow \mathbb{Z}_n \times \langle a \rangle$ defined by the formula
$$
x \mapsto \begin{cases}
    f_{\widetilde{B}}(x + \widetilde{B}) + \sigma.\bigl( x - f_{\widetilde{B}}(x + \widetilde{B}) \bigr) & \text{if } x \in \mathbb{Z}_n \\
    a + f_{\widetilde{B}}(x + a + \widetilde{B}) + \bigl( \sigma\circ \alpha(\sigma) \bigr).\bigl( x + a - f_{\widetilde{B}}(x + a + \widetilde{B}) \bigr)  & \text{otherwise}
\end{cases}
$$
Notice that $\tau(\sigma) {|}_{\mathbb{Z}_n} = \widetilde{\sigma}$. One can also observe, that to copy the action of $\tau(\sigma) {|}_{\widetilde{B} \langle a \rangle}$ onto other cosets of $\widetilde{B} \langle a \rangle$. Since $\tau(\sigma)$ acts on $\widetilde{B} \langle a \rangle$ in the same way as the automorphism of $\Gamma \times K_2$ given by $(\sigma,\gamma(\sigma))$, each edge which have both ends in the same coset of $\widetilde{B} \langle a \rangle$ are mapped onto an edge. Since $C \leq B \leq \widetilde{B}$, any pair $x,y \in \mathbb{Z}_n$ such that $(x,y+a)$ is an edge of $\Gamma \times K_2$ which have ends in two different cosets of $\widetilde{B} \langle a \rangle$, $(x,y)$ is a reflective edge of $\Gamma$. Therefore any permutation from $\mathbb{A}$ maps such edges onto edges. Combining above with the fact that $\widetilde{\sigma} \in \text{Aut}(\Gamma_{\text{reflective}})$ we obtain that $\tau$ indeed is an automorphism of $\Gamma \times K_2$, hence $\widetilde{\sigma} \in \text{Aut}^{\pi}(\Gamma)$.

We are now ready to prove the lemma by firstly showing that $\text{res}_{\widetilde{B}} {|}_{Q}$ is an epimorphism, and then that it is a monomorphism.

Take any $\upsilon \in R$. By definition there exists $\sigma \in P$ such that $\sigma {|}_{\widetilde{B}} = \upsilon$. By the above reasoning, there exists $\widetilde{\sigma} \in Q$ such that $\widetilde{\sigma} {|}_{\widetilde{B}} =\sigma {|}_{\widetilde{B}} = \upsilon$, hence $\text{res}_{\widetilde{B}} {|}_{Q}$ is an epimorphism. 

Let us now consider $\sigma \in Q$ such that $\sigma {|}_{\widetilde{B}} = {id}_{\widetilde{B}}$. Then $\widetilde{\sigma} = {id}_{\mathbb{Z}_n}$ since we define function $\widetilde{\sigma}$ on other cosets of $\widetilde{B}$ as some conjugation of identity. Since $\sigma \in Q$, we obtain $\sigma = \widetilde{\sigma} = {id}_{\mathbb{Z}_n}$, hence $\text{res}_{\widetilde{B}} {|}_{Q}$ is a monomorphism. 

Combining these facts gives us the conclusion of the lemma.
\hfill \qedsymbol{}
\newline

By Lemma \ref{lem: text{res}_{widetilde{B}} {|}_{Q} is an iso} we easily see that there exists unique function $\widetilde{\alpha}_R: R \rightarrow \mathbb{F}_2[\widetilde{B} / L]$ which makes the following diagram commute.
\[\begin{tikzcd}
	Q & {\mathbb{F}_2[\mathcal{L}]} \\
	R & {\mathbb{F}_2[\widetilde{B} / L ]}
	\arrow["{\widetilde{\alpha}_Q}", from=1-1, to=1-2]
	\arrow["{\text{res}_{\widetilde{B}}}"', from=1-1, to=2-1]
	\arrow["{\pi_{\widetilde{B} / L }}", from=1-2, to=2-2]
	\arrow["{\exists ! \text{ } \widetilde{\alpha}_R}", dotted, from=2-1, to=2-2]
\end{tikzcd}\]

Let $B_r \leq \text{Sym}(\mathbb{Z}_n)$ be the group made of permutations which translate elements of $\mathbb{Z}_n$ by some chosen element of $B$. It will also be understood as the permutation group of the subgroup $\widetilde{B} \leq \mathbb{Z}_n$. Before we will be able to define the last and most important function, we have to show the above lemma.

\begin{lem}
    $\text{im}(\widetilde{\alpha}_R) \subseteq {\mathbb{F}_2[\widetilde{B} / L]}^{ B_r }$.
\end{lem}
\textit{Proof of the lemma.} Since $\mathcal{B}$ is by definition $\alpha$-homogeneous partition, for every $\sigma \in {\text{Aut}^{\pi}(\Gamma)}$ and any $x \in \mathbb{Z}_n$ we know that $\alpha(\sigma) {|}_{B+x}$ is constant, hence $\text{im}(\widetilde{\alpha}) \subseteq {\mathbb{F}_2[\mathcal{L}]}^{ B_r }$. If we now chase the diagram which defined $\widetilde{\alpha}_R$, we get conclusion of our lemma.
\hfill \qedsymbol{}
\newline

Now let us put $X = \widetilde{B} / B$. Define $\psi: \mathbb{F}_2[X] \rightarrow {\mathbb{F}_2[\widetilde{B} / L]}^{B_r}$ on the basis by 
$$
\vec{e}_{B+\widetilde{b}} \mapsto \sum_{b \in B \cap 2\mathbb{Z}_n} \vec{e}_{L+b+\widetilde{b}}
$$
and extend linearly. Then $\mathbb{F}_2[X] \overset{\psi}{\cong} {\mathbb{F}_2[\widetilde{B} / L]}^{B_r}$ is an isomorphism of abelian groups. Now we are ready to define the main object of this proof.

\begin{lem} \label{lem: existence of omega}
There exists unique function $\omega: G \rightarrow \mathbb{F}_2[X]$ which makes the following diagram commute.
\[\begin{tikzcd}
	R & {{\mathbb{F}_2[\widetilde{B} / L]}^{B_r}} \\
	G & {\mathbb{F}_2[X]}
	\arrow["{\widetilde{\alpha}_R}", from=1-1, to=1-2]
	\arrow["{\text{ind}_{\widetilde{B} / B}}"', from=1-1, to=2-1]
	\arrow["{\psi^{-1}}", from=1-2, to=2-2]
	\arrow["{\exists ! \text{ } \omega}", dotted, from=2-1, to=2-2]
\end{tikzcd}\]
\end{lem}
\textit{Proof of the lemma.} Take $\tau_1,\tau_2 \in R$ such that $\text{ind}_{\widetilde{B} / B}(\tau_1) = \text{ind}_{\widetilde{B} / B}(\tau_2)$. Let $\sigma_1,\sigma_2 \in Q$ be such that $\text{res}_{ \widetilde{B} }(\sigma_i) = \tau_i$ for $i \in \{ 1,2 \}$. By definition of $Q$ we know that $\widetilde{\sigma}_1 = \sigma_1$ and $\widetilde{\sigma}_2 = \sigma_2$, where $\widetilde{\sigma}_i$ are functions obtained from replacement property of the pair $(\Gamma_{\text{reflective}},\widetilde{B})$ with respect to function $f_{\widetilde{B}}$ described earlier. Therefore, if they give the same permutation of cosets of $B$ which are contained in $\widetilde{B}$, they permute all cosets of $B$ in the same way. This can be stated as $\sigma_1 \sigma_2^{-1} \in \text{Aut}_{\mathcal{B}}^{\pi}(\Gamma)$, hence by our assumption we obtain the information that $\alpha(\sigma_1 \sigma_2^{-1}) = {id}_{\mathbb{Z}_n}$. We can now calculate
$$
\alpha(\sigma_2^{-1}) = \alpha(\sigma_1^{-1} \sigma_1 \sigma_2^{-1}) = \sigma_1^{-1} \alpha(\sigma_1 \sigma_2^{-1}) \sigma_1 \circ \alpha(\sigma_1^{-1}) = \alpha(\sigma_1^{-1}),
$$
hence $\widetilde{\alpha}(\sigma_1) = \widetilde{\alpha}(\sigma_2)$. The last statement ensures us that indeed $\widetilde{\alpha}_{R}(\tau_1) =  \widetilde{\alpha}_{R}( \text{res}_{\widetilde{B}} (\sigma_1)) = \widetilde{\alpha}_Q(\sigma_1)  = \widetilde{\alpha}_Q(\sigma_2) = \widetilde{\alpha}_{R}( \text{res}_{\widetilde{B}} (\sigma_2)) = \widetilde{\alpha}_{R}(\tau_2)$ as wanted.
\hfill \qedsymbol{}
\newline

All of the above definitions and correlations between them can be summarized by the following diagram.
\[\begin{tikzcd}
	{\widetilde{\alpha}} & {\widetilde{\alpha}_P} & {\widetilde{\alpha}_Q} \\
	& {\widetilde{\alpha}_R} \\
	& \omega
	\arrow["{\text{ }\cdot\text{ }{|}_P}", maps to, from=1-1, to=1-2]
	\arrow["{\text{ }\cdot\text{ }{|}_Q}", maps to, from=1-2, to=1-3]
	\arrow["{\text{res}_{\widetilde{B}}}", maps to, from=1-3, to=2-2]
	\arrow["{\text{ind}_{\widetilde{B} /B}}"', maps to, from=2-2, to=3-2]
\end{tikzcd}\]

Now we will state the connection between function $\omega$ and conclusion of the theorem. Just to make it clear, whenever we are working with a module $\mathbb{F}_2[Y]$ for some set $Y$, by $\mathbbm{1}_{Y}$ we denote the vector $\sum_{y \in Y} \vec{e}_y$.

\begin{lem} \label{lem: 1 in image of omega implies conclusion of thm 6.1}
    If for some $g \in G$ we have $\omega(g) = \mathbbm{1}_{X}$, then $\text{Cay}(\mathbb{Z}_{2m},S) \cong \text{Cay}(\mathbb{Z}_{2m},S+m)$.
\end{lem}
\textit{Proof of the lemma.} By the definition of $\omega$, there exists $g \in G$ such that $\omega(g) = \mathbbm{1}_{X}$ if and only if there exists $\tau \in R$ such that $\widetilde{\alpha}_R(\tau) = \psi(\mathbbm{1}_{X}) = \mathbbm{1}_{\widetilde{B} / L}$. Now however, by definition of $\widetilde{\alpha}_R$ there exists $\sigma \in Q$ such that $\pi_{\widetilde{B} / L} \bigl( \widetilde{\alpha}(\sigma) \bigr) = \mathbbm{1}_{\widetilde{B} / L}$. By definition of $\widetilde{\alpha}$ above can be equivalently stated by saying that $\alpha(\sigma^{-1}) {|}_{\widetilde{B}}$ is given by the formula $\widetilde{b} \mapsto \widetilde{b} + m$ for any $\widetilde{b} \in \widetilde{B}$.

Now notice that by the fact that $\widetilde{\sigma^{-1}} = \sigma^{-1} $ we obtain 
$$
\bigl(\sigma^{-1} \bigr) {|}_{ \widetilde{\mathfrak{b}} } = {f(\widetilde{\mathfrak{b}})}_r \circ \bigl( \sigma^{-1} \bigr) {|}_{\widetilde{B}} \circ {f(\widetilde{\mathfrak{b}})}_r^{-1}
$$ 
for arbitrarily chosen $\widetilde{\mathfrak{b}} \in \widetilde{\mathcal{B}}$. Put $\widetilde{\mathfrak{b}} = \widetilde{B} + \widetilde{b}_0$. Now since $C \leq B \leq \widetilde{B}$, Lemma \ref{lem: alpha |C depends only on action on C} shows that
$$
\alpha \bigl( \sigma^{-1} \bigr) {|}_{\widetilde{B}} = \alpha \Bigl(  {f(\widetilde{\mathfrak{b}})}_r \circ \sigma^{-1} \circ {f(\widetilde{\mathfrak{b}})}_r^{-1} \Bigr) {|}_{\widetilde{B}}.
$$
Now we can calculate that 

$$
\alpha \bigl( \sigma^{-1} \bigr) = \alpha \Bigl( {f(\widetilde{\mathfrak{b}})}_r^{-1} {f(\widetilde{\mathfrak{b}})}_r \sigma^{-1} {f(\widetilde{\mathfrak{b}})}_r^{-1} {f(\widetilde{\mathfrak{b}})}_r \Bigr) = {f(\widetilde{\mathfrak{b}})}_r^{-1}
\alpha \Bigl(  {f(\widetilde{\mathfrak{b}})}_r \circ \sigma^{-1} \circ {f(\widetilde{\mathfrak{b}})}_r^{-1} \Bigr)
{f(\widetilde{\mathfrak{b}})}_r,
$$
hence
$$
\alpha \bigl( \sigma^{-1} \bigr) {|}_{\widetilde{B} - \widetilde{b}_0} = {f(\widetilde{\mathfrak{b}})}_r^{-1}
\circ 
\alpha \Bigl(  {f(\widetilde{\mathfrak{b}})}_r \circ \sigma^{-1} \circ {f(\widetilde{\mathfrak{b}})}_r^{-1} \Bigr) {|}_{\widetilde{B}} \circ
{f(\widetilde{\mathfrak{b}})}_r = {f(\widetilde{\mathfrak{b}})}_r^{-1} \circ
\alpha \bigl( \sigma^{-1} \bigr) {|}_{\widetilde{B}} \circ
{f(\widetilde{\mathfrak{b}})}_r.
$$
equality in the last line shows that for any $b_0 \in \widetilde{B} - \widetilde{b}_0$ it holds that $\alpha \bigl( \sigma^{-1} \bigr).b_0 = b_0 + m$. Since $\widetilde{\mathfrak{b}}$ was an arbitrary element of $\widetilde{B}$, it follows that for any $x \in \mathbb{Z}_n$ we obtain
$$
\alpha(\sigma^{-1}).x = x+m.
$$
Above equation for $\alpha \bigl(\sigma^{-1} \bigr)$ combined with Corollary \ref{Cor: (+m)_r in im(alpha) iff. Cay(Z_2m,S) iso Cay(Z_2m,S+m)} ends the proof of this lemma.
\hfill \qedsymbol{}
\newline

We will now prove that all of assumptions of Theorem \ref{Thm: action of a nontrivial cocycle from group G acting primitively on X with additional assumptions} holds, which will imply that the assumptions of Lemma \ref{lem: 1 in image of omega implies conclusion of thm 6.1} holds. Put $|X| = k$, take $\widetilde{b} \in \widetilde{B}$ such that $\langle \widetilde{b} \rangle = \widetilde{B}$ and consider a homomorphism $\chi: X \rightarrow \mathbb{Z}_k$ such that $\widetilde{b} + B \mapsto 1$.

\begin{lem} \label{lem: k odd, (Z_k)_r < G, i(mathfrak i ) in Aut(G)}
    $k$ is an odd integer. Moreover, if we identify $X$ with $\mathbb{Z}_k$ by $\chi$, then ${(\mathbb{Z}_k)}_r \leq G$, $\omega {|}_{{(\mathbb{Z}_k)}_r} \equiv \vec{0}$ and $\iota(\mathfrak{i}) \in \text{Aut}(G)$.
\end{lem}
\textit{Proof of the lemma.} Notice that since $L \leq C \leq B$, order of $B$ is even, hence $[\mathbb{Z}_n,B]$ is odd, hence $k = |\widetilde{B}/B| = [\widetilde{B}/B]$ divides $[\mathbb{Z}_n,B]$, it also is odd. Notice that ${\widetilde{b}}_r \in Q$, hence permutation ${\widetilde{b}}_r {|}_{\widetilde{B}} = \mathfrak{r}: \widetilde{B} \rightarrow \widetilde{B}$ is an element of the group $R$. Function $\text{ind}_{\widetilde{B} / B}(\mathfrak{r})$ is given by the formula $b + B \mapsto b+\widetilde{b} + B$, hence after we identify $X$ with $\mathbb{Z}_k$ via $\chi$, we conclude that permutation $1_r \in {(\mathbb{Z}_k)}_r$ is an element of $G$. Since ${(\mathbb{Z}_k)}_r = \langle 1_r \rangle$, we obtain ${(\mathbb{Z}_k)}_r \leq G$. Additionally, since $\widetilde{\alpha}(\widetilde{b}_r) \vec{0}$, we obtain $\widetilde{\alpha}_Q(\widetilde{b}_r) \vec{0}$, hence $\widetilde{\alpha}_R(\mathfrak{r}) \vec{0}$ and finally $\omega(1_r) = \vec{0}$, which can be checked by following definitions of $\widetilde{\alpha}_Q$, $\widetilde{\alpha}_R$ and $\omega$.

Let $\mathfrak{i}_n: \mathbb{Z}_n \rightarrow \mathbb{Z}_n$ be the function given by $x \mapsto -x$. Then since $\mathfrak{i}_n \in \text{Aut}^{\pi}(\Gamma)$ and $P = \text{Aut}_{\widetilde{B}}^{\pi}(\Gamma) \unlhd \text{Aut}^{\pi}(\Gamma)$ we obtain $\iota(\mathfrak{i}_n) \in \text{Aut}(P)$. This information implies that $\iota(\mathfrak{i}_n {|}_{\widetilde{B}})$ is an automorphism of the group $R$. To simplify the notation let us put $\widetilde{\mathfrak{i}} = \mathfrak{i}_n {|}_{\widetilde{B}}$. Note that permutation $\widetilde{\mathfrak{i}}$ permutes cosets on $B$ contained in $\widetilde{B}$ by $\widetilde{\mathfrak{i}}[B+b] = B-b$. Now we can easily see that if we define $\mathfrak{i}: X \rightarrow X$ by the formula $x \mapsto -x$, then 
$$
\iota(\mathfrak{i}) \in \text{Aut}\bigl( \text{ind}_{\widetilde{B} / B}(R) \bigr) = \text{Aut}(G)
$$
as wanted.
\hfill \qedsymbol{}
\newline

\begin{lem} \label{lem: G acts primitively on X}
    $G$ acts primitively on $X$.
\end{lem}
\textit{Proof of the lemma.} We will prove this lemma by contradiction. Assume that $\mathcal{D}$ is a nontrivial block system of action of $G$ on $X$. Then $\mathcal{E} = \kappa^{-1}[\mathcal{D}] = \{ \kappa^{-1}[\mathfrak{d}] \mid \mathfrak{d} \in \mathcal{D} \}$ is a block system of the group action of $R$ on $\widetilde{B}$. Since ${\widetilde{B}}_r \leq R$, we conclude that $\mathcal{E}$ is made of cosets of some subgroup $E \leq \widetilde{B}$. Since $R = \text{res}_{\widetilde{B}}(P)$, partition $\widetilde{\mathcal{E}} = \{ E+x \mid x \in \mathbb{Z}_n \}$ is an invariant partition for the action of $P = \text{Aut}_{\widetilde{B}}^{\pi}(\Gamma)$ on $\mathbb{Z}_n$ and is such that $\mathcal{B} \prec \widetilde{\mathcal{E}} \prec \widetilde{\mathcal{B}}$. Moreover notice that $B<E<\widetilde{B}$. We will now prove that $\widetilde{\mathcal{E}}$ is a block system with respect to action of $\text{Aut}^{\pi}(\Gamma)$ on $\mathbb{Z}_n$.

Take any $\sigma \in \text{Aut}^{\pi}(\Gamma)$. Notice that since each element of $P$ permutes the elements of the partition $\widetilde{\mathcal{E}}$, elements of $P = \sigma P \sigma^{-1}$ permute elemenets of the partition $\sigma[\widetilde{\mathcal{E}}] = \{ \sigma[{\mathfrak{e}}] \mid {\mathfrak{e}} \in \widetilde{\mathcal{E}} \}$.

Note however, that since $\widetilde{\mathcal{E}} \prec \widetilde{B}$ and $\widetilde{\mathcal{B}}$ is a block system, $\sigma[\widetilde{\mathcal{E}}] \prec \widetilde{B}$. If we additionally observe that $|\sigma[{\mathfrak{e}}]| = |E|$ for all ${\mathfrak{e}} \in \widetilde{\mathcal{E}}$ and ${( \widetilde{B} )}_r \leq P$ we conclude that each element of $\sigma[\widetilde{\mathfrak{e}}]$ is a coset of some subgroup of order $E$, however since $\mathbb{Z}_n$ is cyclic, $E$ is the only such subgroup, hence $\sigma[\widetilde{\mathcal{E}}] = \widetilde{\mathcal{E}}$ so $\sigma$ permutes elements of partition $\widetilde{\mathcal{E}}$. Because $\sigma$ was an arbitrary element of $\text{Aut}^{\pi}(\Gamma)$, we conclude that $\widetilde{\mathcal{E}}$ is a block system as wanted.

To end the proof it is enough to observe that existence of $\widetilde{\mathcal{E}}$ contradicts the fact that $\widetilde{\mathcal{B}}$ is the minimal block system thicker than $\mathcal{B}$, hence we obtained the desired contradiction.
\hfill \qedsymbol{}
\newline

\begin{lem} \label{lem: omega is a nonzero cocycle}
    Function $\omega: G \rightarrow \mathbb{F}_2[X]$ is a nonzero co-cycle.
\end{lem}
\textit{Proof of the lemma.}
    At first notice, that by Observation \ref{Obs: basic properties of gamma} and the fact that $\Gamma$ is unstable, function $\alpha: \text{Aut}^{\pi}(\Gamma) \rightarrow \mathbb{A}$ does not constantly equal identity. Now we will show that $\widetilde{\alpha}$ is a co-cycle. If in Observation \ref{Obs: basic properties of alpha} we substitute $\sigma^{-1}$ and $\tau^{-1}$ for $\sigma$ and $\tau$ respectively, we obtain
    $$
    \alpha(\sigma^{-1} \tau^{-1}) = \tau \alpha(\sigma^{-1}) \tau^{-1} \circ \alpha(\tau^{-1}),
    $$
    which translates to $\widetilde{\alpha}(\tau \sigma) = \tau.\widetilde{\alpha}(\sigma) + \widetilde{\alpha}(\tau)$ as wanted. We will now show that each of $\widetilde{\alpha}_P$, $\widetilde{\alpha}_Q$, $\widetilde{\alpha}_R$ and $\omega$ is a cocycle.

    Since $\widetilde{\alpha}_P$ and $\widetilde{\alpha}_Q$ are just restrictions of $\widetilde{\alpha}$, they obviously are co-cycles. Take now $\tau_1 \tau_2 \in R$ and $\sigma_1,\sigma_2 \in Q$ such that $\text{res}_{\widetilde{B}}(\sigma_i) = \tau_i$ for $i \in \{1,2\}$. Then we get
    $$
    \widetilde{\alpha}_R(\tau_1 \tau_2) = \pi_{\widetilde{B} / L } \bigl( \widetilde{\alpha}_Q(\sigma_1 \sigma_2) \bigr) = \pi_{\widetilde{B} / L } \bigl( \sigma_1.\widetilde{\alpha}_Q(\sigma_2) + \widetilde{\alpha}_Q(\sigma_1) \bigr) = 
    $$
    $$
    = \pi_{\widetilde{B} / L } \bigl(  \sigma_1.\widetilde{\alpha}_Q(\sigma_2) \bigr) +  \pi_{\widetilde{B} / L } \bigl(  \widetilde{\alpha}_Q(\sigma_1) \bigr) = \tau_1.\Bigl(  \pi_{\widetilde{B} / L } \bigl( \widetilde{\alpha}_Q(\sigma_2) \bigr) \Bigl) + \widetilde{\alpha}_R(\tau_1) = \tau_1.  \widetilde{\alpha}_R(\tau_2)  + \widetilde{\alpha}_R(\tau_1)
    $$
    as wanted. Now We will proceed similarly to show that $\omega$ also is a co-cycle. Take $g_1,g_2 \in G$ and let $\tau_1, \tau_2 \in R$ be such that $\text{ind}_{\widetilde{B} / B}(\tau_i) = g_i$ for $i \in \{1,2\}$. Then we get
    $$
    \omega(g_1 g_2) = \psi^{-1} \bigl( \widetilde{\alpha}_R(\tau_1 \tau_2) \bigr) = \psi^{-1} \bigl( \sigma_1.\widetilde{\alpha}_R(\tau_2) + \widetilde{\alpha}_R(\tau_1) \bigr) = 
    $$
    $$
    = \psi^{-1} \bigl(  \tau_1.\widetilde{\alpha}_R(\tau_2) \bigr) +  \psi^{-1} \bigl(  \widetilde{\alpha}_R(\tau_1) \bigr) = g_1.\Bigl(  \psi^{-1} \bigl( \widetilde{\alpha}_R(\tau_2) \bigr) \Bigl) + \omega(g_1) = g_1.  \omega(g_2) + \omega(g_1)
    $$
    as wanted, hence $\omega$ indeed is a co-cycle.
    
    To end the proof of the lemma we need to show that $\omega$ is not a zero function. Partition $\mathcal{B}$ is by definition the thickest $\alpha$-homogeneous partition, hence $\widetilde{\mathcal{B}}$ is not $\alpha$-homogeneous. Let $\sigma^{-1} \in \text{Aut}^{\pi}(\Gamma)$ be such that $\alpha(\sigma^{-1})$ is the counterexample to $\alpha$-homogeneity of $\widetilde{\mathcal{B}}$. This can be equivalently stated as $\widetilde{b}.\widetilde{\alpha}(\sigma) \neq \widetilde{\alpha}(\sigma)$, where $\widetilde{b} \in \widetilde{B}$ is such that $\langle \widetilde{b} \rangle = \widetilde{B}$. Put $\sigma' = \sigma \circ {\widetilde{b}}_r \circ \sigma^{-1}$ and note that since ${\widetilde{b}}_r \in P$ and $P \unlhd \text{Aut}^{\pi}(\Gamma)$, we obtain $\sigma' \in P$. Observe that $\vec{0} = \widetilde{\alpha}(\sigma \circ \sigma^{-1}) = \sigma.\widetilde{\alpha}(\sigma^{-1}) + \widetilde{\alpha}(\sigma)$, hence we see that
$$
\widetilde{\alpha}(\sigma') = \sigma.\widetilde{\alpha}({\widetilde{b}}_r \circ \sigma^{-1}) + \widetilde{\alpha}(\sigma) = 
\sigma. \bigl( {\widetilde{b}}_r.\widetilde{\alpha}(\sigma^{-1}) \bigr) + \sigma.\widetilde{\alpha}({\widetilde{b}}_r) + \widetilde{\alpha}(\sigma) \neq
$$
$$
\neq \sigma. \bigl( {\widetilde{b}}_r.\widetilde{\alpha}(\sigma^{-1}) \bigr) + \sigma.\widetilde{\alpha}({\widetilde{b}}_r) + \sigma.\widetilde{\alpha}(\sigma^{-1}) = \sigma. \bigl( 
{\widetilde{b}}_r.\widetilde{\alpha}(\sigma^{-1}) + \widetilde{\alpha}(\sigma^{-1})
\bigr) = \vec{0},
$$
so indeed $\widetilde{\alpha}_P$ is not a zero function. Now we will show that $\widetilde{\alpha}_R$ also is not a zero function. 

Take $\sigma \in P$ such that $\alpha(\sigma^{-1})$ is not identity. We can additionally assume that $\alpha(\sigma^{-1}) {|}_{\widetilde{B}}$ is nonzero. If not, one only needs to conjugate $\sigma$ by an appropriate rotation. Then function $\widetilde{\sigma^{-1}}$ created from $\sigma^{-1}$ by construction from replacement property of the pair $(\Gamma_{\text{reflective}},\widetilde{B})$ is an element of $Q$ (cf. Lemma \ref{lem: text{res}_{widetilde{B}} {|}_{Q} is an iso}). Since $\sigma^{-1} {|}_{\widetilde{B}} = \widetilde{\sigma^{-1}} {|}_{\widetilde{B}}$ and $C \leq \widetilde{B}$, by Lemma \ref{lem: alpha |C depends only on action on C} we obtain
$$
\alpha(\sigma^{-1}) {|}_{\widetilde{B}} = \alpha(\widetilde{\sigma^{-1}}) {|}_{\widetilde{B}},
$$
hence $\pi_{\widetilde{B} / L} \bigl( \widetilde{\alpha}_Q(\widetilde{\sigma}) \bigr) \neq \vec{0}$. This last equality ensures that both $\widetilde{\alpha}_Q$ and $\widetilde{\alpha}_R$ are nonzero functions. 

Since $\mathbb{F}_2[X] \overset{\psi}{\cong} {\mathbb{F}_2[\widetilde{B} / L]}^{B_r}$ and $\text{ind}_{\widetilde{B} / B}: R \rightarrow G$ is subjective by definition, Lemma \ref{lem: existence of omega} ensures us that $\omega$ also is a nonzero function.
\hfill \qedsymbol{}
\newline

Putting conclusions of Lemma \ref{lem: k odd, (Z_k)_r < G, i(mathfrak i ) in Aut(G)}, Lemma \ref{lem: G acts primitively on X} and Lemma \ref{lem: omega is a nonzero cocycle} we fulfill all assumptions of Theorem \ref{Thm: action of a nontrivial cocycle from group G acting primitively on X with additional assumptions}. Conclusion of this theorem in particular implies the assumptions of Lemma \ref{lem: 1 in image of omega implies conclusion of thm 6.1}, hence the proof of Theorem \ref{thm: handeling the hard case for 2m with assumptions about replacement property} is complete.
\end{proof}

We are now ready to prove the main theorem of this paper. \\

\textit{Proof of Theorem \ref{THM: characterization of unstable circulants of squarefree order}:} \\
The fact that criteria i. and ii. imply instability is widely known (cf. \cite[Theorem 1.4]{HujdurovicMitrovicMorrisOverview}), hence we will focus on proving the contrary. We will divide the proof into two cases. We will start by taking care of non-reduced circulants and then we will deal with reduced ones.\\

\textit{Case 1:} We assume that $\Gamma = \text{Cay}(\mathbb{Z}_n,S)$ is non-reduced, hence there exists nonzero $h \in \mathbb{Z}_n$ such that $S + h = S$. If $h= \frac{n}{2}$, then condition ii. is satisfied for $l=1$. Otherwise $2h \neq 0$ and we get $S + 2h = S$, hence in particular $S \cap 2\mathbb{Z}_n + 2h = S \cap 2\mathbb{Z}_n$ which means that condition i. is satisfied. \\

\textit{Case 2:} We now assume that $\Gamma = \text{Cay}(\mathbb{Z}_n,S)$ is reduced. Since $n$ is square-free, $4$ does not divide $n$, hence $m = \frac{n}{2}$ is an odd integer. By Lemma \ref{Lem: reasons for instability in case 2m, m odd -- messy second condition} we obtain that either
    \begin{enumerate}[i.]
        \item there exists nonzero $h \in 2\mathbb{Z}_{2m}$ such that $S \cap 2\mathbb{Z}_{2m} + h = S \cap 2\mathbb{Z}_{2m}$;
        \item or $\{a,m+a\}$ is a basic set of $\mathcal{A}(\Gamma \times K_2)$.
    \end{enumerate}
First of the above conditions is also a condition i. for the statement of our theorem. Since that case is dealt with, from now on we assume that $\{a,m+a\}$ is a basic set of $\mathcal{A}(\Gamma \times K_2)$, hence $\Gamma$ fulfills Hypothesis \ref{HYPOTHESIS which states the hard case}. By Theorem \ref{Thm: if n is squarefree, all pairs (Gamma,H) satisfy replacement property} pair $(\Gamma_{\text{reflective}},H)$ satisfies replacement property for any subgroup $H \leq \mathbb{Z}_n$, hence we can apply Theorem \ref{thm: handeling the hard case for 2m with assumptions about replacement property} to obtain
$$
\text{Cay}(\mathbb{Z}_{n},S) \cong \text{Cay}(\mathbb{Z}_{n},S+\frac{n}{2}).
$$
From \cite[Theorem 1.1]{MuzychukSolToIsomOfCircul} we deduce that this isomorphism can be given by one of the functions from the set
$$
\{ \cdot l \mid l \in \mathbb{Z} \text{ such that } l \text{ is coprime to } n \},
$$
where $\cdot l: \mathbb{Z}_n \rightarrow \mathbb{Z}_n$ is given by the formula $x \mapsto l \cdot x$. Function $\cdot l$ transforms $\text{Cay}(\mathbb{Z}_{n},S)$ into $\text{Cay}(\mathbb{Z}_{n},l\cdot S)$, hence for some $l$ coprime to $n$ we need to have $l \cdot S = S + \frac{n}{2}$ and therefore condition ii. is satisfied.
\hfill \qedsymbol{}
\newline

\section*{Acknowledgments}
The author is very grateful for help and guidance of Jakub Byszewski, whose experience in the theory of cohomology of group modules significantly contributed to the proof of Theorem \ref{Thm: action of a nontrivial cocycle from group G acting primitively on X with additional assumptions}.

\addtolength{\leftmargin}{0cm}
\setlength{\itemindent}{0cm}
\bibliographystyle{plain}
\bibliography{main.bib}

\begin{thebibliography}{10}

\bibitem{WebsiteM11}
https://brauer.maths.qmul.ac.uk.
\newblock Accesed on September 30, 2024.

\bibitem{WebsiteM23}
https://en.wikipedia.org.
\newblock Accesed on September 30, 2024.

\bibitem{WebsitePSL211}
https://groupprops.subwiki.org.
\newblock Accesed on September 30, 2024.

\bibitem{DOBSON200579}
Edward Dobson and Joy Morris.
\newblock On automorphism groups of circulant digraphs of square-free order.
\newblock {\em Discrete Mathematics}, 299(1):79--98, 2005.
\newblock Graph Theory of Brian Alspach.

\bibitem{FERNANDEZ202249}
Blas Fernandez and Ademir Hujdurović.
\newblock Canonical double covers of circulants.
\newblock {\em Journal of Combinatorial Theory, Series B}, 154:49--59, 2022.

\bibitem{HandOfProdGraphs}
Richard Hammack, Wilfried Imrich, and Sandi Klavzar.
\newblock {\em Handbook of Product Graphs, Second Edition}.
\newblock CRC Press, Inc., USA, 2nd edition, 2011.

\bibitem{StabilityAndSchurRings}
Ademir Hujdurović and István Kovács.
\newblock Stability of cayley graphs and schur rings.
\newblock {\em arXiv:2308.01182}, 08 2023.

\bibitem{HujdurovicMitrovicMorrisOverview}
Ademir Hujdurović, Đorđe Mitrović, and Dave Morris.
\newblock On automorphisms of the double cover of a circulant graph.
\newblock {\em The Electronic Journal of Combinatorics}, 28, 12 2021.

\bibitem{FinGroupsIII}
Bertram Huppert and Norman Blackburn.
\newblock {\em Finite Groups III}.
\newblock Grundlehren der mathematischen Wissenschaften. Springer Berlin, Heidelberg, 1982.

\bibitem{FinPrimPermGroups}
Cai Li.
\newblock The finite primitive permutation groups containing an abelian regular subgroup.
\newblock {\em Proceedings of the London Mathematical Society}, 87:725 -- 747, 11 2003.

\bibitem{MorrisOddAbelianGroups}
Dave~Witte Morris.
\newblock On automorphisms of direct products of cayley graphs on abelian groups.
\newblock {\em The Electronic Journal of Combinatorics}, 28, 07 2021.

\bibitem{MuzychukSolToIsomOfCircul}
Mikhail Muzychuk.
\newblock A solution of the isomorphism problem for circulant graphs.
\newblock {\em Proceedings of the London Mathematical Society}, 88:1 -- 41, 01 2004.

\bibitem{Schur1933}
Issai Schur.
\newblock Zur theorie der einfach transitiven permutationgruppen.
\newblock {\em S.-B. Preuss. Akad. Wiss. Phys.- Math. Kl.}, 18:598–623, 1933.

\bibitem{LocalFields}
Jean-Pierre Serre.
\newblock {\em Local Fields}.
\newblock Graduate Texts in Mathematics. Springer New York, NY, 1979.

\bibitem{wielandt2014finite}
H.~Wielandt, H.~Booker, D.A. Bromley, and N.~DeClaris.
\newblock {\em Finite Permutation Groups}.
\newblock Academic Paperbacks. Mathematics. Elsevier Science, 2014.

\bibitem{WILSON2008359}
Steve Wilson.
\newblock Unexpected symmetries in unstable graphs.
\newblock {\em Journal of Combinatorial Theory, Series B}, 98(2):359--383, 2008.

\end{thebibliography}

\end{document}